\documentclass[12pt]{article}
\usepackage{color}
\usepackage{amsmath,amssymb, amsthm}
\usepackage{fullpage, graphicx}
\usepackage{natbib}

\usepackage{mathptmx}
\usepackage{url}

\usepackage{hyperref}

\usepackage{authblk}

\newtheorem{theorem}{Theorem}
\newtheorem{lemma}[theorem]{Lemma}

\newtheorem{example}[theorem]{Example}
\newtheorem{corollary}[theorem]{Corollary}

\newtheorem{proposition}[theorem]{Proposition}
\newtheorem{remark}[theorem]{Remark}
\newtheorem{definition}{Definition}

\newtheorem{assumption}{Assumption}

\title{Uniform Convergence of the Kernel Density Estimator \\
	Adaptive to Intrinsic Volume Dimension}

\date{}

\author[1]{Jisu Kim}
\author[2]{Jaehyeok Shin}
\author[2]{Alessandro Rinaldo}
\author[2]{Larry Wasserman}
\affil[1]{Inria Saclay -- \^{I}le-de-France}
\affil[2]{Department of Statistics and Data Science, Carnegie Mellon University}

\begin{document}
	
\maketitle

\begin{abstract}
We derive concentration inequalities for the supremum norm of the difference between a kernel density estimator (KDE) and its point-wise expectation that hold uniformly over the selection of the bandwidth and under weaker conditions on the kernel and the data generating distribution than previously used in the literature. We first propose a novel concept, called the volume dimension, to measure the intrinsic dimension of the support of a probability distribution based on the rates of decay of the probability of vanishing Euclidean balls. Our bounds depend on the volume dimension and generalize the existing bounds derived in the literature. In particular, when the data-generating distribution has a bounded Lebesgue density or is supported on a sufficiently well-behaved lower-dimensional manifold, our bound recovers the same convergence rate depending on the intrinsic dimension of the support as ones known in the literature. At the same time, our results apply to more general cases, such as the ones of distribution with unbounded densities or supported on a mixture of manifolds with different dimensions. Analogous bounds are derived for the derivative of the KDE, of any order. Our results are generally applicable but are especially useful for problems in geometric inference and topological data analysis, including level set estimation, density-based clustering, modal clustering and mode hunting, ridge estimation and persistent homology.
\end{abstract}

\section{Introduction}
\label{sec:intro} 

Density estimation \citep[see, e.g.][]{rao1983nonparametric} is a classic and fundamental problem in non-parametric statistics that, especially in recent years, has also become a key step in many geometric inferential tasks. Among the numerous existing methods for density estimation, kernel density estimators (KDEs) are especially popular because of their conceptual simplicity and nice theoretical properties. A KDE is simply the Lebesgue density  of the probability distribution obtained by  convolving the empirical measure induced by the sample with an appropriate function, called kernel, \citep{parzen1962estimation, wand1994kernel}.
Formally, let $X_{1},\ldots,X_{n}$ be an independent and identically distributed sample from an unknown Borel probability distribution $P$ in $\mathbb{R}^d$. For a given kernel $K$, where $K$ is an appropriate function on $\mathbb{R}^d$ (often a density), and bandwidth $h > 0$, the corresponding KDE is the random Lebesgue density function defined as 
\begin{equation}
\label{eq:kde_intro}
x \in \mathbb{R}^d \mapsto \hat{p}_{h}(x):=\frac{1}{nh^{d}}\sum_{i=1}^{n}K\left(\frac{x-X_{i}}{h}\right).
\end{equation}
The point-wise expectation of the KDE is the function 
\[
x \in \mathbb{R}^d \mapsto p_h(x) := \mathbb{E}[\hat{p}_h(x)],
\]
and can be regarded as a smoothed version of the density of $P$, if such a density exists.
In fact, interestingly, both $\hat{p}_{h}$ and $p_h$ are Lebesgue probability densities for any choice of $h>0$, regardless of whether $P$ admits a Lebesgue density. What is more,  $p_h$ is often times able to capture important topological properties of the underlying distribution $P$ or of its support \citep[see, e.g.][Section 4.4]{FasyLRWBS2014}. For instance, if a data-generating distribution consists of two point masses, it has no Lebesgue density but the pointwise mean of the KDE with Gaussian kernel is a density of mixtures of two Gaussian distributions whose mean parameters are the two point masses. Although $P$ is quite different from the distribution corresponding to $p_h$, for practical purposes, one may in fact rely on $p_h$.

Though seemingly contrived, the previous example illustrates a general phenomenon encountered in many geometrical inference problems, namely that using $p_h$ as a target for inference leads to not only well-defined statistical tasks but also to faster or even dimension independent rates. Results of this form, which require a uniform control over $\| \hat{p}_h - p_h \|_\infty := \sup_{x \in \mathbb{R}^d} \| \hat{p}_h(x) - p_h(x)\|$ are plentiful in the literature on density-based clustering \citep{RinaldoW2010, wang2017optimal}, modal clustering and mode hunting \citep{chacon2015population, azizyan2015risk},  mean-shift clustering \citep{arias2016estimation}, ridge estimation \citep{chen2015optimal, chen2015asymptotic} and inference for density level sets \citep{chen2017density}, cluster density trees \citep{balakrishnan2013cluster, KimCBRW2016} and persistent diagrams \citep{FasyLRWBS2014, ChazalFLMRW2014r}.



Asymptotic and finite-sample bounds on $\|\hat{p}_h - p_h\|_\infty$ under the existence of Lebesgue density have been well-studied for fixed bandwidth cases \citep{rao1983nonparametric, gine2002rates, Sriperumbudur2012, SteinwartST2017}.

Bounds for KDEs not only uniform in $x \in \mathbb{R}^d$ but also with respect the choice of the bandwidth $h$ have received relatively less attentions, although such bounds are important to analyze the consistency of  KDEs with adaptive bandwidth, which may depend on the location $x$. \citet{einmahl2005uniform} showed that,
\[
\limsup_{n \rightarrow \infty}\sup_{(c\log n) / n \leq h \leq 1} \frac{\sqrt{nh^d}\|\hat{p}_h - p_h\|_\infty}{\sqrt{\log (1/h) \vee \log \log n }} < \infty,
\]
for regular kernels and bounded Lebesgue densities. \citet{jiang2017uniform} provided a finite-sample bound on $\|\hat{p}_h - p_h\|_\infty$ that holds uniformly on $h$ and under appropriate assumptions on $K$, and extended it to case of densities over well-behaved manifolds.

The main goal of this paper is to extend existing uniform bounds on KDEs by weakening the conditions on the kernel and making it adaptive to the intrinsic dimension of the underlying distribution. We first propose a novel concept, called the {\it volume dimension,} to characterize the intrinsic dimension of the underlying distribution. In detail, 
the volume dimension $d_{\mathrm{vol}}$ is the rate of decay of the probability of vanishing Euclidean balls, i.e. fix a subset $\mathbb{X}\subset\mathbb{R}^{d}$, then
\[
d_{{\rm vol}}=\sup\left\{ \nu\in\mathbb{R}:\,\limsup_{r\to 0}\sup_{x\in\mathbb{X}}\frac{P(\mathbb{B}_{\mathbb{R}^{d}}(x,r))}{r^{\nu}}<\infty\right\}.
\]
We show that, if $K$ satisfies mild regularity conditions, with probability at least $1-\delta$,
\begin{equation}
\sup\limits_{h \geq l_{n}, x\in\mathbb{X}}\left|\hat{p}_{h}(x)-p_{h}(x)\right|
\leq C\sqrt{\frac{\left(\log\left(1/l_{n}\right)\right)_{+}+\log\left(2/\delta\right)}{n l_{n}^{2d-d_{\mathrm{vol}}+\epsilon}}},
\label{eq:bound_upper_intro}
\end{equation}
for any $\epsilon\in(0,d_{{\rm vol}})$, $\{ l_n \}$ a positive sequence approaching $0$ and $C$ is a constant that does not depend on $n$ nor $l_n$. Under additional, weak regularity conditions on $P$, the quantity $\epsilon$ can be taken to be $0$ in \eqref{eq:bound_upper_intro}. If the distribution has a bounded Lebesgue density, $d_{\mathrm{vol}} = d$ so our result recovers existing results in literature in terms of rates of convergence. For a bounded density on a $d_M$-dimensional manifold we obtain, under appropriate conditions, that $d_{\mathrm{vol}} = d_M$. Thus, if KDEs are defined with a correct normalizing factor $h^{d_M}$ instead of $h^d$, our rate also recovers the ones in the literature on density estimation over manifolds. At the same time, our bounds apply to more general cases, such as a distribution with an unbounded density or supported on a mixture of manifolds with different dimensions. We have also shown the optimality of \eqref{eq:bound_upper_intro} up to log terms by showing that
under the mild regularity conditions on $K$ and $P$, 
\begin{equation}
\sup\limits_{h\geq l_{n}, x\in\mathbb{X}}\left|\hat{p}_{h}(x)-p_{h}(x)\right|\geq C'\sqrt{\frac{1}{n l_{n}^{2d-d_{\mathrm{vol}}}}}.
\label{eq:bound_lower_intro}
\end{equation}

We make the following contributions: 
\begin{enumerate}
    \item We propose a novel concept, called the volume dimension, to characterize the convergence rate of the KDE on arbitrary distributions.
	\item We derive high probability finite sample bounds for $\| \hat{p} - p_h \|_\infty$, uniformly over the choice of $h \geq l_n$, for a given $l_n$ depending on $n$.
	\item We derive rates of consistency in the $\text{\L}_\infty$ norm  that are adaptive to the volume dimension of the distribution under  conditions on the kernel that, to the best of our knowledge, are weaker than the ones existing in the literature, and without assumptions on the distribution. Hence, our bounds recover known previous results, and apply to more general cases such as a distribution with unbounded density or supported on a mixture of manifolds with different dimensions.
	\item We show that our bound is optimal up to log terms under weak conditions on the kernel and the distribution.
	\item We also obtain analogous bounds for all higher order derivatives of $\hat{p}_h$ and $p_h$. 
\end{enumerate}

The closest results to the ones we present are by \citet{jiang2017uniform}, who relies on relative VC bounds to derive finite sample bounds on  $\| \hat{p}_h - p_h \|_\infty$ for a special class of kernels and assuming $P$ to have a well-behaved support. Our analysis relies instead on more sophisticated techniques rooted in the theory of empirical process theory as outlined in \cite{Sriperumbudur2012} and are applicable to a broader class of kernels. In addition, we do not assume any condition on the underlying distribution.

\section{Notation}
\label{sec:notation}
Below, we recap basic concepts and establish some notation that are used throughout the paper. For more detailed definitions, see Appendix~\ref{sec:definition}.

We let $\|\cdot\|$ be the  Euclidean $2$-norm. For $x\in\mathbb{R}^{d}$ and $r>0$, we use the notation $\mathbb{B}_{\mathbb{R}^{d}}(x,r)$ for the open Euclidean ball centered at $x$ and radius $r$, i.e. $\mathbb{B}_{\mathbb{R}^{d}}(x,r)=\{y\in\mathbb{R}^{d}:\,\left\Vert y-x\right\Vert <r\}$. We fix a subset $\mathbb{X}\subset\mathbb{R}^{d}$ on which we are considering the uniform convergence of the KDE.

The Hausdorff measure is a generalization of the Lebesgue measure to lower dimensional subsets of $\mathbb{R}^{d}$. The Hausdorff dimension is a generalization of the intrinsic dimension of a manifold to general sets. For $\nu\in \{1, \dots, d\}$, let $\lambda_{\nu}$ be a normalized
$\nu$-dimensional Hausdorff measure on $\mathbb{R}^{d}$ satisfying that its measure on any $\nu$-dimensional unit cube is $1$. We use the notation $\omega_{\nu}:=\lambda_{\nu}(\mathbb{B}_{\mathbb{R}^{\nu}}(0,1))=\frac{\pi^{\frac{\nu}{2}}}{\Gamma\left(\frac{\nu}{2}+1\right)}$ for the volume of the unit ball in $\mathbb{R}^{\nu}$ for $\nu = 1, \dots, d$.

First introduced by \citep{Federer1959}, the reach has been the minimal regularity assumption in the geometric measure theory. A manifold with positive reach means that the projection to the manifold is well defined in a small neighborhood of the manifold. 

\section{Volume Dimension}
\label{sec:dimension}
We first characterize the intrinsic dimension of a probability distribution in terms
of the rate of decay of the probability of  Euclidean balls of vanishing volumes. When a probability distribution $P$ has a
bounded density $p$ with respect to a well-behaved manifold $M$ of dimension
$d_{M}$, it is known that, for any point $x\in M$, the
measure on the ball $\mathbb{B}_{\mathbb{R}^{d}}(x,r)$ centered at
$x$ and radius $r$ decays as
\[
P\left(\mathbb{B}_{\mathbb{R}^{d}}(x,r)\right)\sim r^{d_{M}},
\]
when $r$ is small enough. From this, we define the volume dimension
to be the maximum possible exponent rate that can dominate the probability
volume decay on balls.

\begin{definition}[Volume Dimension]

\label{def:dimension_volume}

Let $P$ be a probability distribution on $\mathbb{R}^{d}$. The \emph{volume
dimension} of $P$ is a non-negative real number defined as 
\begin{equation}
d_{{\rm vol}}(P)
:=\sup\left\{ \nu \geq 0:\,\limsup_{r\to 0}\sup_{x\in\mathbb{X}}\frac{P(\mathbb{B}_{\mathbb{R}^{d}}(x,r))}{r^{\nu}}<\infty\right\} .\label{eq:kde_condition_prob_upper}
\end{equation} 
We will use the notation $d_{{\rm vol}}$ when $P$ is clearly specified by the context.

\end{definition}

The volume dimension has a connection with the Hausdorff dimension. If a probability distribution has a positive measure on a set, then the volume dimension is between $0$ and the Hausdorff dimension of the set. So, if that set is a manifold, then the volume dimension is always between $0$ and the dimension of the manifold.
In particular, the volume dimension of any probability distribution is between
$0$ and the ambient dimension $d$.

\begin{proposition}

\label{prop:dimension_manifold}

Let $P$ be a probability distribution on $\mathbb{R}^{d}$, and $d_{{\rm vol}}$
be its volume dimension.
Suppose there exists a set $A$ satisfying $P(A\cap\mathbb{X})>0$ and with Hausdorff dimension $d_{H}$. Then $0 \leq d_{{\rm vol}}\leq d_{H}$. Hence if $A$ is a $d_{M}$-dimensional manifold, then $0 \leq d_{{\rm vol}}\leq d_{M}$.
In particular, for any probability distribution $P$ on $\mathbb{R}^{d}$,  $0 \leq d_{{\rm vol}}\leq d$.
Also, if $P$ has a point mass, i.e. there exists $x\in\mathbb{X}$
with $P(\{x\})>0$, then $d_{{\rm vol}} = 0$.
\end{proposition}

The volume dimension is well defined with mixtures of distributions. Specifically, the volume dimension of the mixture is the minimum of the volume dimensions of the component distributions.  
\begin{proposition}
    \label{prop:mixture}
    Let $P_{1},\ldots,P_{m}$ be probability distributions on $\mathbb{R}^{d}$,  and  $\lambda_{1},\ldots,\lambda_{m}\in(0,1)$ with $\sum_{i=1}^{m}\lambda_{i}=1$. Then 
    \[
    d_{{\rm vol}}\left(\sum_{i=1}^{m}\lambda_{i}P_{i}\right)=\min\left\{ d_{{\rm vol}}(P_{i}):\,1\leq i\leq m\right\} .
    \]
    In particular, when $d_{{\rm vol}}$ is understood as a real-valued function on the space of probability distributions, both its sublevel sets and superlevel sets are convex.
\end{proposition}

The name ``volume dimension'' suggests that the volume dimension of a probability distribution has a connection with the dimension of the support. The two dimensions are indeed equal when the support is a manifold with positive reach and the probability distribution has a bounded density with respect to the uniform measure on the manifold (e.g. the Hausdorff measure). In particular when the probability distribution has a bounded density with respect to the $d$-dimensional Lebesgue measure,
the volume dimension equals the ambient dimension $d$.

\begin{proposition}

\label{prop:dimension_manifold_bounded_density}

Let $P$ be a probability distribution on $\mathbb{R}^{d}$, and
$d_{{\rm vol}}$ be its volume dimension. Suppose there exists a $d_{M}$-dimensional
manifold $M$ with positive reach satisfying $P(M\cap\mathbb{X})>0$ and ${\rm supp}(P)\subset M$.
If $P$ has a bounded density $p$ with respect to the normalized
$d_{M}$-dimensional Hausdorff measure $\lambda_{d_{M}}$, then  $d_{{\rm vol}}=d_{M}$.
In particular, when $P$ has a bounded density $p$ with respect
to the $d$-dimensional Lebesgue measure $\lambda_{d}$, then $d_{{\rm vol}}=d$.

\end{proposition}

See Section \ref{sec:comparison} for a comparison of the volume dimension with the Hausdorff dimension and other notions of the dimension.

Even though, as we will soon show, our bounds for KDEs hold without any assumptions
on the probability distribution and lead to convergence rates arbitrary
close to the optimal minimax rates, in order to actually achieve such exact optimal rate,
we require weak additional conditions on the probability distributions. Note that,
from the definition of the volume dimension, the ratio 
$\frac{P(\mathbb{B}_{\mathbb{R}^{d}}(x,r))}{r^{\nu}}$
is uniformly bounded for $\nu$ smaller than the volume dimension.

\begin{lemma}

\label{lem:dimension_volume_lower}

Let $P$ be a probability distribution on $\mathbb{R}^{d}$, and $d_{{\rm vol}}$
be its volume dimension. Then for any $\nu\in [0,d_{{\rm vol}})$, there
exists a constant $C_{\nu,P}$ depending only on $P$ and $\nu$ such
that for all $x\in\mathbb{X}$ and $r>0$, 
\begin{equation}
\frac{P(\mathbb{B}_{\mathbb{R}^{d}}(x,r))}{r^{\nu}}\leq C_{\nu,P}.\label{eq:dimension_volume_lower}
\end{equation}

\end{lemma}

For the exact optimal rate, we impose conditions on how the probability volume decay in \eqref{eq:dimension_volume_lower}
behaves with respect to the volume dimension.

\begin{assumption}

\label{ass:dimension_exact}

Let $P$ be a probability distribution $P$ on $\mathbb{R}^{d}$, and $d_{{\rm vol}}$
be its volume dimension. We assume that 
\begin{equation}
\limsup_{r\to 0}\sup_{x\in\mathbb{X}}\frac{P(\mathbb{B}_{\mathbb{R}^{d}}(x,r))}{r^{d_{{\rm vol}}}}<\infty.
\label{eq:dimension_exact}
\end{equation}
 
\end{assumption}

\begin{assumption}

\label{ass:dimension_exact_lower}

Let $P$ be a probability distribution on $\mathbb{R}^{d}$, and $d_{{\rm vol}}$
be its volume dimension. We assume that 
\begin{equation}
\sup_{x\in\mathbb{X}}\liminf_{r\to0}\frac{P(\mathbb{B}_{\mathbb{R}^{d}}(x,r))}{r^{d_{{\rm vol}}}}>0.
\label{eq:dimension_exact_lower}
\end{equation}

\end{assumption}

These assumptions are in fact weak and hold for common probability distributions. For
example, if a probability distribution is supported on a manifold, Assumption~\ref{ass:dimension_exact} and~\ref{ass:dimension_exact_lower}
hold under the same condition as in Proposition~\ref{prop:dimension_manifold_bounded_density}.
In particular, Assumption~\ref{ass:dimension_exact} and~\ref{ass:dimension_exact_lower} hold when the probability distribution has a bounded density with respect to the $d$-dimensional
Lebesgue measure.

\begin{proposition}

\label{prop:dimension_exact_manifold_bounded_density}

Under the same condition as in Proposition~\ref{prop:dimension_manifold_bounded_density}, Assumption
\ref{ass:dimension_exact} and~\ref{ass:dimension_exact_lower} hold.

\end{proposition}

Also, the Assumption~\ref{ass:dimension_exact} and~\ref{ass:dimension_exact_lower} is closed under the convex combination. In other words, a mixture of probability distributions satisfy Assumption~\ref{ass:dimension_exact} and~\ref{ass:dimension_exact_lower} if all its component satisfy those assumptions.

\begin{proposition}
    \label{prop:dimension_assumption_convex}
    The set of probability distributions satisfying Assumption~\ref{ass:dimension_exact} is convex. And so is the set of probability distributions satisfying Assumption~\ref{ass:dimension_exact_lower}.
    
\end{proposition}

We end this section with an example of an unbounded density. In this case, the volume dimension is strictly smaller than the dimension of the support which illustrates why the dimension of the support is not enough to characterize the dimensionality of a distribution.

\begin{example}
\label{ex:density_unbounded}

Let $P$ be a distribution on $\mathbb{R}^{d}$ having a density $p$
with respect to the $d$-dimensional Lebesgue measure. Fix $\beta<d$,
and suppose $p:\mathbb{R}^{d}\to\mathbb{R}$ is defined as 
\[
p(x)=\frac{(d-\beta)\Gamma\left(\frac{d}{2}\right)}{2\pi^{\frac{d}{2}}}\left\Vert x\right\Vert^{-\beta}I(\left\Vert x\right\Vert\leq 1).
\]
Then, for each fixed $r\in[0,1]$, 
\begin{align*}
\sup_{x\in\mathbb{R}^{d}}P(\mathbb{B}_{\mathbb{R}^{d}}(x,r)) & =P(\mathbb{B}_{\mathbb{R}^{d}}(0,r))=r^{d-\beta}.
\end{align*}
Hence from Definition~\ref{def:dimension_volume}, 
the volume dimension is 
\[
d_{{\rm vol}}(P)=d-\beta,
\]
 and from \eqref{eq:dimension_exact} and \eqref{eq:dimension_exact_lower}, Assumption~\ref{ass:dimension_exact} and~\ref{ass:dimension_exact_lower} are satisfied.


\end{example}

\section{Uniform convergence of the Kernel Density Estimator}
\label{sec:kde}

To derive a bound on the performance of a kernel density estimator that is valid uniformly in $h$ and $x \in \mathbb{X}$,
we first rewrite 
$$
\sup_{h\geq l_{n},x\in\mathbb{X}}\left|\hat{p}_{h}(x)-p_{h}(x)\right|
$$
as a supremum over a function class. Formally, for $x\in\mathbb{X}$ and $h\geq l_{n} > 0$,
let $K_{x,h}(\cdot) :=K\left(\frac{x-\cdot}{h}\right)$ and  consider the following class of 
normalized kernel functions  centered around each point in $\mathbb{X}$
and with bandwidth greater than or equal to $l_{n} > 0$:
$$
\tilde{\mathcal{F}}{}_{K,[l_{n},\infty)}:=\left\{ (1/h^{d})K_{x,h}:\,x\in\mathbb{X},\,h\geq l_{n}\right\}.
$$
Then $\sup_{h\geq l_{n},x\in\mathbb{X}}\left|\hat{p}_{h}(x)-p_{h}(x)\right|$
can be rewritten as a supremum of an empirical process indexed by $\tilde{\mathcal{F}}$, that is,
\begin{equation}
\sup_{h\geq l_{n},x\in\mathbb{X}}\left|\hat{p}_{h}(x)-p_{h}(x)\right|
=\sup_{f\in\tilde{\mathcal{F}}{}_{K,[l_{n},\infty)}}\left|\frac{1}{n}\sum_{i=1}^{n}f(X_{i})-\mathbb{E}[f(X)]\right|.\label{eq:kde_supremum_function_class}
\end{equation}

We combine  Talagrand's inequality and a VC type bound to bound \eqref{eq:kde_supremum_function_class}, following the approach of \citet[Theorem 3.1]{Sriperumbudur2012}.
The following version of Talagrand's inequality is from \citet[Theorem 2.3]{Bousquet2002}
and simplified in \citet[Theorem 7.5]{SteinwartC2008}.

\begin{proposition}{\citep[Theorem 2.3]{Bousquet2002}, \citep[Theorem 7.5, Theorem A.9.1]{SteinwartC2008}}
	
	\label{prop:function_talagrand}
	
	Let $(\mathbb{R}^{d},P)$ be a probability space and let $X_{1},\ldots,X_{n}$
	be i.i.d. from $P$. Let $\mathcal{F}$ be a class of functions from
	$\mathbb{R}^{d}$ to $\mathbb{R}$ that is separable in $L_{\infty}(\mathbb{R}^{d})$.
	Suppose all functions $f\in\mathcal{F}$ are $P$-measurable, and
	there exists $B,\sigma>0$ such that 
	$\mathbb{E}_{P}f=0$, $\mathbb{E}_{P}f^{2}\leq\sigma^{2}$, and $\left\Vert f\right\Vert _{\infty}\leq B$,
	for all $f\in\mathcal{F}$. Let 
	\[
	Z:=\sup_{f\in\mathcal{F}}\left|\frac{1}{n}\sum_{i=1}^{n}f(X_{i})\right|,
	\]
	Then for any $\delta>0$, 
	\[
	P\left(Z\geq \mathbb{E}_{P}[Z]+\sqrt{\left(\frac{2}{n}\log\frac{1}{\delta}\right)\left(\sigma^{2}+2B\mathbb{E}_{P}[Z]\right)}
	+\frac{2B\log\frac{1}{\delta}}{3n}\right)\leq\delta.
	\]
	
\end{proposition}

By applying Talagrand's inequality to \eqref{eq:kde_supremum_function_class}, 
$\sup_{h\geq l_{n},x\in\mathbb{X}}\left|\hat{p}_{h}(x)-p_{h}(x)\right|$
can be upper bounded in
terms of $n$, $\left\Vert K_{x,h}\right\Vert _{\infty}$, $\mathbb{E}_{P}[K_{x,h}^{2}]$,
and 
\begin{equation}
\mathbb{E}_{P}\left[\sup_{f\in \tilde{\mathcal{F}}{}_{K,[l_{n},\infty)}}\left|\frac{1}{n}\sum_{i=1}^{n}f(X_{i})-\mathbb{E}[f(X)]\right|\right].\label{eq:kde_supremum_expectation}
\end{equation}
To bound the last term, we use the uniformly bounded VC class assumption on the kernel. The following bound on the expected suprema
of empirical processes of VC classes of
functions is from \citet[Proposition 2.1]{GineG2001}.

\begin{proposition} (\citet[Proposition 2.1]{GineG2001}, \citep[Theorem A.2]{Sriperumbudur2012})
	
	\label{prop:function_vc_bound}
	
	Let $(\mathbb{R}^{d},P)$ be a probability space and let $X_{1},\ldots,X_{n}$
	be i.i.d. from $P$. Let $\mathcal{F}$ be a class of functions from
	$\mathbb{R}^{d}$ to $\mathbb{R}$ that is uniformly bounded VC-class
	with dimension $\nu$, i.e. there exists positive numbers $A$,$B$
	such that, for all $f\in\mathcal{F}$, $\left\Vert f\right\Vert _{\infty}\leq B$, and the covering number $\mathcal{N}(\mathcal{F},L_{2}(Q),\epsilon)$
	satisfies 
	\[
	\mathcal{N}(\mathcal{F},L_{2}(Q),\epsilon)\leq\left(\frac{AB}{\epsilon}\right)^{\nu}.
	\]
	for every probability measure $Q$ on $\mathbb{R}^{d}$ and for
	every $\epsilon\in(0,B)$. Let $\sigma>0$ be a positive number such that $\mathbb{E}_{P}f^{2}\leq\sigma^{2}$ for all
	$f\in\mathcal{F}$. Then there exists a universal constant $C$ not
	depending on any parameters such that 
	\[
	\mathbb{E}_{P}\left[\sup_{f\in\mathcal{F}}\left|\frac{1}{n}\sum_{i=1}^{n}f(X_{i})\right|\right]
	\leq C\left(\frac{\nu B}{n}\log \left(\frac{AB}{\sigma}\right)+\sqrt{\frac{\nu\sigma^{2}}{n}\log\left(\frac{AB}{\sigma}\right)}\right).   
	\]
	
\end{proposition}

By applying Proposition~\ref{prop:function_talagrand} and Proposition~\ref{prop:function_vc_bound} to $\tilde{\mathcal{F}}{}_{K,[l_{n},\infty)}$, it can be shown that the upper bound of 
\[
\sup_{h\geq l_{n},x\in\mathbb{X}}\left|\hat{p}_{h}(x)-p_{h}(x)\right|
\] 
can be written as a function of $\left\Vert K_{x,h}\right\Vert _{\infty}$ and $\mathbb{E}_{P}[K_{x,h}^{2}]$. When the lower bound on the interval $l_{n}$ is not too small,
the terms relating to $\mathbb{E}_{P}[K_{x,h}^{2}]$ are
more dominant. Hence, to get a good upper bound with respect to both $n$ and $h$, it is important to get a tight upper bound for
$\mathbb{E}_{P}[K_{x,h}^{2}]$. Under the
existence of the Lebesgue density of $P$, it can be shown that 
\[
\mathbb{E}_{P}[K_{x,h}^{2}] \leq \|K\|_2 \|p\|_\infty h^d,
\]
by change of variables. (see, e.g. the proof of Proposition A.5. in \citet{Sriperumbudur2012}.)

For general distributions (such as the ones supported on a lower-dimensional manifold), the change of variables argument is no longer directly applicable. However, under an integrability condition on the kernel, detailed below,  we can provide a bound based on the volume dimension. 

\begin{assumption}
    \label{ass:integrable}
	Let $K:\mathbb{R}^{d}\to\mathbb{R}$ be a kernel function with $\left\Vert K\right\Vert _{\infty}<\infty$, and fix $k>0$. We impose an integrability condition: either $d_{{\rm vol}}=0$ or 
	\begin{equation}
	\int_{0}^{\infty}t^{d_{\mathrm{vol}}-1}\sup_{\left\Vert x\right\Vert \geq t}|K(x)|^{k}dt<\infty.\label{eq:kde_condition_integral_finite}
	\end{equation}
	We set $k=2$ by default unless it is specified in otherwise.
\end{assumption}

\begin{remark}
	It is important to emphasize that Assumption~\ref{ass:integrable} is weak, as it is satisfied by commonly used kernels. For instance, if the kernel function $K(x)$ decays at a polynomial rate strictly faster than $d_{\mathrm{vol}}/k$ (which is at most $d/k$) as $x\to\infty$, that is, if 
	\[
	\limsup_{x\to\infty} \left\Vert x\right\Vert ^{d_{\mathrm{vol}}/k+\epsilon}K(x)<\infty,
	\]
	for any $\epsilon >0$, the integrability condition \eqref{eq:kde_condition_integral_finite} is satisfied. Also, if the kernel function $K(x)$ is spherically symmetric, that is, if there exists
	$\tilde{K}:[0,\infty)\to\mathbb{R}$ with $K(x)=\tilde{K}(\left\Vert x\right\Vert )$,
	then the integrability condition \eqref{eq:kde_condition_integral_finite} is satisfied provided $\left\Vert K\right\Vert _{k} <\infty$. Kernels with bounded support also satisfy the condition \eqref{eq:kde_condition_integral_finite}. Thus, most of the commonly used kernels including Uniform, Epanechnikov, and Gaussian kernels satisfy the above integrability condition.
\end{remark}


By combining Assumption~\ref{ass:integrable} and Lemma~\ref{lem:dimension_volume_lower}, we can bound $\mathbb{E}_{P}[K_{x,h}^{2}]$ in terms of the volume dimension $d_{\mathrm{vol}}$.
\begin{lemma}\label{lem:kde_kernel_bound_two} Let $(\mathbb{R}^{d},P)$
	be a probability space and let $X\sim P$. For any kernel $K$ satisfying Assumption~\ref{ass:integrable} with $k>0$,
	the expectation of the $k$-moment of the kernel is upper bounded as 
	\begin{equation}
	\mathbb{E}_{P}\left[\left|K\left(\frac{x-X}{h}\right)\right|^{k}\right]\leq C_{k,P,K,\epsilon}h^{d_{\mathrm{vol}}-\epsilon},
	\label{eq:kde_kernel_bound}
	\end{equation}
	for any $\epsilon\in (0,d_{{\rm vol}})$, where $C_{k,P,K,\epsilon}$ is a constant depending only on $k$, $P$, $K$, and $\epsilon$. Further, if $d_{{\rm vol}}=0$ or under Assumption~\ref{ass:dimension_exact}, $\epsilon$ can be $0$ in \eqref{eq:kde_kernel_bound}.
	\end{lemma}



\subsection{Uniformity on a ray of bandwidths}
\label{subsec:kde_band_ray}

In this subsection, we demonstrate an $L_\infty$ convergence rate for  kernel
density estimators, that is valid is uniformly on a ray of bandwidths $[l_{n},\infty)$.

To apply the VC type bound from Proposition~\ref{prop:function_vc_bound}, 
the function class,
\[
\mathcal{F}{}_{K,[l_{n},\infty)}:=\left\{ K_{x,h}:\,x\in\mathbb{X}, h \geq l_{n} \right\},
\]
should be not too complex. One common approach is to assume that $\mathcal{F}{}_{K,[l_{n},\infty)}$
is a uniformly bounded VC-class, which is defined imposing appropriate bounds on the metric entropy of the function class  \citep{GineG1999,Sriperumbudur2012}.


\begin{assumption} \label{ass:kde_vc} Let $K:\mathbb{R}^{d}\to\mathbb{R}$
	be a kernel function with $\left\Vert K\right\Vert _{\infty},\left\Vert K\right\Vert _{2}<\infty$. We assume that,
	\[
	\mathcal{F}{}_{K,[l_{n},\infty)}:=\left\{ K_{x,h}:\,x\in\mathbb{X},h\geq l_{n}\right\}
	\]
	is a uniformly bounded VC-class with dimension $\nu$, i.e.,
	there exists positive numbers $A$ and $\nu$ such that, for every probability
	measure $Q$ on $\mathbb{R}^{d}$ and for every $\epsilon\in(0,\left\Vert K\right\Vert _{\infty})$,
	the covering numbers $\mathcal{N}(\mathcal{F}_{K,[l_{n},\infty)},L_{2}(Q),\epsilon)$
	satisfies
	\[
	\mathcal{N}(\mathcal{F}_{K,[l_{n},\infty)},L_{2}(Q),\epsilon)\leq\left(\frac{A\left\Vert K\right\Vert _{\infty}}{\epsilon}\right)^{\nu},
	\]
	where the covering number is defined as the minimal number of open balls of radius $\epsilon$ with respect to  $L_2(Q)$ distance whose centers are in $\mathcal{F}{}_{K,[l_{n},\infty)}$ to cover $\mathcal{F}{}_{K,[l_{n},\infty)}$.
\end{assumption}



Since $[l_{n},\infty) \subset (0,\infty)$, one sufficient condition for Assumption~\ref{ass:kde_vc} is  to impose uniformly bounded VC class condition on a larger function class,
\[
\mathcal{F}_{K,(0,\infty)}=\left\{ K_{x,h}:\,x\in\mathbb{X},h>0\right\}. 
\]
This is implied by condition ($K$) in \citet{GineKZ2004} or condition
($K_{1}$) in \citet{GineG2001}, which are standard conditions to assume for the uniform bound on the KDE. In particular, the condition is satisfied when
$K(x)=\phi(p(x))$, where $p$ is a polynomial and $\phi$ is a bounded
real function of bounded variation as in \citet{NolanP1987}, which covers commonly used kernels, such as Gaussian, Epanechnikov, Uniform, etc.

Under Assumption~\ref{ass:integrable} and~\ref{ass:kde_vc}, we derive our main concentration inequality for $\sup_{h\geq l_{n},x\in\mathbb{X}}\bigl|\hat{p}_{h}(x)-p_{h}(x)\bigr|$.
\begin{theorem} \label{thm:kde_band_ray_uniform}
	Let $P$ be a probability distribution and let $K$ be a kernel function
	satisfying Assumption~\ref{ass:integrable} and~\ref{ass:kde_vc}.
	Then, with probability at least $1-\delta$, 
	\begin{equation}
	\sup_{h\geq l_{n},x\in\mathbb{X}}\left|\hat{p}_{h}(x)-p_{h}(x)\right|
	\leq C\left(\frac{\left(\log\left(1 / l_{n}\right)\right)_{+}}{nl_{n}^{d}}+\sqrt{\frac{\left(\log\left(1 / l_{n}\right)\right)_{+}}{nl_{n}^{2d-d_{\mathrm{vol}}+\epsilon}}}
	+\sqrt{\frac{\log\left(2 / \delta\right)}{nl_{n}^{2d-d_{\mathrm{vol}}+\epsilon}}}+\frac{\log\left(2  / \delta\right)}{nl_{n}^{d}}\right),\label{eq:kde_band_ray_uniform_bound}
	\end{equation}
	for any $\epsilon\in (0,d_{{\rm vol}})$, where $C$ is a constant depending only on $A$, $\left\Vert K\right\Vert _{\infty}$, $d$,
	$\nu$, $d_{\mathrm{vol}}$, $C_{k=2,P,K,\epsilon}$, $\epsilon$. Further, if $d_{{\rm vol}}=0$ or under Assumption~\ref{ass:dimension_exact}, $\epsilon$ can be $0$ in \eqref{eq:kde_band_ray_uniform_bound}.
	
\end{theorem}

When $\delta$ is fixed and $l_{n}<1$, the dominating terms in \eqref{eq:kde_band_ray_uniform_bound}
are $\frac{\log(1/l_{n})}{nl_{n}^{d}}$ and $\sqrt{\frac{\log(1/l_{n})}{nl_{n}^{2d-d_{\mathrm{vol}}}}}$.
If $l_{n}$ does not vanish too rapidly, then the second term dominates the upper bound in \eqref{eq:kde_band_ray_uniform_bound} as in the following corollary.

\begin{corollary}
	\label{cor:kde_band_ray_uniform_probconv}
	
	Let $P$ be a probability distribution and let $K$ be a kernel function
	satisfying Assumption~\ref{ass:integrable} and~\ref{ass:kde_vc}.
	Fix $\epsilon\in(0,d_{\rm{vol}})$. Further, if $d_{{\rm vol}}=0$ or under Assumption~\ref{ass:dimension_exact}, $\epsilon$ can be $0$.
	Suppose 
	\[
	\limsup_{n}\frac{\left(\log\left(1 / \ell_{n}\right)\right)_{+}+\log\left(2 / \delta\right)}{n\ell_{n}^{d_{\mathrm{vol}}-\epsilon}}<\infty.
	\]
	Then, with probability at least $1-\delta$, 
	\begin{equation}
	\sup_{h\geq l_{n},x\in\mathbb{X}}\left|\hat{p}_{h}(x)-p_{h}(x)\right|
	\leq C'\sqrt{\frac{(\log(\frac{1}{l_{n}}))_{+}+\log(\frac{2}{\delta})}{nl_{n}^{2d-d_{\mathrm{vol}}+\epsilon}}},
	\label{eq:kde_band_ray_uniform_probconv}
	\end{equation}
	where $C'$
	depending only on $A$, $\left\Vert K\right\Vert _{\infty}$, $d$,
	$\nu$, $d_{\mathrm{vol}}$, $C_{k=2,P,K,\epsilon}$, $\epsilon$.
	
\end{corollary}

\subsection{Fixed bandwidth}
\label{subsec:kde_band_one}

In this subsection, we prove a finite-sample uniform convergence bound on kernel density estimators for one {\it fixed} choice  $h_n >0$ of the bandwidth (we leave the dependence on $n$ explicit in our notation to emphasize that the choice of the bandwidth may still depend on $n$). We are interested in a high probability bound on 
\[
\sup_{x\in\mathbb{X}}\left|\hat{p}_{h_{n}}(x)-p_{h_{n}}(x)\right|.
\]
Of course, the above quantity can be bounded by the results in the previous subsection because
\begin{equation}
\sup_{x\in\mathbb{X}}\left|\hat{p}_{h_{n}}(x)-p_{h_{n}}(x)\right|\leq\sup_{h\geq h_{n},x\in\mathbb{X}}\left|\hat{p}_{h}(x)-p_{h}(x)\right|,\label{eq:kde_band_one_kdebound}
\end{equation}
Therefore, the convergence bound uniform on a ray of bandwidths in Theorem
\ref{thm:kde_band_ray_uniform} and Corollary~\ref{cor:kde_band_ray_uniform_probconv}
is applicable to fixed bandwidth cases.

However, if the set $\mathbb{X}$ is bounded, that is, if there exists $R>0$ such that $\mathbb{X}\subset\mathbb{B}_{\mathbb{R}^{d}}(0,R)$, then, for the kernel density estimator with a  $M_K$-Lipschitz continuous kernel and fixed bandwidth, we can derive a uniform convergence bound without the finite VC condition of \citep{GineG2001, GineKZ2004}
based on the following lemma.
\begin{lemma} \label{lem:kde_vc}
	Suppose there exists $R>0$ with $\mathbb{X}\subset\mathbb{B}_{\mathbb{R}^{d}}(0,R)$. Let the kernel $K$ is $M_K$-Lipschitz continuous. Then for all $\eta\in\left(0,\left\Vert K\right\Vert _{\infty}\right)$,
	the supremum of the $\eta$-covering number $\mathcal{N}(\mathcal{F}_{K,h},L_{2}(Q),\eta)$
	over all measure $Q$ is upper bounded as 
	\[
	\sup_{Q}\mathcal{N}(\mathcal{F}_{K,h},L_{2}(Q),\eta)\leq\left(\frac{2RM_{K}h^{-1}+\left\Vert K\right\Vert _{\infty}}{\eta}\right)^{d}.
	\]
\end{lemma}


\begin{corollary}
	
	\label{cor:kde_uniform_band_one_probconv}
	Suppose there exists $R>0$ with $\mathbb{X}\subset\mathbb{B}_{\mathbb{R}^{d}}(0,R)$. Let $K$ be a $M_K$-Lipschitz continuous kernel function
	satisfying Assumption~\ref{ass:integrable}.
	Fix $\epsilon\in(0,d_{\rm{vol}})$. Further, if $d_{{\rm vol}}=0$ or under Assumption~\ref{ass:dimension_exact}, $\epsilon$ can be $0$. Suppose
	\[
	\limsup_{n}\frac{\left(\log\left(1 / h_{n}\right)\right)_{+}+\log\left(2 / \delta\right)}{nh_{n}^{d_{\mathrm{vol}}-\epsilon}}<\infty.
	\]
	Then  with probability at least $1-\delta$,
	\begin{equation}
	\sup_{x\in\mathbb{X}}\left|\hat{p}_{h_{n}}(x)-p_{h_{n}}(x)\right|
	\leq C''\sqrt{\frac{(\log(\frac{1}{ h_{n}}))_{+}+\log(\frac{2}{\delta})}{nh_{n}^{2d-d_{\mathrm{vol}}+\epsilon}}},
	\label{eq:kde_uniform_band_one_probconv}
	\end{equation}
	where $C''$ is a constant depending only on $R$, $M_{K}$, $\left\Vert K\right\Vert _{\infty}$,
	$d$, $\nu$, $d_{\mathrm{vol}}$, $C_{k=2,P,K,\epsilon}$, $\epsilon$.
	
\end{corollary}

\section{Lower bound for the convergence of the Kernel Density Estimator}
\label{sec:lower}

Consider the fixed bandwidth case. In Corollary
\ref{cor:kde_uniform_band_one_probconv}, it was shown that, with probability $1-\delta$,
\[
\sup_{x\in\mathbb{X}}\left|\hat{p}_{h_{n}}(x)-p_{h_{n}}(x)\right|\leq C_{\delta}''\sqrt{\frac{\left(\log\left(1/h_{n}\right)\right)_{+}}{nh_{n}^{2d-d_{\mathrm{vol}}}}},
\]
where $C_{\delta}''$ might depend on $\delta$ but not on $n$ or $h_{n}$. In this Section, we show that this upper bound is not improvable and is therefore optimal up
to a $\log(1/h_{n})$ term, by showing that there exists a high probability lower bound
of order $1/\sqrt{nh_{n}^{2d-d_{{\rm vol}}}}$. 

\begin{proposition}

\label{prop:bound_lower}

Suppose $P$ is a distribution satisfying Assumption~\ref{ass:dimension_exact_lower}
and with positive volume dimension $d_{{\rm vol}}>0$. Let $K$ be
a kernel function satisfying Assumption~\ref{ass:integrable}
with $k=1$ and $\lim_{t\to0}\inf_{\left\Vert x\right\Vert \leq t}K(x)>0$. Suppose
	$\lim_{n} n h_{n}^{d_{{\rm vol}}} = \infty$.
Then, with probability $1-\delta$, the following holds for all large enough $n$ and small enough $h_{n}$: 
\[
\sup_{x\in\mathbb{X}}\left|\hat{p}_{h_{n}}(x)-p_{h_{n}}(x)\right|\geq C_{P,K,\delta}\sqrt{\frac{1}{nh_{n}^{2d-d_{{\rm vol}}}}}.
\]
where $C_{P,K,\delta}$ is a constant depending only on $P$, $K,$and
$\delta$.

\end{proposition}

This gives an immediate corollary for a ray of bandwidths.

\begin{corollary}

\label{cor:bound_lower_ray}

Assume the same condition as in Proposition \ref{prop:bound_lower}, and suppose $l_{n}\to 0$ with $n l_{n}^{d_{{\rm vol}}} \to \infty$. Then, with probability $1-\delta$, the following holds for all large $n$: 
\[
\sup_{h\geq l_{n}, x\in\mathbb{X}}\left|\hat{p}_{h}(x)-p_{h}(x)\right|\geq C_{P,K,\delta}\sqrt{\frac{1}{nl_{n}^{2d-d_{{\rm vol}}}}}.
\]

\end{corollary}

By combining the lower and upper bounds together, we  conclude that, with high probability,
\[
\sqrt{\frac{1}{n h_{n}^{2d-d_{{\rm vol}}}}}
\lesssim
\sup_{x\in\mathbb{X}} \left|\hat{p}_{h_{n}}(x)-p_{h_{n}}(x)\right|
\lesssim
\sqrt{\frac{(\log(\frac{1}{h_{n}}))_{+}}{nh_{n}^{2d-d_{\mathrm{vol}}}}},
\]
for all large enough $n$. Similar holds for a ray of bandwidths as well. They imply that the uniform convergence KDE bounds in our paper
are optimal up to $\log(1/h_{n})$ terms for both the fixed bandwidth and the ray on bandwidths cases.

\begin{example}[Example~\ref{ex:density_unbounded}, revisited]

Let $P$ be as in Example~\ref{ex:density_unbounded} and let $K$ be
any Lipschitz continuous kernel function with $K(0)>0$ and compact
support. It can be easily checked that the conditions in Corollary~\ref{cor:kde_uniform_band_one_probconv} are satisfied with $R = 2$, $d_{\rm vol} = d- \beta$ and the kernel satisfies the integrability Assumption~\ref{ass:integrable} with
$k =1, 2$. It can be also shown that $\lim_{t\to0}\inf_{\left\Vert x\right\Vert \leq t}K(x)>0$. Therefore, 
for small enough $h_{n}$, Corollary~\ref{cor:kde_uniform_band_one_probconv}
and Proposition~\ref{prop:bound_lower} imply
\[
C'\sqrt{\frac{1}{nh_{n}^{d+\beta}}}\leq\sup_{x\in\mathbb{X}}\left|\hat{p}_{h_{n}}(x)-p_{h_{n}}(x)\right|\leq C''\sqrt{\frac{\log(\frac{1}{h_{n}})}{nh_{n}^{d+\beta}}},
\]
with high probability for all large enough $n$.
That is, the $L_\infty$ convergence rate of the KDE is of order
$\sqrt{\frac{1}{nh_{n}^{d+\beta}}}$ (up to a $\log(1/h_{n})$ term).
Hence, although it has a Lebesgue density, its convergence rate is
different from $\sqrt{\frac{1}{nh_{n}^{d}}}$, which is the usual rate
for probability distributions with bounded Lebesgue density.

\end{example}

\section{Uniform convergence of the Derivatives of the Kernel Density Estimator}
\label{sec:derivative}

In this final section, we provide analogous finite-sample uniform convergence bound on
the derivatives of the kernel density estimator. For a nonnegative
integer vector $s=(s_{1},\ldots,s_{d})\in (\{0\} \cup \mathbb{N})^{d}$,
define $|s|=s_{1}+\cdots+s_{d}$ and 
\[
D^{s}:=\frac{\partial^{|s|}}{\partial x_{1}^{s_{1}}\cdots\partial x_{d}^{s_{d}}}.
\]
For $D^{s}$ operator to be well defined and interchange with integration,
we need the following smoothness condition on the kernel $K$.

\begin{assumption}
	
	\label{ass:derivative_leibniz}
	
	For given $s\in \left(\{0\} \cup \mathbb{N}\right)^d$, let $K:\mathbb{R}^{d}\to\mathbb{R}$
	be a kernel function satisfying such that the partial derivative
	$D^{s}K:\mathbb{R}^{d}\to\mathbb{R}$ exists and $\left\Vert D^{s}K\right\Vert _{\infty}<\infty$.
	
\end{assumption}

Under Assumption~\ref{ass:derivative_leibniz}, Leibniz's rule is
applicable and, for each $x \in \mathbb{X}$, $D^{s}\hat{p}_{h}(x)-D^{s}p_{h}(x)$ can be written
as 
\[
D^{s}\hat{p}_{h}(x)-D^{s}p_{h}(x) 
=\frac{1}{n}\sum_{i=1}^{n}\frac{1}{h^{d+|s|}}D^{s}K_{x,h}(X_{i})-\mathbb{E}_{P}\left[\frac{1}{h^{d+|s|}}D^{s}K_{x,h}\right],
\]
where $K_{x,h}(\cdot)=K\left(\frac{x-\cdot}{h}\right)$, as defined it in Section
\ref{sec:kde}. Following the arguments from Section~\ref{sec:kde}, let 
\[
\mathcal{F}{}_{K,[l_{n},\infty)}^{s}:=\left\{ D^{s}K_{x,h}:\,x\in\mathbb{X},h\geq l_{n}\right\} 
\]
be a class of unnormalized kernel functions centered on $\mathbb{X}$
and bandwidth greater than or equal to $l_{n}$, and let 
\[
\tilde{\mathcal{F}}{}_{K,[l_{n},\infty)}^{s}:=\left\{ \frac{1}{h^{d+|s|}}D^{s}K_{x,h}:\,x\in\mathbb{X},\,h\geq l_{n}\right\} 
\]
be a class of normalized kernel functions. Then 
$\sup_{h\geq l_{n},x\in\mathbb{X}}\left|D^{s}\hat{p}_{h}(x)-D^{s}p_{h}(x)\right|$
can be rewritten as 
\begin{equation}
\sup_{h\geq l_{n},x\in\mathbb{X}}\left|D^{s}\hat{p}_{h}(x)-D^{s}p_{h}(x)\right|
=\sup_{f\in\tilde{\mathcal{F}}{}_{K,[l_{n},\infty)}^{s}}\left|\frac{1}{n}\sum_{i=1}^{n}f(X_{i})-\mathbb{E}[f(X)]\right|.\label{eq:derivative_supremum_function_class}
\end{equation}

To derive a good upper bound on 
$
\sup_{h\geq l_{n},x\in\mathbb{X}}\left|D^{s}\hat{p}_{h}(x)-D^{s}p_{h}(x)\right|
$, 
it is important to first show a tight upper bound for
$\mathbb{E}_{P}[(D^{s}K_{x,h})^{2}]$. Towards that end, we impose the following integrability condition.
\begin{assumption}
    \label{ass:derivative_integrable}
	The derivative of kernel is such that
	\begin{equation}
	\int_{0}^{\infty}t^{d_{\mathrm{vol}}-1}\sup_{\left\Vert x\right\Vert \geq t}(D^s K)^{2}(x) dt<\infty.\label{eq:deriv_kde_condition_integral_finite}
	\end{equation}
\end{assumption}

Under Assumption~\ref{ass:derivative_integrable},  we can bound $\mathbb{E}_{P}[D^{s}K_{x,h}^{2}]$ in terms of the volume dimension $d_{\mathrm{vol}}$ as follows.

\begin{lemma}\label{lem:derivative_kernel_bound_two} Let $(\mathbb{R}^{d},P)$
	be a probability space and let $X\sim P$. For any kernel $K$ satisfying Assumption~\ref{ass:derivative_integrable},
	the expectation of the square of the derivative of the kernel is upper bounded as 
	\begin{equation}
	\mathbb{E}_{P}\left[\left(D^{s}K\left(\frac{x-X}{h}\right)\right)^{2}\right]\leq C_{s,P,K,\epsilon}h^{d_{\mathrm{vol}}-\epsilon},
	\label{eq:derivative_kernel_bound}
	\end{equation}
	for any $\epsilon \in(0,d_{\rm{vol}})$, where $C_{s,P,K,\epsilon}$ is a constant depending only on $s$, $P$, $K$, $\epsilon$. Further, if $d_{\rm{vol}}=0$ or under Assumption~\ref{ass:dimension_exact}, $\epsilon$ can be $0$ in \eqref{eq:derivative_kernel_bound}.
\end{lemma}



To apply the VC type bound on  \eqref{eq:derivative_supremum_function_class},
the function class $\mathcal{F}{}_{K,[l_{n},\infty)}^{s}$ should
be not too complex. Like in Section~\ref{sec:kde}, we
assume that $\mathcal{F}{}_{K,[l_{n},\infty)}^{s}$ is a uniformly
bounded VC-class.

\begin{assumption} \label{ass:derivative_vc} 
	Let $K:\mathbb{R}^{d}\to\mathbb{R}$
	be a kernel function with $\left\Vert D^{s}K\right\Vert _{\infty},\left\Vert D^{s}K\right\Vert _{2}<\infty$. We assume that 
	$$
	\mathcal{F}{}_{K,[l_{n},\infty)}^{s} :=\left\{ D^{s}K_{x,h}:\,x\in\mathbb{X},h\geq l_{n}\right\}
	$$
	is a uniformly bounded VC-class with dimension $\nu$, i.e.
	there exists positive numbers $A$ and $\nu$ such that, for every probability
	measure $Q$ on $\mathbb{R}^{d}$ and for every $\epsilon\in(0,\left\Vert D^s K\right\Vert _{\infty})$,
	the covering numbers $\mathcal{N}(\mathcal{F}^s_{K,[l_{n},\infty)},L_{2}(Q),\epsilon)$
	satisfies
	\[
	\mathcal{N}({\mathcal{F}_{K,[l_{n},\infty)}^s},L_{2}(Q),\epsilon)\leq\left(\frac{A\left\Vert D^{s}K\right\Vert _{\infty}}{\epsilon}\right)^{\nu}.
	\]

\end{assumption}

Finally, to bound $\sup_{h\geq l_{n},x\in\mathbb{X}}\left|D^{s}\hat{p}_{h}(x)-D^{s}p_{h}(x)\right|$
with high probability, we combine the Talagrand inequality and VC type
bound with Lemma~\ref{lem:derivative_kernel_bound_two}. The
following theorem provides a high probability upper bound for \eqref{eq:derivative_supremum_function_class},
and is analogous to Theorem~\ref{thm:kde_band_ray_uniform}.

\begin{theorem} \label{thm:derivative_band_ray_uniform}
	
	Let $P$ be a distribution and $K$ be a kernel function satisfying
	Assumption~\ref{ass:derivative_leibniz},~\ref{ass:derivative_integrable}, and~\ref{ass:derivative_vc}. Then, with probability at least $1-\delta$,
	\begin{align}
	& \sup_{h\geq l_{n},x\in\mathbb{X}}\left|D^{s}\hat{p}_{h}(x)-D^{s}p_{h}(x)\right|\nonumber\\
	& \leq C\left(\frac{\left(\log\left(1 / l_{n}\right)\right)_{+}}{nl_{n}^{d+|s|}}+\sqrt{\frac{\left(\log\left(1 / l_{n}\right)\right)_{+}}{nl_{n}^{2d+2|s|-d_{\mathrm{vol}}+\epsilon}}}
	+\sqrt{\frac{\log\left(2 / \delta\right)}{nl_{n}^{2d+2|s|-d_{\mathrm{vol}}+\epsilon}}}+\frac{\log\left(2 / \delta\right)}{nl_{n}^{d+|s|}}\right),\label{eq:derivative_band_ray_uniform_bound}
	\end{align}
	for any $\epsilon\in(0,d_{\rm{vol}})$, where $C$ is a constant depending only on $A$, $\left\Vert D^{s}K\right\Vert _{\infty}$, $d$,
	$\nu$, $d_{\mathrm{vol}}$, $C_{s,P,K,\epsilon}$, $\epsilon$. Further, if $d_{\rm{vol}}=0$ or under Assumption~\ref{ass:dimension_exact}, $\epsilon$ can be $0$ in \eqref{eq:derivative_band_ray_uniform_bound}.
	
\end{theorem}

When $l_{n}$ is not going to $0$ too fast, then $\sqrt{\frac{\log(1/l_{n})}{nl_{n}^{2d+2|s|-d_{\mathrm{vol}}}}}$
term dominates the upper bound in \eqref{eq:derivative_band_ray_uniform_bound}
as follows.

\begin{corollary}
	
	\label{cor:derivative_band_ray_uniform_probconv}
	
	Let $P$ be a distribution and $K$ be a kernel function satisfying
	Assumption~\ref{ass:derivative_leibniz},~\ref{ass:derivative_integrable}, and~\ref{ass:derivative_vc}.
	Suppose 
	\[
	\limsup_{n}\frac{\left(\log\left(1 / l_{n}\right)\right)_{+}+\log\left(2 / \delta\right)}{nl_{n}^{d_{\mathrm{vol}}-\epsilon}}<\infty,
	\]
	for fixed $\epsilon\in(0,d_{\rm{vol}})$. Then, with probability at least $1-\delta$,
	\begin{equation}
	\sup_{h\geq l_{n},x\in\mathbb{X}}\left|D^{s}\hat{p}_{h}(x)-D^{s}p_{h}(x)\right|
	\leq C'\sqrt{\frac{\left(\log\left(1 / l_{n}\right)\right)_{+}+\log\left(2 / \delta\right)}{nl_{n}^{2d+2|s|-d_{\mathrm{vol}}+\epsilon}}},
	\label{eq:derivative_band_ray_uniform_probconv}
	\end{equation}
	where $C'$ is a constant depending only on $A$, $\left\Vert D^{s}K\right\Vert _{\infty}$, $d$,
	$\nu$, $d_{\mathrm{vol}}$, $C_{s,P,K,\epsilon}$, $\epsilon$. Further, if $d_{\rm{vol}}=0$ or under Assumption~\ref{ass:dimension_exact}, $\epsilon$ can be $0$.
	
\end{corollary}

We now turn to the case of a fixed  bandwidth $h_{n}>0$. We are interested in a high probability bound on 
\[
\sup_{x\in\mathbb{X}}\left|D^{s}\hat{p}_{h_{n}}(x)-D^{s}p_{h_{n}}(x)\right|.
\]
Of course, Theorem~\ref{thm:derivative_band_ray_uniform} and Corollary~\ref{cor:derivative_band_ray_uniform_probconv}
are applicable to the fixed bandwidth case.

But if the support of $P$ is bounded, then, for a $M_K$-Lipschitz continuous derivative of kernel density estimator and fixed bandwidth, we can again derive a uniform convergence bound without the finite VC condition of \citep{GineG2001, GineKZ2004}. 

\begin{lemma} \label{lem:derivative_vc}
	Suppose there exists $R>0$ with $\mathbb{X}\subset\mathbb{B}_{\mathbb{R}^{d}}(0,R)$.
	Also, suppose that $D^{s}K$ is $M_{K}$-Lipschitz, i.e. 
	$$
	\left\Vert D^{s}K(x)-D^{s}K(y)\right\Vert _{2}\leq M_{K}\left\Vert x-y\right\Vert _{2}.
	$$
	Then for all $\eta\in\left(0,\left\Vert D^{s}K\right\Vert _{\infty}\right)$,
	the supremum of the $\eta$-covering number $\mathcal{N}(\mathcal{F}_{K,h}^{s},L_{2}(Q),\eta)$
	over all measure $Q$ is upper bounded as 
	\[
	\sup_{Q}\mathcal{N}(\mathcal{F}_{K,h}^{s},L_{2}(Q),\eta)\leq\left(\frac{2RM_{K}h^{-1}+\left\Vert D^{s}K\right\Vert _{\infty}}{\eta}\right)^{d}.
	\]
\end{lemma}
\begin{corollary}
	
	\label{cor:derivative_uniform_band_one_probconv}
	Suppose there exists $R>0$ with $\mathrm{supp}(P) = \mathbb{X}\subset\mathbb{B}_{\mathbb{R}^{d}}(0,R)$. Let $K$ be a kernel function with $M_K$-Lipschitz continuous derivative satisfying Assumption~\ref{ass:derivative_integrable}.
	If
	\[
	\limsup_{n}\frac{\left(\log\left(1 / h_{n}\right)\right)_{+}+\log\left(2 / \delta\right)}{nh_{n}^{d_{\mathrm{vol}}-\epsilon}}<\infty,
	\]
	for fixed $\epsilon\in(0,d_{\rm{vol}})$. Then, with probability at least $1-\delta$, 
	\begin{equation}
	\sup_{x\in\mathbb{X}}\left|D^{s}\hat{p}_{h}(x)-D^{s}p_{h}(x)\right|
	\leq C''\sqrt{\frac{(\log(\frac{1}{h_{n}}))_{+}+\log(\frac{2}{\delta})}{nh_{n}^{2d+2|s|-d_{\mathrm{vol}+\epsilon}}}},
	\label{eq:derivative_uniform_band_one_probconv}
	\end{equation}
	where $C''$ is a constant depending only on $A$, $\left\Vert D^{s}K\right\Vert _{\infty}$, $d$,
	$M_k$, $d_{\mathrm{vol}}$, $C_{s,P,K,\epsilon}$, $\epsilon$.
	Further, if $d_{\rm{vol}}=0$ or under Assumption~\ref{ass:dimension_exact}, $\epsilon$ can be $0$.
	
\end{corollary}

\bibliographystyle{plainnat}
\bibliography{myref}

\appendix

\clearpage
\newpage

\section*{\centerline{SUPPLEMENTARY MATERIAL}}

\section{Backgrounds and Basic Definitions}
\label{sec:definition}

First, we define the Hausdorff measure (\citep[Section 6]{Pesin1997}, \citep[Section 2.2]{Falconer2014}), which is a generalization of the Lebesgue
measure to lower dimensional subsets of $\mathbb{R}^{d}$. For a subset $A\subset\mathbb{R}^{d}$,
we let ${\rm diam}(A)$ be its diameter, that is 
\[
{\rm diam}(A)=\sup\{\left\Vert x-y\right\Vert :\,x,y\in A\}.
\]

\begin{definition}

Fix $\nu>0$ and $\delta>0$. For any set $A\subset\mathbb{R}^{d}$,
define $H_{\delta}^{\nu}$ be 
\[
H_{\delta}^{\nu}(A):=\inf\left\{ \sum_{i=1}^{\infty}({\rm diam}U_{i})^{\nu}:\,A\subset\bigcup_{i=1}^{\infty}U_{i}\text{ and }{\rm diam}(U_{i})<\delta\right\} ,
\]
where the infimum is over all countable covers of $A$ by sets $U_{i}\subset\mathbb{R}^{d}$
satisfying ${\rm diam}(U_{i})<\delta$. Then, let the \emph{$\nu$-dimensional
Hausdorff measure} $H^{\nu}$ be 
\[
H^{\nu}(A):=\lim_{\delta\to0}H_{\delta}^{\nu}(A).
\]

\end{definition}

Then, the Hausdorff dimension of a set is the infimum over dimensions that make the Hausdorff measure on that set to be $0$.
\begin{definition}
	\label{def:definition_hausdorff_dimension_set}
	
	For any set $A\subset\mathbb{R}^{d}$, its \emph{Hausdorff dimension} $d_{H}(A)$ is 
	\[
	d_{H}(A):=\inf\left\{ \nu:\,H^{\nu}(A)=0\right\} .
	\]
	
\end{definition}

We use the normalized $\nu$-dimensional Hausdorff measure so that
when $\nu$ is an integer, its measure on $\nu$-dimensional unit
cube is $1$. This can be done by defining the normalized $\nu$-dimensional
Hausdorff measure $\lambda_{\nu}$ as 
\[
\lambda_{\nu}=\frac{\pi^{\frac{\nu}{2}}}{2^{\nu}\Gamma(\frac{\nu}{2}+1)}H^{\nu}.
\]

Now, we define the reach, which is a regularity parameter in geometric measure theory. Given a closed subset $A\subset\mathbb{R}^{d}$,
the medial axis of $A$, denoted by ${\rm Med}(A)$, is the subset
of $\mathbb{R}^{d}$ composed of the points that have at least two
nearest neighbors on $A$. Namely, denoting by $d(x,A)=\inf_{q\in A}||q-x||$
the distance function of a generic point $x$  to $A$, 
\begin{equation}
{\rm Med}(A)=\left\{ x\in\mathbb{R}^{d} \setminus A |\exists q_{1}\neq q_{2}\in A,||q_{1}-x||=||q_{2}-x||=d(x,A)\right\} .\label{eq:definition_medialaxis}
\end{equation}
The reach of $A$ is then defined as the minimal distance from $A$ to ${\rm Med}(A)$. 
\begin{definition}
	The \emph{reach} of a closed subset $A\subset\mathbb{R}^{d}$ is defined
	as 
	\begin{align}
	\tau_{A}=\inf_{q\in A}d\left(q,{\rm Med}(A)\right) = \inf_{q\in A,x\in \mathrm{Med}(A)}||q-x||.\label{eq:definition_reach_medial_axis}
	\end{align}
\end{definition}

\section{Proof for Section \ref{sec:dimension}}

We show Lemma \ref{lem:dimension_volume_lower} first, which is a simple argument from the definition of $d_{{\rm vol}}$ in \eqref{eq:kde_condition_prob_upper} in Definition \ref{def:dimension_volume}.

\textbf{Lemma \ref{lem:dimension_volume_lower}.} \textit{
Let $P$ be a probability distribution on $\mathbb{R}^{d}$, and $d_{{\rm vol}}$
be its volume dimension. Then for any $\nu\in [0,d_{{\rm vol}})$, there
exists a constant $C_{\nu,P}$ depending only on $P$ and $\nu$ such
that for all $x\in\mathbb{X}$ and $r>0$, 
\[
\frac{P(\mathbb{B}_{\mathbb{R}^{d}}(x,r))}{r^{\nu}}\leq C_{\nu,P}.
\]
}

\begin{proof}[Proof of Lemma \ref{lem:dimension_volume_lower}]
	
	From the definition of $d_{{\rm vol}}$ in \eqref{eq:kde_condition_prob_upper} in Definition \ref{def:dimension_volume},
	$\nu\in[0,d_{{\rm vol}})$ implies that 
	\[
	\limsup_{r\to0}\sup_{x\in\mathbb{X}}\frac{P(\mathbb{B}_{\mathbb{R}^{d}}(x,r))}{r^{\nu}}<\infty.
	\]
	Then there exist $r_{0}>0$ and $C_{\nu,P}'>0$ such that for all
	$r\leq r_{0}$ and for all $x\in\mathbb{X}$, 
	\begin{equation}
	\frac{P(\mathbb{B}_{\mathbb{R}^{d}}(x,r))}{r^{\nu}}\leq C_{\nu,P}'. \label{eq:dimension_volume_lower_radius_small}
	\end{equation}
	And for all $r\geq r_{0}$ and for all $x\in\mathbb{X}$, 
	\begin{equation}
	\frac{P(\mathbb{B}_{\mathbb{R}^{d}}(x,r))}{r^{\nu}}\leq\frac{1}{r_{0}^{\nu}}. \label{eq:dimension_volume_lower_radius_all}
	\end{equation}
	Hence combining \eqref{eq:dimension_volume_lower_radius_small} and \eqref{eq:dimension_volume_lower_radius_all} gives that for all $r>0$ and for all $x\in\mathbb{X}$, 
	\[
	\frac{P(\mathbb{B}_{\mathbb{R}^{d}}(x,r))}{r^{\nu}}\leq\max\left\{ C_{\nu,P}',\frac{1}{r_{0}^{\nu}}\right\} .
	\]
	
\end{proof}

Then we can show Proposition \ref{prop:dimension_manifold} by using Lemma \ref{lem:dimension_volume_lower} and the definition of Hausdorff dimension in Definition \ref{def:definition_hausdorff_dimension_set}.

\textbf{Proposition \ref{prop:dimension_manifold}.} \textit{
	Let $P$ be a probability distribution on $\mathbb{R}^{d}$, and $d_{{\rm vol}}$
	be its volume dimension.
	Suppose there exists a set $A$ satisfying $P(A\cap\mathbb{X})>0$ and with Hausdorff dimension $d_{H}$. Then $0 \leq d_{{\rm vol}}\leq d_{H}$. Hence if $A$ is a $d_{M}$-dimensional manifold, then $0 \leq d_{{\rm vol}}\leq d_{M}$.
	In particular, for any probability distribution $P$ on $\mathbb{R}^{d}$,  $0 \leq d_{{\rm vol}}\leq d$.
	Also, if $P$ has a point mass, i.e. there exists $x\in\mathbb{X}$
	with $P(\{x\})>0$, then $d_{{\rm vol}} = 0$.
}

\begin{proof}[Proof of Proposition \ref{prop:dimension_manifold}]

We first show $d_{\rm{vol}}\geq 0$. For any $x\in\mathbb{X}$ and $r\geq0$, 
\[
\frac{P(\mathbb{B}_{\mathbb{R}^{d}}(x,r))}{r^{0}}\leq 1<\infty.
\]
Hence $d_{{\rm vol}}\geq 0$ holds.

Now we show $d_{{\rm vol}}\leq d_{H}=d_{H}(A)$. Fix any $\nu<d_{{\rm vol}}$,
and we will show that $H^{\nu}(A\cap\mathbb{X})>0$. Let $\left\{ U_{i}\right\} $
be a countable cover of $A\cap\mathbb{X}$, i.e. $A\cap\mathbb{X}\subset\bigcup_{i=1}^{\infty}U_{i}$,
and let $r_{i}={\rm diam}(U_{i})$. For each $i$, we can assume that
$U_{i}\cap(A\cap\mathbb{X})\neq\emptyset$ and choose $x_{i}\in U_{i}\cap(A\cap\mathbb{X})$.
Then $U_{i}\subset\overline{\mathbb{B}_{\mathbb{R}^{d}}(x_{i},r_{i})}\subset\mathbb{B}_{\mathbb{R}^{d}}(x_{i},2r_{i})$,
and hence 
\[
A\cap\mathbb{X}\subset\bigcup_{i=1}^{\infty}\mathbb{B}_{\mathbb{R}^{d}}(x_{i},2r_{i}).
\]
Then with $x_{i}\in\mathbb{X}$, applying \eqref{eq:dimension_volume_lower} from Lemma \ref{lem:dimension_volume_lower} gives 
\begin{align*}
P(A\cap\mathbb{X}) & <P\left(\bigcup_{i=1}^{\infty}\mathbb{B}_{\mathbb{R}^{d}}(x_{i},2r_{i})\right)=\sum_{i=1}^{\infty}P(\mathbb{B}_{\mathbb{R}^{d}}(x_{i},2r_{i}))\\
& \leq\sum_{i=1}^{\infty}2^{\nu}C_{\nu,P}r_{i}^{\nu}.
\end{align*}
Hence 
\[
\sum_{i=1}^{\infty}r_{i}^{\nu}\geq\frac{P(A\cap\mathbb{X})}{2^{\nu}C_{\nu,P}}>0.
\]
Since this holds for arbitrary covers of $A\cap\mathbb{X}$, $H_{\delta}^{\nu}(A\cap\mathbb{X})\geq\frac{P(A\cap\mathbb{X})}{2^{\nu}C_{\nu,P}}$
for all $\delta>0$. And $A\cap\mathbb{X}\subset A$ implies 
\[
H^{\nu}(A)\geq H^{\nu}(A\cap\mathbb{X})=\lim_{\delta\to0}H_{\delta}^{\nu}(A\cap\mathbb{X})\geq\frac{P(A\cap\mathbb{X})}{2^{\nu}C_{\nu,P}}>0.
\]
Since this holds for arbitrary $\nu<d_{{\rm vol}}$, the definition of Hausdorff dimension in Definition \ref{def:definition_hausdorff_dimension_set} gives that 
\[
d_{H}=\inf\left\{ \nu:\,H^{\nu}(A)=0\right\} \geq d_{{\rm vol}}.
\]

Now, if $A$ is a $d_{M}$-dimensional manifold, then the Hausdorff dimension of $A$ is $d_{M}$, and hence $0 \leq d_{{\rm vol}}\leq d_{M}$ holds. 
In particular, setting $A=\mathbb{R}^{d}$ gives $0 \leq d_{{\rm vol}}\leq d$ for all probability distributions. Also, if there exists $x\in\mathbb{X}$ with $P(\{x\})>0$, then setting $A=\{x\}$ gives $d_{{\rm vol}} = 0$.
\end{proof}

Proposition \ref{prop:mixture} is again a simple argument from the definition of $d_{{\rm vol}}$ in \eqref{eq:kde_condition_prob_upper} in Definition \ref{def:dimension_volume}.

\textbf{Proposition \ref{prop:mixture}.} \textit{
Let $P_{1},\ldots,P_{m}$ be probability distributions on $\mathbb{R}^{d}$,  and  $\lambda_{1},\ldots,\lambda_{m}\in(0,1)$ with $\sum_{i=1}^{m}\lambda_{i}=1$. Then 
\[
d_{{\rm vol}}\left(\sum_{i=1}^{m}\lambda_{i}P_{i}\right)=\min\left\{ d_{{\rm vol}}(P_{i}):\,1\leq i\leq m\right\} .
\]
In particular, when $d_{{\rm vol}}$ is understood as a real-valued function on the space of probability distributions, both its sublevel sets and superlevel sets are convex.
}

\begin{proof}[Proof of Proposition \ref{prop:mixture}]

It is enough to show for the case $m=2$. Let $P:=\lambda_{1}P_{1}+\lambda_{2}P_{2}$.

We first show $d_{{\rm vol}}(P)\geq\min\left\{ d_{{\rm vol}}(P_{1}),d_{{\rm vol}}(P_{2})\right\} $.
Fix $\nu<\min\left\{ d_{{\rm vol}}(P_{1}),d_{{\rm vol}}(P_{2})\right\} $,
then Definition \ref{def:dimension_volume} gives that
\[
\limsup_{r\to0}\sup_{x\in\mathbb{X}}\frac{P_{1}(\mathbb{B}_{\mathbb{R}^{d}}(x,r))}{r^{\nu}},\,\limsup_{r\to0}\sup_{x\in\mathbb{X}}\frac{P_{2}(\mathbb{B}_{\mathbb{R}^{d}}(x,r))}{r^{\nu}}<\infty.
\]
And hence 
\begin{align*}
\limsup_{r\to0}\sup_{x\in\mathbb{X}}\frac{P(\mathbb{B}_{\mathbb{R}^{d}}(x,r))}{r^{\nu}} & =\limsup_{r\to0}\sup_{x\in\mathbb{X}}\left\{ \frac{\lambda_{1}P_{1}(\mathbb{B}_{\mathbb{R}^{d}}(x,r))}{r^{\nu}}+\frac{\lambda_{2}P_{2}(\mathbb{B}_{\mathbb{R}^{d}}(x,r))}{r^{\nu}}\right\} \\
 & \leq\lambda_{1}\limsup_{r\to0}\sup_{x\in\mathbb{X}}\frac{P_{1}(\mathbb{B}_{\mathbb{R}^{d}}(x,r))}{r^{\nu}}+\lambda_{2}\limsup_{r\to0}\sup_{x\in\mathbb{X}}\frac{P_{2}(\mathbb{B}_{\mathbb{R}^{d}}(x,r))}{r^{\nu}}<\infty.
\end{align*}
And hence $d_{{\rm vol}}(P)\geq\min\left\{ d_{{\rm vol}}(P_{1}),d_{{\rm vol}}(P_{2})\right\} $
holds.

Next, we show $d_{{\rm vol}}(P)\leq\min\left\{ d_{{\rm vol}}(P_{1}),d_{{\rm vol}}(P_{2})\right\} $.
Without loss of generality, suppose $d_{{\rm vol}}(P_{1})\leq d_{{\rm vol}}(P_{2})$,
and fix $\nu>d_{{\rm vol}}(P_{1})$. Then Definition \ref{def:dimension_volume}
gives that 
\[
\limsup_{r\to0}\sup_{x\in\mathbb{X}}\frac{P_{1}(\mathbb{B}_{\mathbb{R}^{d}}(x,r))}{r^{\nu}}=\infty.
\]
Then from $P\geq\lambda_{1}P_{1}$,
\begin{align*}
\limsup_{r\to0}\sup_{x\in\mathbb{X}}\frac{P(\mathbb{B}_{\mathbb{R}^{d}}(x,r))}{r^{\nu}} & \geq\limsup_{r\to0}\sup_{x\in\mathbb{X}}\frac{\lambda_{1}P_{1}(\mathbb{B}_{\mathbb{R}^{d}}(x,r))}{r^{\nu}}=\infty.
\end{align*}
And hence $d_{{\rm vol}}(P)\leq d_{{\rm vol}}(P_{1})=\min\left\{ d_{{\rm vol}}(P_{1}),d_{{\rm vol}}(P_{2})\right\} $
holds.

\end{proof}

For Proposition \ref{prop:dimension_manifold_bounded_density} and \ref{prop:dimension_exact_manifold_bounded_density}, we need to bound the volume of the ball on the manifold. The following is rephrased from Lemma 3 in \citet{KimRW2019}.

\begin{lemma}
	
	\label{lem:dimension_volume_bound_upper}
	
	Let $M\subset\mathbb{R}^{d}$ be a $d_{M}$-dimensional submanifold
	with reach $\tau_{M}$. For a subset $U\subset M$ and $r<\tau_{M}$,
	let $U_{r}:=\{x\in\mathbb{R}^{d}:\,{\rm dist}(x,U)<r\}$ be an $r$-neighborhood
	of $U$ in $\mathbb{R}^{d}$. Then 
	\[
	\lambda_{d_{M}}(U)\leq\frac{d!}{d_{M}!}r^{d_{M}-d}\lambda_{d}(U_{r}).
	\]
	
\end{lemma}

Then, the following Lemma is by combining Lemma 5.3 in \citet{NiyogiSW2008}
and Lemma \ref{lem:dimension_volume_bound_upper}.

\begin{lemma}
	
	\label{lem:dimension_ball_volume_bound}
	
	Let $M\subset\mathbb{R}^{d}$ be a $d_{M}$-dimensional submanifold
	with reach $\tau_{M}$. Then, for $x\in M$ and $r<\tau_{M}$, 
	\[
	\left(1-\frac{r^{2}}{4\tau_{M}^{2}}\right)^{\frac{d_{M}}{2}}r^{d_{M}}\omega_{d}\leq\lambda_{d_{M}}(M\cap\mathbb{B}_{\mathbb{R}^{d}}(x,r))\leq\frac{d!}{d_{M}!}2^{d}r^{d_{M}}\omega_{d}.
	\]
	
\end{lemma}

\begin{proof}[Proof of Lemma \ref{lem:dimension_ball_volume_bound}]
	
	The LHS inequality is from Lemma 5.3 in \citet{NiyogiSW2008}. The RHS inequality is
	applying $U=M\cap\mathbb{B}_{\mathbb{R}^{d}}(x,r)$ to Lemma \ref{lem:dimension_volume_bound_upper}
	and $\lambda_{d}(U_{r})\leq\lambda_{d}(\mathbb{B}_{\mathbb{R}^{d}}(x,2r))=(2r)^{d}\omega_{d}$.
	
\end{proof}

Now, we show Proposition \ref{prop:dimension_manifold_bounded_density} and \ref{prop:dimension_exact_manifold_bounded_density} simultaneously via the following Proposition:

\begin{proposition}
	
	\label{prop:dimension_manifold_bounded_density_with_exact}
	
	Let $P$ be a probability distribution on $\mathbb{R}^{d}$, and
	$d_{{\rm vol}}$ be its volume dimension. Suppose there exists a $d_{M}$-dimensional
	manifold $M$ with positive reach satisfying $P(M\cap\mathbb{X})>0$ and ${\rm supp}(P)\subset M$.
	If $P$ has a bounded density $p$ with respect to the normalized
	$d_{M}$-dimensional Hausdorff measure $\lambda_{d_{M}}$, then  $d_{{\rm vol}}=d_{M}$, and Assumption \ref{ass:dimension_exact} and \ref{ass:dimension_exact_lower} are satisfied.
	In particular, when $P$ has a bounded density $p$ with respect
	to the $d$-dimensional Lebesgue measure $\lambda_{d}$, then $d_{{\rm vol}}=d$, and Assumption \ref{ass:dimension_exact} and \ref{ass:dimension_exact_lower} are satisfied.
	
\end{proposition}

\begin{proof}[Proof for Proposition \ref{prop:dimension_manifold_bounded_density_with_exact}]

Let $\tau_{M}$ be the reach of $M$.

We first show $d_{{\rm vol}}=d_{M}$ and Assumption \ref{ass:dimension_exact}.
Since the density $p$ is bounded, for all $x\in\mathbb{X}$ and
$r>0$, the probability on the ball $\mathbb{B}_{\mathbb{R}^{d}}(x,r)$
is bounded as 
\begin{equation}
P(\mathbb{B}_{\mathbb{R}^{d}}(x,r))\leq\left\Vert p\right\Vert _{\infty}\lambda_{d_{M}}(M\cap\mathbb{B}_{\mathbb{R}^{d}}(0,r)).\label{eq:kde_prob_upper_lebesgue_density}
\end{equation}
Then for $r<\tau_{M}$, Lemma \ref{lem:dimension_ball_volume_bound}
implies $\lambda_{d_{M}}(M\cap\mathbb{B}_{\mathbb{R}^{d}}(x,r))\leq\frac{d!}{d_{M}!}2^{d}r^{d_{M}}\omega_{d}$,
and hence 
\begin{equation}
\limsup_{r\to 0}\sup_{x\in\mathbb{X}}\frac{P(\mathbb{B}_{\mathbb{R}^{d}}(x,r))}{r^{d_{M}}}\leq\left\Vert p\right\Vert _{\infty}\frac{d!}{d_{M}!}2^{d}\omega_{d}<\infty,\label{eq:dimension_manifold_bounded_density_with_exact_upper_bound}
\end{equation}
which implies 
\[
d_{{\rm vol}}\geq d_{M}.
\]
Then from Proposition \ref{prop:dimension_manifold}, 
\[
d_{{\rm vol}}=d_{M}.
\]
Now, \eqref{eq:dimension_manifold_bounded_density_with_exact_upper_bound} shows that Assumption \ref{ass:dimension_exact} is satisfied.

For Assumption \ref{ass:dimension_exact_lower}, define a density
$q:\mathbb{R}^{d}\to\mathbb{R}$ as 
\[
q(x)=\lim_{r\to0}\frac{\Gamma(\frac{d_{M}}{2}+1)}{\pi^{\frac{d_{M}}{2}}}\frac{P(\mathbb{B}_{\mathbb{R}^{d}}(x,r))}{r^{d_{M}}}.
\]
Since $M$ is a submanifold with positive reach, $P$ is $\lambda_{d_{M}}$-rectifiable.
This imply that such limit $q(x)$ exists a.e. $[\lambda_{d_{M}}]$,
and for any measurable set $A$, 
\[
P(A)=\int_{A\cap M}q(x)d\lambda_{d_{M}}(x).
\]
See, for instance, \citet[Appendix]{RinaldoW2010}, \citet[Corollary 17.9]{Mattila1995}, or \citet[Theorem 2.83]{AmbrosioFP2000}. Then from 
\[
P(M\cap\mathbb{X})=\int_{M\cap\mathbb{X}}q(x)d\lambda_{d_{M}}(x)>0,
\]
there exists $x_{0}\in M\cap\mathbb{X}$ with $q(x_{0})>0$. And hence
\[
\sup_{x\in\mathbb{X}}\liminf_{r\to0}\frac{P(\mathbb{B}_{\mathbb{R}^{d}}(x,r))}{r^{d_{M}}}\geq q(x_{0})>0,
\]
and hence Assumption \ref{ass:dimension_exact_lower} is satisfied.

\end{proof}

The proof of Proposition \ref{prop:dimension_assumption_convex} is simply checking the convexities for Assumption \ref{ass:dimension_exact} and Assumption \ref{ass:dimension_exact_lower}.

\textbf{Proposition \ref{prop:dimension_assumption_convex}.} \textit{
The set of probability distributions satisfying Assumption~\ref{ass:dimension_exact} is convex. And so is the set of probability distributions satisfying Assumption~\ref{ass:dimension_exact_lower}.
}

\begin{proof}[Proof of Proposition \ref{prop:dimension_assumption_convex}]

Suppose $P_{1}$, $P_{2}$ are two probability distributions and $\lambda\in(0,1)$.
Let $P:=\lambda P_{1}+(1-\lambda)P_{2}$. Then Proposition
\ref{prop:mixture} implies that 
\[
d_{{\rm vol}}(P)=\min\{d_{{\rm vol}}(P_{1}),d_{{\rm vol}}(P_{2})\}.
\]

Consider Assumption \ref{ass:dimension_exact} first. Suppose $P_{1}$
and $P_{2}$ satisfies Assumption \ref{ass:dimension_exact}. Then
for all $x\in\mathbb{X}$ and $r\leq1$, applying $d_{{\rm vol}}(P_{1})\leq d_{{\rm vol}}(P_{1}),d_{{\rm vol}}(P_{2})$
gives 
\begin{align*}
\frac{P(\mathbb{B}_{\mathbb{R}^{d}}(x,r))}{r^{d_{{\rm vol}}(P)}} & =\lambda\frac{P_{1}(\mathbb{B}_{\mathbb{R}^{d}}(x,r))}{r^{d_{{\rm vol}}(P)}}+(1-\lambda)\frac{P_{2}(\mathbb{B}_{\mathbb{R}^{d}}(x,r))}{r^{d_{{\rm vol}}(P)}}\\
 & \leq\lambda\frac{P_{1}(\mathbb{B}_{\mathbb{R}^{d}}(x,r))}{r^{d_{{\rm vol}}(P_{1})}}+(1-\lambda)\frac{P_{2}(\mathbb{B}_{\mathbb{R}^{d}}(x,r))}{r^{d_{{\rm vol}}(P_{2})}}.
\end{align*}
Hence, 
\begin{align*}
\limsup_{r\to0}\sup_{x\in\mathbb{X}}\frac{P(\mathbb{B}_{\mathbb{R}^{d}}(x,r))}{r^{d_{{\rm vol}}(P)}} & \leq\limsup_{r\to0}\sup_{x\in\mathbb{X}}\left\{ \lambda\frac{P_{1}(\mathbb{B}_{\mathbb{R}^{d}}(x,r))}{r^{d_{{\rm vol}}(P_{1})}}+(1-\lambda)\frac{P_{2}(\mathbb{B}_{\mathbb{R}^{d}}(x,r))}{r^{d_{{\rm vol}}(P_{2})}}\right\} \\
 & \leq\lambda\limsup_{r\to0}\sup_{x\in\mathbb{X}}\frac{P_{1}(\mathbb{B}_{\mathbb{R}^{d}}(x,r))}{r^{d_{{\rm vol}}(P_{1})}}+(1-\lambda)\limsup_{r\to0}\sup_{x\in\mathbb{X}}\frac{P_{2}(\mathbb{B}_{\mathbb{R}^{d}}(x,r))}{r^{d_{{\rm vol}}(P_{2})}}\\
 & <\infty,
\end{align*}
and Assumption \ref{ass:dimension_exact} is satisfied for $P=\lambda P_{1} + (1-\lambda) P_{2}$.

Now, consider Assumption \ref{ass:dimension_exact_lower}. Suppose
$P_{1}$ and $P_{2}$ satisfies Assumption \ref{ass:dimension_exact}, and without loss of generality,
assume $d_{{\rm vol}}(P_{1})\leq d_{{\rm vol}}(P_{2})$.
Then there exists $x_{0}\in\mathbb{X}$ such that 
\[
\liminf_{r\to0}\frac{P_{1}(\mathbb{B}_{\mathbb{R}^{d}}(x_{0},r))}{r^{d_{{\rm vol}}(P_{1})}}>0.
\]
Then $P\geq\lambda P_{1}$ and $d_{{\rm vol}}(P)=d_{{\rm vol}}(P_{1})$
give 
\[
\liminf_{r\to0}\frac{P(\mathbb{B}_{\mathbb{R}^{d}}(x_{0},r))}{r^{d_{{\rm vol}}(P)}}\geq\liminf_{r\to0}\frac{\lambda P_{1}(\mathbb{B}_{\mathbb{R}^{d}}(x_{0},r))}{r^{d_{{\rm vol}}(P_{1})}}>\infty.
\]
Hence 
\[
\sup_{x\in\mathbb{X}}\liminf_{r\to0}\frac{P(\mathbb{B}_{\mathbb{R}^{d}}(x,r))}{r^{d_{{\rm vol}}(P)}}>0,
\]
and Assumption \ref{ass:dimension_exact_lower} is satisfied for $P=\lambda P_{1} + (1-\lambda) P_{2}$.

\end{proof}

\section{Volume Dimension and Other Dimensions}
\label{sec:comparison}

In this section, we compare the volume dimension with other various
dimensions.

For a set, one commonly used dimension other than the Hausdorff dimension
is the box dimension (\citep[Section 6]{Pesin1997}, \citep[Section 3.1]{Falconer2014}).
This has various names as Kolmogorov entropy, entropy dimension, capacity
dimension, metric dimension, logarithmic density or Minkowski dimension.

\begin{definition}
	
	For any set $A\subset\mathbb{R}^{d}$ and $\delta>0$, let $N(A,\delta)$
	be the smallest number of balls of radius $\delta$ to cover $A$.
	Then the \emph{lower box dimension} of $A$ is defined as 
	\[
	d_{B}^{-}(A):=\liminf_{\delta\to0}\frac{\log N(A,\delta)}{-\log\delta},
	\]
	and the \emph{upper box dimension} of $A$ is defined as 
	\[
	d_{B}^{+}(A):=\limsup_{\delta\to0}\frac{\log N(A,\delta)}{-\log\delta}.
	\]
	
\end{definition}

The Hausdorff dimension and the lower and upper box dimensions are
related as \citep[Theorem 6.2 (2)]{Pesin1997}: 
\begin{equation}
\text{forall }A\subset\mathbb{R}^{d},\,d_{H}(A)\leq d_{B}^{-}(A)\leq d_{B}^{+}(A).\label{eq:comparison_hausdorff_box}
\end{equation}

So far the Hausdorff dimension in Section \ref{sec:definition} and
the box dimension is defined for a set. For a probability distribution,
there are two ways for natural extension. One way is to take the infimum
of the set dimensions over all sets with positive probabilities (\citep[Section 2]{Mattila2000}, \citep[Section 13.7]{Falconer2014}).
We will use this as the definition of the Hausdorff dimension and
the box dimension.

\begin{definition}
	
	Let $P$ be a probability distribution on $\mathbb{R}^{d}$. Its \emph{Hausdorff
		dimension} $d_{H}(P)$ is the infimum of the Hausdorff dimensions over a set
	with positive probability, i.e., 
	\[
	d_{H}(P):=\inf_{A:P(A)>0}d_{H}(A).
	\]
	Similarly, the \emph{lower box dimension} $d_{B}^{-}(P)$ and the
	\emph{upper box dimension} $d_{B}^{+}(P)$ is the infimum of the lower box
	dimensions and the upper box dimensions, respectively, over a set with
	positive probability, i.e. 
	\begin{align*}
	d_{B}^{-}(P) & :=\inf_{A:P(A)>0}d_{B}^{-}(A),\\
	d_{B}^{+}(P) & :=\inf_{A:P(A)>0}d_{B}^{+}(A).
	\end{align*}
	
\end{definition}

Another way is to take the infimum of the set dimensions over all
sets with probabilities $1$ \citep[Section 6]{Pesin1997}. We will
denote these dimensions as Hausdorff support dimension and the box
support dimension to differentiate from the previous dimensions.

\begin{definition}
	
	Let $P$ be a probability distribution on $\mathbb{R}^{d}$. Its \emph{Hausdorff
		support dimension} $d_{HS}(P)$ is the infimum of the Hausdorff dimensions
	over a set with probability $1$, i.e., 
	\[
	d_{HS}(P):=\inf_{A:P(A)=1}d_{H}(A).
	\]
	Similarly, the \emph{lower box dimension} $d_{BS}^{-}(P)$ and the
	\emph{upper box dimension} $d_{BS}^{+}(P)$ is the infimum of the lower box
	dimensions and the upper box dimensions, respectively, over a set with
	positive probability, i.e. 
	\begin{align*}
	d_{BS}^{-}(P) & :=\inf_{A:P(A)=1}d_{B}^{-}(A),\\
	d_{BS}^{+}(P) & :=\inf_{A:P(A)=1}d_{B}^{+}(A).
	\end{align*}
	
\end{definition}

The volume dimension, the Hausdorff dimension, and the lower and upper
box dimensions have the following relations.

\begin{proposition}
	
	Let $P$ be a probability distribution on $\mathbb{R}^{d}$ with $P(\mathbb{X})>0$.
	Then its volume dimension, Hausdorff dimension, lower and upper box
	dimension, Hausdorff support dimension, and lower and upper box support
	dimension satisfy the following inequality: 
	\[
	d_{{\rm vol}}(P)\leq d_{H}(P)\leq d_{B}^{-}(P)\leq d_{B}^{+}(P),
	\]
	and 
	\[
	d_{{\rm vol}}(P)\leq d_{HS}(P)\leq d_{BS}^{-}(P)\leq d_{BS}^{+}(P).
	\]
	
\end{proposition}

\begin{proof}
	
	Since $P({\rm supp}(P)\cap\mathbb{X})=P(\mathbb{X})>0$, $d_{{\rm vol}}(P)\leq d_{H}(P)$
	is direct from Proposition \ref{prop:dimension_manifold}. Now, combining
	this with $d_{H}(P)\leq d_{B}^{-}(P)\leq d_{B}^{+}(P)$ and $d_{HS}(P)\leq d_{BS}^{-}(P)\leq d_{BS}^{+}(P)$
	from \eqref{eq:comparison_hausdorff_box} and that $d_{H}(P)\leq d_{HS}(P)$
	gives the statement.
	
\end{proof}

Now, we introduce the $q$-dimension, which generalizes the box support
dimension \citep[Section 3.2.1]{LeeV2007}.

\begin{definition}
	
	Let $P$ be a probability distribution on $\mathbb{R}^{d}$. For $q\geq0$
	and $\delta>0$, define $C_{q}(P,\delta)$ as 
	\[
	C_{q}(P,\delta):=\int[P(\overline{\mathbb{B}_{\mathbb{R}^{d}}(x,\delta)})]^{q-1}dP(x).
	\]
	Now for $q\geq0$ and $q\neq1$, the \emph{lower $q$-dimension} of
	$P$ is 
	\[
	d_{q}^{-}(P):=\liminf_{\delta\to0}\frac{\log C_{q}(P,\delta)}{(q-1)\log\delta},
	\]
	and the \emph{upper $q$-dimension} of $P$ is 
	\[
	d_{q}^{+}(P):=\limsup_{\delta\to0}\frac{\log C_{q}(P,\delta)}{(q-1)\log\delta}.
	\]
	For $q=1$, we understand in the limit sense, i.e., $d_{1}^{-}(P)=\lim_{q\to1}d_{q}^{-}(P)$
	and $d_{1}^{+}(P)=\lim_{q\to1}d_{q}^{+}(P).$
	
\end{definition}

This $q$-dimension is a generalization of the box support dimension
in the sense that when $q=0$, the lower and upper $q$-dimensions
reduce to the lower and upper box support dimensions, respectively,
i.e. $d_{0}^{-}(P)=d_{BS}^{-}(P)$ and $d_{0}^{+}(P)=d_{BS}^{+}(P)$
\citet[Section 8]{Pesin1997}. When $q=1$, the $q$-dimension is
called the information dimension, and when $q=2$, the $q$-dimension
is called the correlation dimension.

The volume dimension and the $q$-dimension have the following relation.

\begin{proposition}
	
	Let $P$ be a probability distribution on $\mathbb{R}^{d}$ with $P(\mathbb{X})=1$.
	Then for any $q\geq0$, the volume dimension and the $q$-dimension
	has the following inequality:
	\[
	d_{{\rm vol}}(P)\leq d_{q}^{-}(P)\leq d_{q}^{+}(P).
	\]
	
\end{proposition}

\begin{proof}
	
	Since $d_{q}^{-}(P)\leq d_{q}^{+}(P)$ is obvious, we only need to
	show $d_{{\rm vol}}(P)\leq d_{q}^{-}(P)$.
	
	Fix any $\nu<d_{{\rm vol}}(P)$. Then from $P(\mathbb{X})=1$, $C_{q}(P,\delta)$
	can be expressed as taking an integration over $\mathbb{X}$. Hence
	applying \eqref{eq:dimension_volume_lower} from Lemma \ref{lem:dimension_volume_lower}
	gives
	\begin{align*}
	C_{q}(P,\delta) & =\int_{\mathbb{X}}[P(\overline{B_{\mathbb{R}^{d}}(x,\delta)})]^{q-1}dP(x)\\
	& \leq\int_{\mathbb{X}}[P(B_{\mathbb{R}^{d}}(x,2\delta))]^{q-1}dP(x)\\
	& \leq(2^{\nu}C_{\nu,P}\delta^{\nu})^{q-1}.
	\end{align*}
	And hence $d_{q}^{-}(P)$ is lower bounded as 
	\begin{align*}
	d_{q}^{-}(P) & =\liminf_{\delta\to0}\frac{\log C_{q}(P,\delta)}{(q-1)\log\delta}\geq\liminf_{\delta\to0}\frac{\log(2^{\nu}C_{\nu,P}\delta^{\nu})}{\log\delta}\\
	& =\nu+\liminf_{\delta\to0}\frac{\log(2^{\nu}C_{\nu,P})}{\log\delta}=\nu.
	\end{align*}
	Since this holds for arbitrary $\nu<d_{{\rm vol}}(P)$, we have 
	\[
	d_{{\rm vol}}(P)\leq d_{q}^{-}(P).
	\]
	
\end{proof}

We end this section by comparing the volume dimension and the Wasserstein
dimension \citep[Definition 4]{WeedB2017}.

\begin{definition}
	
	Let $P$ be a probability distribution on $\mathbb{R}^{d}$. For any
	$\delta>0$ and $\tau\in[0,1]$, let the $(\delta,\tau)$-covering
	number of $P$ be 
	\[
	N(P,\delta,\tau):=\inf\{N(A,\delta):P(A)\geq1-\tau\},
	\]
	and let the $(\delta,\tau)$-dimension be 
	\[
	d_{\delta}(P,\tau):=\frac{\log N(P,\delta,\tau)}{-\log\delta}.
	\]
	Then for a fixed $p>0$, the lower and upper Wasserstein dimensions
	are respectively,
	\begin{align*}
	d_{*}(P) & =\lim_{\tau\to0}\liminf_{\delta\to0}d_{\delta}(P,\tau)\\
	d_{p}^{*}(P) & =\inf\{s\in(2p,\infty):\,\limsup_{\delta\to0}d_{\epsilon}(P,\delta^{\frac{sp}{s-2p}})\leq s\}.
	\end{align*}
	
\end{definition}

\begin{proposition}
	
	Let $P$ be a probability distribution on $\mathbb{R}^{d}$ with $P(\mathbb{X})>0$.
	Then its volume dimension and lower and upper Wasserstein dimensions
	satisfy the following inequality:
	
	\[
	d_{{\rm vol}}(P)\leq d_{HS}(P)\leq d_{*}(P)\leq d_{p}^{*}(P).
	\]
	
\end{proposition}

\begin{proof}
	
	Since $P({\rm supp}(P)\cap\mathbb{X})=P(\mathbb{X})>0$, $d_{{\rm vol}}(P)\leq d_{H}(P)$
	is direct from Proposition \ref{prop:dimension_manifold}. The inequality
	$d_{H}(P)\leq d_{*}(P)\leq d_{p}^{*}(P)$ is from \citet[Proposition 2]{WeedB2017}.
	
\end{proof}

\section{Uniform convergence on a function class}

As we have seen in \eqref{eq:kde_supremum_function_class} in Section
\ref{sec:kde}, uniform bound on the kernel density estimator $\sup_{h\geq l_{n},x\in\mathbb{X}}\bigl|\hat{p}_{h}(x)-p_{h}(x)\bigr|$
boils down to uniformly bounding on the function class $\sup_{f\in\tilde{\mathcal{F}}_{K,[l_{n},\infty)}}\left|\frac{1}{n}\sum_{i=1}^{n}f(X_{i})-\mathbb{E}[f(X)]\right|$.
In this section, we derive a uniform convergence for a more general class of
functions. Let $\mathcal{F}$ be a class of functions from $\mathbb{R}^{d}$
to $\mathbb{R}$, and consider a random variable 
\begin{equation}
\sup_{f\in\mathcal{F}}\left|\frac{1}{n}\sum_{i=1}^{n}f(X_{i})-\mathbb{E}[f(X)]\right|.\label{eq:function_iidsum}
\end{equation}
As discussed in Section \ref{sec:kde}, we combine the Talagrand inequality
(Theorem \ref{prop:function_talagrand}) and VC type bound (Theorem
\ref{prop:function_vc_bound}) to bound \eqref{eq:function_iidsum},
which is generalizing the approach in \citet[Theorem 3.1]{Sriperumbudur2012}. 

\begin{theorem} \label{thm:function_uniform}
	
	Let $(\mathbb{R}^{d},P)$ be a probability space and let $X_{1},\ldots,X_{n}$
	be i.i.d. from $P$. Let $\mathcal{F}$ be a class of functions from
	$\mathbb{R}^{d}$ to $\mathbb{R}$ that is uniformly bounded VC-class
	with dimension $\nu$, i.e. there exists positive numbers $A$,$B$
	such that, for all $f\in\mathcal{F}$, $\left\Vert f\right\Vert _{\infty}\leq B$,
	and for every probability measure $Q$ on $\mathbb{R}^{d}$ and for
	every $\epsilon\in(0,B)$, the covering number $\mathcal{N}(\mathcal{F},L_{2}(Q),\epsilon)$
	satisfies 
	\[
	\mathcal{N}(\mathcal{F},L_{2}(Q),\epsilon)\leq\left(\frac{AB}{\epsilon}\right)^{\nu}.
	\]
	Let $\sigma>0$ with $\mathbb{E}_{P}f^{2}\leq\sigma^{2}$ for all
	$f\in\mathcal{F}$. Then there exists a universal constant $C$ not
	depending on any parameters such that $\sup_{f\in\mathcal{F}}\left|\frac{1}{n}\sum_{i=1}^{n}f(X_{i})-\mathbb{E}[f(X)]\right|$
	is upper bounded with probability at least $1-\delta$, 
	\begin{align*}
	&\sup_{f\in\mathcal{F}}\left|\frac{1}{n}\sum_{i=1}^{n}f(X_{i})-\mathbb{E}[f(X)]\right|\\
	&\leq C\left(\frac{\nu B}{n}\log\left(\frac{2AB}{\sigma}\right)+\sqrt{\frac{\nu\sigma^{2}}{n}\log\left(\frac{2AB}{\sigma}\right)}+\sqrt{\frac{\sigma^{2}\log(\frac{1}{\delta})}{n}}+\frac{B\log(\frac{1}{\delta})}{n}\right).
	\end{align*}
\end{theorem}

\begin{proof}[Proof of Theorem \ref{thm:function_uniform}]
	
	Let $\mathcal{G}:=\left\{ f-\mathbb{E}_{P}[f]:\,f\in\mathcal{F}\right\} $.
	Then it is immediate to check that for all $g\in\mathcal{G}$, 
	\begin{align}
	\mathbb{E}_{P}g & =\mathbb{E}_{P}f-\mathbb{E}_{P}f=0,\nonumber \\
	\mathbb{E}_{P}g^{2} & =\mathbb{E}_{P}(f-\mathbb{E}_{P}f)^{2}\leq\mathbb{E}_{P}f^{2}\leq\sigma^{2},\nonumber \\
	\left\Vert g\right\Vert _{\infty} & \leq\left\Vert f\right\Vert _{\infty}+\mathbb{E}_{P}f\leq2B.\label{eq:function_uniform_bound_onetwoinf}
	\end{align}
	Now, $\sup_{f\in\mathcal{F}}\left|\frac{1}{n}\sum_{i=1}^{n}f(X_{i})-\mathbb{E}[f(X)]\right|$
	is expanded as 
	\[
	\sup_{f\in\mathcal{F}}\left|\frac{1}{n}\sum_{i=1}^{n}f(X_{i})-\mathbb{E}[f(X)]\right|=\sup_{g\in\mathcal{G}}\left|\frac{1}{n}\sum_{i=1}^{n}g(X_{i})\right|.
	\]
	Hence from \eqref{eq:function_uniform_bound_onetwoinf}, applying
	Proposition \ref{prop:function_talagrand} to above gives the probabilistic
	bound on $\sup_{g\in\mathcal{G}}\left|\frac{1}{n}\sum_{i=1}^{n}g(X_{i})\right|$
	as 
	\begin{equation}
	P\left(\sup_{g\in\mathcal{G}}\left|\frac{1}{n}\sum_{i=1}^{n}g(X_{i})\right|<4\mathbb{E}_{P}\sup_{g\in\mathcal{G}}\left|\frac{1}{n}\sum_{i=1}^{n}g(X_{i})\right|+\sqrt{\frac{2\sigma^{2}\log(\frac{1}{\delta})}{n}}+\frac{2B\log(\frac{1}{\delta})}{n}\right)\geq1-\delta.\label{eq:function_uniform_talagrand}
	\end{equation}
	It thus remains to bound the term $\mathbb{E}_{P}\sup_{g\in\mathcal{G}}\left|\frac{1}{n}\sum_{i=1}^{n}g(X_{i})\right|$.
	Let $\tilde{\mathcal{F}}:=\left\{ f-a:\,f\in\mathcal{F},\,a\in[-B,B]\right\} $.
	Then $\mathcal{F}$ being a uniform VC-class with dimension $\nu$ implies that for all $\epsilon\in(0,B)$,
	\begin{align*}
	\sup_{P}\mathcal{N}\left(\tilde{\mathcal{F}},L_{2}(P),\epsilon\right) & \leq\sup_{P}\mathcal{N}\left(\mathcal{F},L_{2}(P),\frac{\epsilon}{2}\right)\sup_{P}\mathcal{N}\left([-B,B],|\cdot|,\frac{\epsilon}{2}\right)\\
	& \leq\left(\frac{2AB}{\epsilon}\right)^{\nu+1}.
	\end{align*}
	Hence from \eqref{eq:function_uniform_bound_onetwoinf}, applying
	Proposition \ref{prop:function_vc_bound} yields the upper bound for
	$\mathbb{E}_{P}\sup_{g\in\mathcal{G}}\left|\frac{1}{n}\sum_{i=1}^{n}g(X_{i})\right|$
	as 
	\begin{equation}
	\mathbb{E}_{P^{n}}\sup_{g\in\mathcal{G}}\left|\frac{1}{n}\sum_{i=1}^{n}g(X_{i})\right|\leq2C\left(\frac{2(\nu+1)B}{n}\log\left(\frac{2AB}{\sigma}\right)+\sqrt{\frac{(\nu+1)\sigma^{2}}{n}\log\left(\frac{2AB}{\sigma}\right)}\right).\label{eq:function_uniform_vc}
	\end{equation}
	Hence applying \eqref{eq:function_uniform_vc} to \eqref{eq:function_uniform_talagrand}
	yields that, $\sup_{f\in\mathcal{F}}\left|\frac{1}{n}\sum_{i=1}^{n}f(X_{i})-\mathbb{E}[f(X)]\right|$
	is upper bounded with probability at least $1-\delta$ as 
	\begin{align*}
	& \sup_{f\in\mathcal{F}}\left|\frac{1}{n}\sum_{i=1}^{n}f(X_{i})-\mathbb{E}[f(X)]\right| \\
	& \leq 4C\left(\frac{2(\nu+1)B}{n}\log\left(\frac{2AB}{\sigma}\right)+\sqrt{\frac{(\nu+1)\sigma^{2}}{n}\log\left(\frac{2AB}{\sigma}\right)}\right)\\
	& \quad+\sqrt{\frac{2\sigma^{2}\log(\frac{1}{\delta})}{n}}+\frac{2B\log(\frac{1}{\delta})}{n}\\
	& \leq16C\left(\frac{\nu B}{n}\log\left(\frac{2AB}{\sigma}\right)+\sqrt{\frac{\nu\sigma^{2}}{n}\log\left(\frac{2AB}{\sigma}\right)}+\sqrt{\frac{\sigma^{2}\log(\frac{1}{\delta})}{n}}+\frac{B\log(\frac{1}{\delta})}{n}\right).
	\end{align*}
\end{proof}

\section{Proof for Section \ref{sec:kde}}

Lemma \ref{lem:kde_kernel_bound_two} is shown by the calculation using integral by parts and change of variables.

\textbf{Lemma \ref{lem:kde_kernel_bound_two}.} \textit{Let $(\mathbb{R}^{d},P)$
be a probability space and let $X\sim P$. For any kernel $K$ satisfying Assumption~\ref{ass:integrable} with $k>0$,
the expectation of the $k$-moment of the kernel is upper bounded as 
\[
\mathbb{E}_{P}\left[\left|K\left(\frac{x-X}{h}\right)\right|^{k}\right]\leq C_{k,P,K,\epsilon}h^{d_{\mathrm{vol}}-\epsilon},
\]
for any $\epsilon\in (0,d_{{\rm vol}})$, where $C_{k,P,K,\epsilon}$ is a constant depending only on $k$, $P$, $K$, and $\epsilon$. Further, if $d_{{\rm vol}}=0$ or under Assumption~\ref{ass:dimension_exact}, $\epsilon$ can be $0$ in \eqref{eq:kde_kernel_bound}.
}

\begin{proof}[Proof of Lemma \ref{lem:kde_kernel_bound_two}]
	
    We first consider the case when $d_{{\rm vol}}=0$. Then $\mathbb{E}_{P}\left[\left|K\left(\frac{x-X}{h}\right)\right|^{k}\right]$
    is simply bounded as 
    \[
    \mathbb{E}_{P}\left[\left|K\left(\frac{x-X}{h}\right)\right|^{k}\right]\leq\left\Vert K\right\Vert _{\infty}^{k}h^{0}.
    \]
    
    Now, we consider the case when $d_{{\rm vol}}>0$. Fix $\epsilon\in(0,d_{{\rm vol}})$.
    Under Assumption \ref{ass:dimension_exact}, $\epsilon$ can be chosen
    to be $0$.
    
    Let $C_{k,K,d_{\mathrm{vol}},\epsilon}:=\int_{0}^{\infty}t^{d_{\mathrm{vol}}-\epsilon-1}\sup_{\left\Vert x\right\Vert \leq t}|K(x)|^{k}dt$,
    then it is finite from \eqref{eq:kde_condition_integral_finite}
    and $\left\Vert K\right\Vert _{\infty}<\infty$ in Assumption \ref{ass:kde_vc}
    as 
    \begin{align*}
    \int_{0}^{\infty}t^{d_{\mathrm{vol}}-\epsilon-1}\sup_{\left\Vert x\right\Vert \leq t}|K(x)|^{k}dt & \leq\int_{0}^{1}t^{d_{\mathrm{vol}}-\epsilon-1}\left\Vert K\right\Vert _{\infty}dt+\int_{1}^{\infty}t^{d_{{\rm vol}}-1}\sup_{\left\Vert x\right\Vert \leq t}|K(x)|^{k}dt\\
     & \leq\frac{\left\Vert K\right\Vert _{\infty}}{d_{{\rm vol}}-\epsilon}+\int_{0}^{\infty}t^{d_{{\rm vol}}-1}\sup_{\left\Vert x\right\Vert \leq t}|K(x)|^{k}dt<\infty.
    \end{align*}
    Fix $\eta>0$, and let $\tilde{K}_{\eta}:[0,\infty)\to\mathbb{R}$
    be a continuous and strictly decreasing function satisfying $\tilde{K}_{\eta}(t)>\sup_{\left\Vert x\right\Vert \geq t}|K(x)|^{k}$
    for all $t\geq0$ and $\int_{0}^{\infty}t^{d_{\mathrm{vol}}-\epsilon-1}(\tilde{K}_{\eta}(t)-\sup_{\left\Vert x\right\Vert \geq t}|K(x)|^{k})dt=\eta$.
    Such existence is possible since $t\mapsto\sup_{\left\Vert x\right\Vert \geq t}|K(x)|^{k}$
    is nonincreasing function, so have at most countable discontinuous
    points, and $\int_{0}^{\infty}t^{d_{\mathrm{vol}}-\epsilon-1}\sup_{\left\Vert x\right\Vert \leq t}|K(x)|^{k}dt<\infty$.
    Then it is immediate to check that 
    \begin{equation}
    |K(x)|^{k}<\tilde{K}_{\eta}(\left\Vert x\right\Vert )\text{ for all }x\in\mathbb{R}.\label{eq:kde_kernel_bound_upper_onedim}
    \end{equation}
    Then $\int_{0}^{\infty}t^{d_{\mathrm{vol}}-\epsilon-1}\tilde{K}(t)dt$
    can be expanded as 
    \begin{align}
    \int_{0}^{\infty}t^{d_{\mathrm{vol}}-\epsilon-1}\tilde{K}_{\eta}(t)dt & =\int_{0}^{\infty}t^{d_{\mathrm{vol}}-\epsilon-1}\sup_{\left\Vert x\right\Vert \leq t}|K(x)|^{k}dt+\int_{0}^{\infty}t^{d_{\mathrm{vol}}-\epsilon-1}(\tilde{K}_{\eta}(t)-\sup_{\left\Vert x\right\Vert \geq t}|K(x)|^{k})dt\nonumber \\
     & =C_{k,K,d_{\mathrm{vol}},\epsilon}+\eta<\infty.\label{eq:kde_kernel_bound_upper_integral}
    \end{align}
    Now since $\tilde{K}_{\eta}$ is continuous and strictly decreasing,
    change of variables $t=\tilde{K}_{\eta}(u)$ is applicable, and then
    $\mathbb{E}_{P}\left[\left|K\left(\frac{x-X}{h}\right)\right|^{k}\right]$
    can be expanded as 
    \begin{align*}
    \mathbb{E}_{P}\left[\left|K\left(\frac{x-X}{h}\right)\right|^{k}\right] & =\int_{0}^{\infty}P\left(\left|K\left(\frac{x-X}{h}\right)\right|^{k}>t\right)dt\\
     & =\int_{\infty}^{0}P\left(\left|K\left(\frac{x-X}{h}\right)\right|^{k}>\tilde{K}_{\eta}(u)\right)d\tilde{K}_{\eta}(u).
    \end{align*}
    Now, from \eqref{eq:kde_kernel_bound_upper_onedim} and $\tilde{K}_{\eta}$
    being a strictly decreasing, we can upper bound $\mathbb{E}_{P}\left[\left|K\left(\frac{x-X}{h}\right)\right|^{k}\right]$
    as 
    \begin{align*}
    \mathbb{E}_{P}\left[\left|K\left(\frac{x-X}{h}\right)\right|^{k}\right] & \leq\int_{\infty}^{0}P\left(\tilde{K}_{\eta}\left(\frac{\left\Vert x-X\right\Vert }{h}\right)>\tilde{K}_{\eta}(u)\right)d\tilde{K}_{\eta}(u)\\
     & =\int_{\infty}^{0}P\left(\frac{\left\Vert x-X\right\Vert }{h}<u\right)d\tilde{K}_{\eta}(u)\\
     & =\int_{\infty}^{0}P\left(\mathbb{B}_{\mathbb{R}^{d}}(x,hu)\right)d\tilde{K}_{\eta}(u).
    \end{align*}
    Now, from Lemma \ref{lem:dimension_volume_lower} (and \eqref{eq:dimension_exact}
    for Assumption \ref{ass:dimension_exact} case), there exists $C_{d_{{\rm vol}}-\epsilon,P}<\infty$
    with $P\left(\mathbb{B}_{\mathbb{R}^{d}}(x,r)\right)\leq C_{d_{{\rm vol}}-\epsilon,P}r^{d_{\mathrm{vol}}-\epsilon}$
    for all $x\in\mathbb{X}$ and $r>0$. Then $\mathbb{E}_{P}\left[\left|K\left(\frac{x-X}{h}\right)\right|^{k}\right]$
    is further upper bounded as 
    \begin{align}
    \mathbb{E}_{P}\left[\left|K\left(\frac{x-X}{h}\right)\right|^{k}\right] & \leq\int_{\infty}^{0}C_{d_{{\rm vol}}-\epsilon,P}(hu)^{d_{\mathrm{vol}}-\epsilon}d\tilde{K}(u)\nonumber \\
     & =C_{d_{{\rm vol}}-\epsilon,P}h^{d_{\mathrm{vol}}-\epsilon}\int_{\infty}^{0}u^{d_{\mathrm{vol}}-\epsilon}d\tilde{K}(u).\label{eq:kde_kernel_bound_upperbound_stieltjes}
    \end{align}
    Now, $\int_{\infty}^{0}u^{d_{\mathrm{vol}}-\epsilon}d\tilde{K}(u)$
    can be computed using integration by part. Note first that $\int_{0}^{\infty}t^{d_{\mathrm{vol}}-\epsilon-1}\tilde{K}(t)dt<\infty$
    implies 
    \[
    \lim_{t\to\infty}t^{d_{\mathrm{vol}}-\epsilon}\tilde{K}(t)=0.
    \]
    To see this, note that $t^{d_{\mathrm{vol}}-\epsilon}\tilde{K}(t)$
    is expanded as 
    \[
    t^{d_{\mathrm{vol}}-\epsilon}\tilde{K}(t)=\int_{0}^{t}u^{d_{\mathrm{vol}}-\epsilon}d\tilde{K}(u)+\int_{0}^{t}(d_{\mathrm{vol}}-\epsilon)u^{d_{\mathrm{vol}}-\epsilon-1}\tilde{K}(u)du,
    \]
    then $\int_{0}^{\infty}(d_{\mathrm{vol}}-\epsilon)u^{d_{\mathrm{vol}}-\epsilon-1}\tilde{K}(u)du<\infty$
    and $\int_{0}^{t}u^{d_{\mathrm{vol}}-\epsilon}d\tilde{K}(u)$ being
    monotone function of $t$ imply that $\lim_{t\to\infty}t^{d_{\mathrm{vol}}-\epsilon}\tilde{K}(t)$
    exists. Now, suppose $\lim_{t\to\infty}t^{d_{\mathrm{vol}}-\epsilon}\tilde{K}(t)=a>0$,
    then we can choose $t_{0}>0$ such that $t^{d_{\mathrm{vol}}-\epsilon}\tilde{K}(t)>\frac{a}{2}$
    for all $t\geq t_{0}$, and then 
    \[
    \infty>\int_{0}^{\infty}t^{d_{\mathrm{vol}}-\epsilon-1}\tilde{K}(t)dt\geq\int_{t_{0}}^{\infty}t^{d_{\mathrm{vol}}-\epsilon-1}\tilde{K}(t)dt\geq\frac{a}{2}\int_{t_{0}}^{\infty}t^{-1}dt=\infty,
    \]
    which is a contradiction. Hence $\lim_{t\to\infty}t^{d_{\mathrm{vol}}-\epsilon}\tilde{K}(t)=0$.
    Now, applying integration by part to $\int_{\infty}^{0}u^{d_{\mathrm{vol}}-\epsilon}d\tilde{K}(u)$
    with $d_{{\rm vol}}-\epsilon>0$ gives 
    \begin{align}
    \int_{\infty}^{0}u^{d_{\mathrm{vol}}-\epsilon}d\tilde{K}(u) & =\left[u^{d_{\mathrm{vol}}-\epsilon}\tilde{K}(u)\right]_{\infty}^{0}-\int_{\infty}^{0}(d_{\mathrm{vol}}-\epsilon)u^{d_{\mathrm{vol}}-\epsilon-1}\tilde{K}(u)du\nonumber \\
    & =\int_{0}^{\infty}(d_{\mathrm{vol}}-\epsilon)u^{d_{\mathrm{vol}}-\epsilon-1}\tilde{K}(u)du.\label{eq:kde_kernel_bound_integral_part}
    \end{align}
    Then applying \eqref{eq:kde_kernel_bound_upper_integral} and \eqref{eq:kde_kernel_bound_integral_part}
    to \eqref{eq:kde_kernel_bound_upperbound_stieltjes} gives an upper
    bound for $\mathbb{E}_{P}\left[\left|K\left(\frac{x-X}{h}\right)\right|^{k}\right]$
    as 
    \begin{equation}
    \mathbb{E}_{P}\left[\left|K\left(\frac{x-X}{h}\right)\right|^{k}\right]\leq C_{d_{{\rm vol}}-\epsilon,P}(d_{\mathrm{vol}}-\epsilon)h^{d_{\mathrm{vol}}-\epsilon}(C_{k,K,d_{\mathrm{vol}},\epsilon}+\eta).\label{eq:kde_kernel_bound_upperbound_epsilon}
    \end{equation}
    And then note that RHS of \eqref{eq:kde_kernel_bound_upperbound_epsilon}
    holds for any $\eta>0$, and hence $\mathbb{E}_{P}\left[\left|K\left(\frac{x-X}{h}\right)\right|^{k}\right]$
    is further upper bounded as 
    \begin{align*}
    \mathbb{E}_{P}\left[\left|K\left(\frac{x-X}{h}\right)\right|^{k}\right] & \leq\inf_{\eta>0}\left\{ C_{d_{{\rm vol}}-\epsilon,P}(d_{\mathrm{vol}}-\epsilon)h^{d_{\mathrm{vol}}-\epsilon}(C_{k,K,d_{\mathrm{vol}},\epsilon}+\eta)\right\} \\
     & =C_{d_{{\rm vol}}-\epsilon,P}(d_{\mathrm{vol}}-\epsilon)C_{k,K,d_{\mathrm{vol}},\epsilon}h^{d_{\mathrm{vol}}-\epsilon}\\
     & =C_{k,P,K,\epsilon}h^{d_{{\rm vol}}-\epsilon},
    \end{align*}
    where $C_{k,P,K,\epsilon}=C_{d_{{\rm vol}}-\epsilon,P}(d_{\mathrm{vol}}-\epsilon)C_{k,K,d_{\mathrm{vol}},\epsilon}$.

\end{proof}

\subsection{Proof for Section \ref{subsec:kde_band_ray}}

Theorem \ref{thm:kde_band_ray_uniform} follows from applying Theorem \ref{thm:function_uniform}.

\textbf{Theorem \ref{thm:kde_band_ray_uniform}.} \textit{
	Let $P$ be a probability distribution and let $K$ be a kernel function
	satisfying Assumption~\ref{ass:integrable} and~\ref{ass:kde_vc}.
	Then, with probability at least $1-\delta$, 
	\[
	\sup_{h\geq l_{n},x\in\mathbb{X}}\left|\hat{p}_{h}(x)-p_{h}(x)\right|
	\leq C\left(\frac{\left(\log\left(1 / l_{n}\right)\right)_{+}}{nl_{n}^{d}}+\sqrt{\frac{\left(\log\left(1 / l_{n}\right)\right)_{+}}{nl_{n}^{2d-d_{\mathrm{vol}}+\epsilon}}}
	+\sqrt{\frac{\log\left(2 / \delta\right)}{nl_{n}^{2d-d_{\mathrm{vol}}+\epsilon}}}+\frac{\log\left(2  / \delta\right)}{nl_{n}^{d}}\right),
	\]
	for any $\epsilon\in (0,d_{{\rm vol}})$, where $C$ is a constant depending only on $A$, $\left\Vert K\right\Vert _{\infty}$, $d$,
	$\nu$, $d_{\mathrm{vol}}$, $C_{k=2,P,K,\epsilon}$, $\epsilon$. Further, if $d_{{\rm vol}}=0$ or under Assumption~\ref{ass:dimension_exact}, $\epsilon$ can be $0$ in \eqref{eq:kde_band_ray_uniform_bound}.
}

\begin{proof}[Proof of Theorem \ref{thm:kde_band_ray_uniform}]
	
	For $x\in\mathbb{X}$ and $h\geq l_{n}$, let $K_{x,h}:\mathbb{R}^{d}\to\mathbb{R}$
	be $K_{x,h}(\cdot)=K\left(\frac{x-\cdot}{h}\right)$, and let
	$\tilde{\mathcal{F}}_{K,[l_{n},\infty)}:=\left\{ \frac{1}{h^{d}}K_{x,h}:\,x\in\mathbb{X},h\geq l_{n}\right\} $
	be a class of normalized kernel functions centered on $\mathbb{X}$
	and bandwidth in $[l_{n},\infty)$. Note that $\hat{p}_{h}(x)-p_{h}(x)$
	can be expanded as 
	\[
	\hat{p}_{h}(x)-p_{h}(x)=\frac{1}{nh^{d}}\sum_{i=1}^{n}K\left(\frac{x-X_{i}}{h}\right)-\mathbb{E}_{P}\left[\frac{1}{h^{d}}K\left(\frac{x-X_{i}}{h}\right)\right]=\frac{1}{n}\sum_{i=1}^{n}\frac{1}{h^{d}}K_{x,h}(X_{i})-\mathbb{E}_{P}\left[\frac{1}{h^{d}}K_{x,h}\right].
	\]
	Hence $\sup_{h\geq l_{n},x\in\mathbb{X}}\left|\hat{p}_{h}(x)-p_{h}(x)\right|$ can be
	expanded as 
	\begin{equation}
	\sup_{h\geq l_{n},x\in\mathbb{X}}\left|\hat{p}_{h}(x)-p_{h}(x)\right| =\sup_{f\in\tilde{\mathcal{F}}_{K,[l_{n},\infty)}}\left|\frac{1}{n}\sum_{i=1}^{n}f(X_{i})-\mathbb{E}_{P}\left[f(X)\right]\right|.\label{eq:kde_uniform_expansion}
	\end{equation}
	Now, it is immediate to check that 
	\begin{equation}
	\left\Vert f\right\Vert _{\infty}\leq l_{n}^{-d}\left\Vert K\right\Vert _{\infty}.\label{eq:kde_uniform_bound_inf}
	\end{equation}
	For bounding the VC dimension of $\tilde{\mathcal{F}}_{K,[l_{n},\infty)}$,
	consider $\mathcal{F}_{K,[l_{n},\infty)}:=\left\{ K_{x,h}:\,x\in\mathbb{X},h\geq l_{n}\right\} $
	be a class of unnormalized kernel functions centered on $\mathbb{X}$
	and bandwidth in $[l_{n},\infty)$. Fix $\eta<l_{n}^{-d}\left\Vert K\right\Vert _{\infty}$
	and a   probability measure $Q$ on $\mathbb{R}^{d}$. Suppose $\left[l_{n},\left(\frac{\eta}{2\left\Vert K\right\Vert _{\infty}}\right)^{-1/d}\right]$
	is covered by balls $\Bigl\{ \Bigl(h_{i}-\frac{l_{n}^{d+1}\eta}{2d\left\Vert K\right\Vert _{\infty}},h_{i}+\frac{l_{n}^{d+1}\eta}{2d\left\Vert K\right\Vert _{\infty}}\Bigr):\,1\leq i\leq N_{1}\Bigr\} $
	and $(\mathcal{F}_{K,[l_{n},\infty)},L_{2}(Q))$ is covered by balls
	$\left\{ \mathbb{B}_{L_{2}(Q)}\left(f_{j},\frac{l_{n}^{d}\eta}{2}\right):\,1\leq j\leq N_{2}\right\} $,
	and let $f_{i,j}:=h_{i}^{-d}f_{j}$ for $1\leq i\leq N_{1}$ and $1\leq j\leq N_{2}$.
	Also, choose $h_{0}>\left(\frac{\eta}{2\left\Vert K\right\Vert _{\infty}}\right)^{-1/d}$,
	$x_{0}\in\mathbb{X}$, and let $f_{0}=\frac{1}{h_{0}^{d}}K_{x_{0},h_{0}}$.
	We will show that 
	\begin{equation}
	\left\{ \mathbb{B}_{L_{2}(Q)}\left(f_{i,j},\eta\right):\,1\leq i\leq N_{1},\,1\leq j\leq N_{2}\right\} \cup\left\{ \mathbb{B}_{L_{2}(Q)}\left(f_{0},\eta\right)\right\} \text{ covers }\tilde{\mathcal{F}}_{K,[l_{n},\infty)}.\label{eq:kde_uniform_covering}
	\end{equation}
	For the first case when $h\leq\left(\frac{\eta}{\left\Vert K\right\Vert _{\infty}}\right)^{-1/d}$,
	find $h_{i}$ and $f_{j}$ with $h\in\left(h_{i}-\frac{l_{n}^{d+1}\eta}{2d\left\Vert K\right\Vert _{\infty}},h_{i}+\frac{l_{n}^{d+1}\eta}{2d\left\Vert K\right\Vert _{\infty}}\right)$
	and $K_{x,h}\in\mathbb{B}_{L_{2}(Q)}\left(f_{j},\frac{l_{n}^{d}\eta}{2}\right)$.
	Then the distance between $\frac{1}{h^{d}}K_{x,h}$ and $\frac{1}{h_{i}^{d}}f_{j}$
	is upper bounded as 
	\begin{equation}
	\left\Vert \frac{1}{h^{d}}K_{x,h}-\frac{1}{h_{i}^{d}}f_{j}\right\Vert _{L_{2}(Q)}\leq\left\Vert \frac{1}{h^{d}}K_{x,h}-\frac{1}{h_{i}^{d}}K_{x,h}\right\Vert _{L_{2}(Q)}+\left\Vert \frac{1}{h_{i}^{d}}K_{x,h}-\frac{1}{h_{i}^{d}}f_{j}\right\Vert _{L_{2}(Q)}.\label{eq:kde_uniform_covering_decomposition}
	\end{equation}
	Now, the first term of \eqref{eq:kde_uniform_covering_decomposition}
	is upper bounded as 
	\begin{align}
	\left\Vert \frac{1}{h^{d}}K_{x,h}-\frac{1}{h_{i}^{d}}K_{x,h}\right\Vert _{L_{2}(Q)} & =\left|\frac{1}{h^{d}}-\frac{1}{h_{i}^{d}}\right|\left\Vert K_{x,h}\right\Vert _{L_{2}(Q)}\nonumber \\
	& =\left|h_{i}-h\right|\sum_{k=0}^{d-1}h_{i}^{k-d}h^{-1-k}\left\Vert K_{x,h}\right\Vert _{L_{2}(Q)}\nonumber \\
	& \leq\left|h_{i}-h\right|dl_{n}^{-d-1}\left\Vert K\right\Vert _{\infty}<\frac{\eta}{2}.\label{eq:kde_uniform_covering_decomposition_first}
	\end{align}
	Also, the second term of \eqref{eq:kde_uniform_covering_decomposition}
	is upper bounded as 
	\begin{align}
	\left\Vert \frac{1}{h_{i}^{d}}K_{x,h}-\frac{1}{h_{i}^{d}}f_{j}\right\Vert _{L_{2}(Q)} & =\frac{1}{h_{i}^{d}}\left\Vert K_{x,h}-f_{j}\right\Vert _{L_{2}(Q)}\nonumber \\
	& \leq l_{n}^{-d}\left\Vert K_{x,h}-f_{j}\right\Vert _{L_{2}(Q)}<\frac{\eta}{2}.\label{eq:kde_uniform_covering_decomposition_second}
	\end{align}
	Hence applying \eqref{eq:kde_uniform_covering_decomposition_first}
	and \eqref{eq:kde_uniform_covering_decomposition_second} to \eqref{eq:kde_uniform_covering_decomposition}
	gives 
	\[
	\left\Vert \frac{1}{h^{d}}K_{x,h}-\frac{1}{h_{i}^{d}}f_{j}\right\Vert _{L_{2}(Q)}<\eta.
	\]
	For the second case when $h>\left(\frac{\eta}{2\left\Vert K\right\Vert _{\infty}}\right)^{-1/d}$,
	$\left\Vert \frac{1}{h^{d}}K_{x,h}\right\Vert _{L_{2}(Q)}\leq\left\Vert \frac{1}{h^{d}}K_{x,h}\right\Vert _{\infty}<\frac{\eta}{2}$
	holds, and hence 
	\[
	\left\Vert \frac{1}{h^{d}}K_{x,h}-f_{0}\right\Vert _{L_{2}(Q)}\leq\left\Vert \frac{1}{h^{d}}K_{x,h}\right\Vert _{L_{2}(Q)}+\left\Vert f_{0}\right\Vert _{L_{2}(Q)}<\eta.
	\]
	Therefore, \eqref{eq:kde_uniform_covering} is shown. Hence combined
	with Assumption \ref{ass:kde_vc} gives that for every probability
	measure $Q$ on $\mathbb{R}^{d}$ and for every $\eta\in(0,h^{-d}\left\Vert K\right\Vert _{\infty})$,
	the covering number $\mathcal{N}(\tilde{\mathcal{F}}_{K,[l_{n},\infty)},L_{2}(Q),\eta)$
	is upper bounded as 
	\begin{align}
	& \sup_{Q}\mathcal{N}(\tilde{\mathcal{F}}_{K,[l_{n},\infty)},L_{2}(Q),\eta) \nonumber\\
	& \leq\mathcal{N}\left(\left[l_{n},\left(\frac{\eta}{2\left\Vert K\right\Vert _{\infty}}\right)^{-1/d}\right],|\cdot|,\frac{l_{n}^{d+1}\eta}{2d\left\Vert K\right\Vert _{\infty}}\right)\sup_{Q}\mathcal{N}\left(\mathcal{F}_{K,[l_{n},\infty)},L_{2}(Q),\frac{l_{n}^{d}\eta}{2}\right)+1\nonumber \\
	& \leq\frac{2d\left\Vert K\right\Vert _{\infty}}{l_{n}^{d+1}\eta}\left(\frac{2\left\Vert K\right\Vert _{\infty}}{\eta}\right)^{1/d}\left(\frac{2A\left\Vert K\right\Vert _{\infty}}{l_{n}^{d}\eta}\right)^{\nu}+1\nonumber \\
	& \leq\left(\frac{2Ad\left\Vert K\right\Vert _{\infty}}{l_{n}^{d}\eta}\right)^{\nu+2}.\label{eq:kde_uniform_bound_vc}
	\end{align}
	Also, Lemma \ref{lem:kde_kernel_bound_two} implies that under Assumption
	\ref{ass:integrable}, for any $\epsilon \in (0, d_{\rm{vol}})$ (and $\epsilon$ can be $0$ if $d_{\rm{vol}}=0$ or under Assumption \ref{ass:dimension_exact}), 
	\begin{equation}
	\mathbb{E}_{P}\left[\left(\frac{1}{h^{d}}K_{x,h}\right)^{2}\right]\leq C_{k=2,P,K,\epsilon}l_{n}^{-2d+d_{\mathrm{vol}}-\epsilon}.\label{eq:kde_uniform_bound_two}
	\end{equation}
	Hence from \eqref{eq:kde_uniform_bound_inf}, \eqref{eq:kde_uniform_bound_vc},
	and \eqref{eq:kde_uniform_bound_two}, applying Theorem \ref{thm:function_uniform}
	to \eqref{eq:kde_uniform_expansion} gives that $\sup_{h\geq l_{n},x\in\mathbb{X}}\left|\hat{p}_{h}(x)-p_{h}(x)\right|$
	is upper bounded with probability at least $1-\delta$ as 
	\begin{align*}
	& \sup_{h\geq l_{n},x\in\mathbb{X}}\left|\hat{p}_{h}(x)-p_{h}(x)\right|\\
	& \leq C\left(\frac{2(\nu+2)\left\Vert K\right\Vert _{\infty}\log\left(\frac{2Ad\left\Vert K\right\Vert _{\infty}}{\sqrt{C_{k=2,P,K,\epsilon}}l_{n}^{(d_{\mathrm{vol}}-\epsilon)/2}}\right)}{nl_{n}^{d}}+\sqrt{\frac{2(\nu+2)C_{k=2,P,K,\epsilon}\log\left(\frac{2Ad\left\Vert K\right\Vert _{\infty}}{\sqrt{C_{k=2,P,K,\epsilon}}l_{n}^{(d_{\mathrm{vol}}-\epsilon)/2}}\right)}{nl_{n}^{2d-d_{\mathrm{vol}}+\epsilon}}}\right.\\
	& \qquad\qquad+\left.\sqrt{\frac{C_{k=2,P,K,\epsilon}\log(\frac{1}{\delta})}{nl_{n}^{2d-d_{\mathrm{vol}}+\epsilon}}}+\frac{\left\Vert K\right\Vert _{\infty}\log(\frac{1}{\delta})}{nl_{n}^{d}}\right)\\
	& \leq C_{A,\left\Vert K\right\Vert _{\infty},d,\nu,d_{\mathrm{vol}},C_{k=2,P,K,\epsilon}}\left(\frac{\left(\log\left(\frac{1}{l_{n}}\right)\right)_{+}}{nl_{n}^{d}}+\sqrt{\frac{\left(\log\left(\frac{1}{l_{n}}\right)\right)_{+}}{nl_{n}^{2d-d_{\mathrm{vol}}+\epsilon}}}+\sqrt{\frac{\log\left(\frac{2}{\delta}\right)}{nl_{n}^{2d-d_{\mathrm{vol}}+\epsilon}}}+\frac{\log\left(\frac{2}{\delta}\right)}{nl_{n}^{d}}\right),
	\end{align*}
	where $C_{A,\left\Vert K\right\Vert _{\infty},d,\nu,d_{\mathrm{vol}},C_{k=2,P,K,\epsilon},\epsilon}$
	depends only on $A$, $\left\Vert K\right\Vert _{\infty}$, $d$,
	$\nu$, $d_{\mathrm{vol}}$, $C_{k=2,P,K,\epsilon}$, $\epsilon$.
	
\end{proof}

Then Corollary \ref{cor:kde_band_ray_uniform_probconv} is just simplifying the result in Theorem \ref{thm:kde_band_ray_uniform}. 

\textbf{Corollary \ref{cor:kde_band_ray_uniform_probconv}.} \textit{
Let $P$ be a probability distribution and let $K$ be a kernel function
satisfying Assumption~\ref{ass:integrable} and~\ref{ass:kde_vc}.
Fix $\epsilon\in(0,d_{\rm{vol}})$. Further, if $d_{{\rm vol}}=0$ or under Assumption~\ref{ass:dimension_exact}, $\epsilon$ can be $0$.
Suppose 
\[
\limsup_{n}\frac{\left(\log\left(1 / \ell_{n}\right)\right)_{+}+\log\left(2 / \delta\right)}{n\ell_{n}^{d_{\mathrm{vol}}-\epsilon}}<\infty.
\]
Then, with probability at least $1-\delta$, 
\[
\sup_{h\geq l_{n},x\in\mathbb{X}}\left|\hat{p}_{h}(x)-p_{h}(x)\right|
\leq C'\sqrt{\frac{(\log(\frac{1}{l_{n}}))_{+}+\log(\frac{2}{\delta})}{nl_{n}^{2d-d_{\mathrm{vol}}+\epsilon}}},
\]
where $C'$
depending only on $A$, $\left\Vert K\right\Vert _{\infty}$, $d$,
$\nu$, $d_{\mathrm{vol}}$, $C_{k=2,P,K,\epsilon}$, $\epsilon$.
}

\begin{proof}[Proof of Corollary \ref{cor:kde_band_ray_uniform_probconv}]
	
	From \eqref{eq:kde_band_ray_uniform_bound} in Theorem \ref{thm:kde_band_ray_uniform},
	$\sup_{h\geq l_{n},x\in\mathbb{X}}\left|\hat{p}_{h}(x)-p_{h}(x)\right|$
	is upper bounded with probability at least $1-\delta$ as 
	\begin{align*}
	& \sup_{h\geq l_{n},x\in\mathbb{X}}\left|\hat{p}_{h}(x)-p_{h}(x)\right|\\
	& \leq C_{A,\left\Vert K\right\Vert _{\infty},d,\nu,d_{\mathrm{vol}},C_{k=2,P,K,\epsilon},\epsilon}\left(\frac{\left(\log\left(\frac{1}{l_{n}}\right)\right)_{+}}{nl_{n}^{d}}+\sqrt{\frac{\left(\log\left(\frac{1}{l_{n}}\right)\right)_{+}}{nl_{n}^{2d-d_{\mathrm{vol}}+\epsilon}}}+\sqrt{\frac{\log\left(\frac{2}{\delta}\right)}{nl_{n}^{2d-d_{\mathrm{vol}+\epsilon}+\epsilon}}}+\frac{\log\left(\frac{2}{\delta}\right)}{nl_{n}^{d}}\right)\\
	& =C_{A,\left\Vert K\right\Vert _{\infty},d,\nu,d_{\mathrm{vol}},C_{k=2,P,K,\epsilon},\epsilon}\\
	& \quad \times\left(\sqrt{\frac{\left(\log\left(\frac{1}{l_{n}}\right)\right)_{+}}{nl_{n}^{2d-d_{\mathrm{vol}}+\epsilon}}}\left(\sqrt{\frac{\left(\log\left(\frac{1}{l_{n}}\right)\right)_{+}}{nl_{n}^{d_{\mathrm{vol}}-\epsilon}}}+1\right)+\sqrt{\frac{\log\left(\frac{2}{\delta}\right)}{nl_{n}^{2d-d_{\mathrm{vol}}+\epsilon}}}\left(\sqrt{\frac{\log\left(\frac{2}{\delta}\right)}{nl_{n}^{d_{\mathrm{vol}}-\epsilon}}}+1\right)\right).
	\end{align*}
	
	Then from $\lim\sup_{n}\frac{\left(\log\left(\frac{1}{l_{n}}\right)\right)_{+}+\log\left(\frac{2}{\delta}\right)}{nl_{n}^{d_{\mathrm{vol}}-\epsilon}}<\infty$,
	there exists some constant $C'$ with $\left(\log\left(\frac{1}{l_{n}}\right)\right)_{+}+\log\left(\frac{2}{\delta}\right)\leq C'nl_{n}^{d_{\mathrm{vol}}+\epsilon}$.
	And hence $\sup_{h\geq l_{n},x\in\mathbb{X}}\left|\hat{p}_{h}(x)-p_{h}(x)\right|$
	is upper bounded with probability $1-\delta$ as 
	\begin{align*}
	& \sup_{h\geq l_{n},x\in\mathbb{X}}\left|\hat{p}_{h}(x)-p_{h}(x)\right|\\
	& \leq C_{A,\left\Vert K\right\Vert _{\infty},d,\nu,d_{\mathrm{vol}},C_{k=2,P,K,\epsilon},\epsilon}\left(\sqrt{\frac{\left(\log\left(\frac{1}{l_{n}}\right)\right)_{+}}{nl_{n}^{2d-d_{\mathrm{vol}}+\epsilon}}}\left(\sqrt{C'}+1\right)+\sqrt{\frac{\log\left(\frac{1}{\delta}\right)}{nl_{n}^{2d-d_{\mathrm{vol}}+\epsilon}}}\left(\sqrt{C'}+1\right)\right)\\
	& \leq C'_{A,\left\Vert K\right\Vert _{\infty},d,\nu,d_{\mathrm{vol}},C_{k=2,P,K,\epsilon},\epsilon}\sqrt{\frac{\left(\log\left(\frac{1}{l_{n}}\right)\right)_{+}+\log\left(\frac{1}{\delta}\right)}{nl_{n}^{2d-d_{\mathrm{vol}}+\epsilon}}},
	\end{align*}
	where $C'_{A,\left\Vert K\right\Vert _{\infty},d,\nu,d_{\mathrm{vol}},C_{k=2,P,K,\epsilon},\epsilon}$
	depending only on $A$, $\left\Vert K\right\Vert _{\infty}$, $d$,
	$\nu$, $d_{\mathrm{vol}}$, $C_{k=2,P,K,\epsilon}$, $\epsilon$.
	
\end{proof}

\subsection{Proof for Section \ref{subsec:kde_band_one}}

Lemma \ref{lem:kde_vc} is by covering $\mathbb{X}$ and then using the Lipschitz property of the kernel function $K$.

\textbf{Lemma \ref{lem:kde_vc}.} \textit{
	Suppose there exists $R>0$ with $\mathbb{X}\subset\mathbb{B}_{\mathbb{R}^{d}}(0,R)$. Let the kernel $K$ is $M_K$-Lipschitz continuous. Then for all $\eta\in\left(0,\left\Vert K\right\Vert _{\infty}\right)$,
	the supremum of the $\eta$-covering number $\mathcal{N}(\mathcal{F}_{K,h},L_{2}(Q),\eta)$
	over all measure $Q$ is upper bounded as 
	\[
	\sup_{Q}\mathcal{N}(\mathcal{F}_{K,h},L_{2}(Q),\eta)\leq\left(\frac{2RM_{K}h^{-1}+\left\Vert K\right\Vert _{\infty}}{\eta}\right)^{d}.
	\]
}

\begin{proof}[Proof of Lemma \ref{lem:kde_vc}]
	
	For fixed $\eta>0$, let $x_{1},\ldots,x_{M}$ be the maximal
	$\eta$-covering of $\mathbb{B}_{\mathbb{R}^{d}}(0,R)$, with
	$M=\mathcal{M}\left(\mathbb{B}_{\mathbb{R}^{d}}(0,R),\left\Vert \cdot\right\Vert _{2},\eta\right)$
	being the packing number of $\mathbb{B}_{\mathbb{R}^{d}}(0,R)$. Then
	$\mathbb{B}_{\mathbb{R}^{d}}(x_{i},\eta)$ and $\mathbb{B}_{\mathbb{R}^{d}}(x_{j},\eta)$
	do not intersect for any $i,j$ and $\bigcup_{i=1}^{M}\mathbb{B}_{\mathbb{R}^{d}}(x_{i},\eta)\subset\mathbb{B}_{\mathbb{R}^{d}}(x_{i},R+\eta)$,
	and hence 
	\begin{equation}
	\sum_{i=1}^{M}\lambda_{d}\left(\mathbb{B}_{\mathbb{R}^{d}}(x_{i},\eta)\right)\leq\lambda_{d}\left(\mathbb{B}_{\mathbb{R}^{d}}(x_{i},R+\eta)\right).\label{eq:kde_vc_support_probability}
	\end{equation}
	Then $\lambda_{d}\left(\mathbb{B}_{\mathbb{R}^{d}}(x,r)\right)=r^{d}\lambda_{d}\left(\mathbb{B}_{\mathbb{R}^{d}}(0,1)\right)$
	gives the upper bound on $\mathcal{M}(\mathbb{B}_{\mathbb{R}^{d}}(0,R),\left\Vert \cdot\right\Vert_{2} ,\eta)$
	as 
	\[
	\mathcal{M}\left(\mathbb{B}_{\mathbb{R}^{d}}(0,R),\left\Vert \cdot\right\Vert _{2},\eta\right)\leq\left(1+\frac{R}{\eta}\right)^{d}.
	\]
	Then $\mathbb{X}\subset\mathbb{B}_{\mathbb{R}^{d}}(0,R)$ and the
	relationship between covering number and packing number gives the
	upper bound on the covering number $\mathcal{N}\left(\mathbb{X},\left\Vert \cdot\right\Vert _{2},\eta\right)$
	as 
	\begin{equation}
	\mathcal{N}\left(\mathbb{X},\left\Vert \cdot\right\Vert _{2},\eta\right)\leq\mathcal{N}\left(\mathbb{B}_{\mathbb{R}^{d}}(0,R),\left\Vert \cdot\right\Vert _{2},\eta\right)\leq\mathcal{M}\left(\mathbb{B}_{\mathbb{R}^{d}}(0,R),\left\Vert \cdot\right\Vert _{2},\frac{\eta}{2}\right)\leq\left(1+\frac{2R}{\eta}\right)^{d}.\label{eq:kde_vc_support_covering}
	\end{equation}
	Now, note that for all $x,y\in\mathbb{X}$ and for all $z\in\mathbb{R}^{d}$,
	$\left|K_{x,h}(z)-K_{y,h}(z)\right|$ is upper bounded as 
	\[
	\left|K_{x,h}(z)-K_{y,h}(z)\right|=\left|K\left(\frac{x-z}{h}\right)-K\left(\frac{y-z}{h}\right)\right|\leq\frac{M_{K}}{h}\left\Vert (x-z)-(y-z)\right\Vert _{2}=\frac{M_{K}}{h}\left\Vert x-y\right\Vert _{2}.
	\]
	Hence for any measure $Q$ on $\mathbb{R}^{d}$, $\left\Vert K_{x,h}-K_{y,h}\right\Vert _{L_{2}(Q)}$
	is upper bounded as 
	\[
	\left\Vert K_{x,h}-K_{y,h}\right\Vert _{L_{2}(Q)}=\sqrt{\int(K_{x,h}(z)-K_{y,h}(z))^{2}dQ(z)}\leq\frac{M_{K}}{h}\left\Vert x-y\right\Vert _{2}.
	\]
	Hence applying this to \eqref{eq:kde_vc_support_covering} implies
	that for all $\eta>0$, the supremum of the covering number $\mathcal{N}(\mathcal{F}_{K,h},L_{2}(Q),\eta)$
	over all measure $Q$ is upper bounded as 
	\[
	\sup_{Q}\mathcal{N}(\mathcal{F}_{K,h},L_{2}(Q),\eta)\leq\mathcal{N}\left(\mathbb{X},\left\Vert \cdot\right\Vert _{2},\frac{h\eta}{M_{K}}\right)\leq\left(1+\frac{2RM_{K}}{h\eta}\right)^{d}.
	\]
	Hence for all $\eta\in\left(0,\left\Vert K\right\Vert _{\infty}\right)$,
	\[
	\sup_{Q}\mathcal{N}(\mathcal{F}_{K,h},L_{2}(Q),\eta)\leq\left(\frac{2RM_{K}h^{-1}+\left\Vert K\right\Vert _{\infty}}{\eta}\right)^{d}.
	\]
	
\end{proof}

Then Corollary \ref{cor:kde_uniform_band_one_probconv} follows from applying Theorem \ref{thm:function_uniform} with bounding the covering number from Lemma \ref{lem:kde_vc}.

\textbf{Corollary \ref{cor:kde_uniform_band_one_probconv}.} \textit{
	Suppose there exists $R>0$ with $\mathbb{X}\subset\mathbb{B}_{\mathbb{R}^{d}}(0,R)$. Let $K$ be a $M_K$-Lipschitz continuous kernel function
	satisfying Assumption~\ref{ass:integrable}.
	Fix $\epsilon\in(0,d_{\rm{vol}})$. Further, if $d_{{\rm vol}}=0$ or under Assumption~\ref{ass:dimension_exact}, $\epsilon$ can be $0$. Suppose
	\[
	\limsup_{n}\frac{\left(\log\left(1 / h_{n}\right)\right)_{+}+\log\left(2 / \delta\right)}{nh_{n}^{d_{\mathrm{vol}}-\epsilon}}<\infty.
	\]
	Then  with probability at least $1-\delta$,
	\[
	\sup_{x\in\mathbb{X}}\left|\hat{p}_{h_{n}}(x)-p_{h_{n}}(x)\right|
	\leq C''\sqrt{\frac{(\log(\frac{1}{ h_{n}}))_{+}+\log(\frac{2}{\delta})}{nh_{n}^{2d-d_{\mathrm{vol}}+\epsilon}}},
	\]
	where $C''$ is a constant depending only on $R$, $M_{K}$, $\left\Vert K\right\Vert _{\infty}$,
	$d$, $\nu$, $d_{\mathrm{vol}}$, $C_{k=2,P,K,\epsilon}$, $\epsilon$.
}

\begin{proof}[Proof of Corollary \ref{cor:kde_uniform_band_one_probconv}]	
	For $x\in\mathbb{X}$, let $K_{x,h}:\mathbb{R}^{d}\to\mathbb{R}$
	be $K_{x,h}(\cdot)=K\left(\frac{x-\cdot}{h}\right)$, and let $\tilde{\mathcal{F}}_{K,h}:=\bigl\{ \frac{1}{h^{d}}K_{x,h}:\,x\in\mathbb{X}\bigr\} $
	be a class of normalized kernel functions centered on $\mathbb{X}$
	and bandwidth $h$. Note that $\hat{p}_{h}(x)-p_{h}(x)$ can be expanded
	as 
	\[
	\hat{p}_{h}(x)-p_{h}(x)=\frac{1}{nh^{d}}\sum_{i=1}^{n}K\left(\frac{x-X_{i}}{h}\right)-\mathbb{E}_{P}\left[\frac{1}{h^{d}}K\left(\frac{x-X_{i}}{h}\right)\right]=\frac{1}{n}\sum_{i=1}^{n}\frac{1}{h^{d}}K_{x,h}(X_{i})-\mathbb{E}_{P}\left[\frac{1}{h^{d}}K_{x,h}\right].
	\]
	Hence $\sup_{x\in\mathbb{X}}\left|\hat{p}_{h}(x)-p_{h}(x)\right|$
	can be expanded as 
	\begin{equation}
	\sup_{x\in\mathbb{X}}\left|\hat{p}_{h}(x)-p_{h}(x)\right|=\sup_{f\in\tilde{\mathcal{F}}_{K,h}}\left|\frac{1}{n}\sum_{i=1}^{n}f(X_{i})-\mathbb{E}_{P}\left[f(X)\right]\right|.\label{eq:kde_uniform_expansion-1}
	\end{equation}
	Now, it is immediate to check that 
	\begin{equation}
	\left\Vert f\right\Vert _{\infty}\leq h^{-d}\left\Vert K\right\Vert _{\infty}.\label{eq:kde_uniform_bound_inf-1}
	\end{equation}
	Also, Since $\tilde{\mathcal{F}}_{K,h}=h^{-d}\mathcal{F}_{K,h}$,
	VC dimension is uniformly bounded as Lemma \ref{lem:kde_vc} gives
	that for every probability measure $Q$ on $\mathbb{R}^{d}$ and for
	every $\eta\in(0,h^{-d}\left\Vert K\right\Vert _{\infty})$, the covering
	number $\mathcal{N}(\tilde{\mathcal{F}}_{K,h},L_{2}(Q),\eta)$ is
	upper bounded as 
	\begin{align}
	\sup_{Q}\mathcal{N}(\tilde{\mathcal{F}}_{K,h},L_{2}(Q),\eta) & =\sup_{Q}\mathcal{N}(\mathcal{F}_{K,h},L_{2}(Q),h^{d}\eta)\nonumber \\
	& \leq\left(\frac{2RM_{K}h^{-1}+\left\Vert K\right\Vert _{\infty}}{h^{d}\eta}\right)^{d}\nonumber \\
	& \leq\left(\frac{2RM_{K}\left\Vert K\right\Vert _{\infty}}{h^{d+1}\eta}\right)^{d}.\label{eq:kde_uniform_bound_vc-1}
	\end{align}
	Also, Lemma \ref{lem:kde_kernel_bound_two} implies that under Assumption
	\ref{ass:integrable}, for any $\epsilon\in(0,d_{{\rm {vol}}})$ (and
	$\epsilon$ can be $0$ if $d_{{\rm {vol}}}=0$ or under Assumption
	\ref{ass:dimension_exact}), 
	\begin{equation}
	\mathbb{E}_{P}\left[\left(\frac{1}{h^{d}}K_{x,h}\right)^{2}\right]\leq C_{k=2,P,K,\epsilon}h^{-2d+d_{\mathrm{vol}}-\epsilon}.\label{eq:kde_uniform_bound_two-1}
	\end{equation}
	Hence from \eqref{eq:kde_uniform_bound_inf-1}, \eqref{eq:kde_uniform_bound_vc-1},
	and \eqref{eq:kde_uniform_bound_two-1}, applying Theorem \ref{thm:function_uniform}
	to \eqref{eq:kde_uniform_expansion-1} gives that $\sup_{x\in\mathbb{X}}\left|\hat{p}_{h}(x)-p_{h}(x)\right|$
	is upper bounded with probability at least $1-\delta$ as 
	\begin{align*}
	& \sup_{x\in\mathbb{X}}\left|\hat{p}_{h}(x)-p_{h}(x)\right|\\
	& \leq C\left(\frac{2d\left\Vert K\right\Vert _{\infty}\log\left(\frac{2RM_{K}\left\Vert K\right\Vert _{\infty}}{\sqrt{C_{k=2,P,K,\epsilon}}h^{1+(d_{\mathrm{vol}}-\epsilon)/2}}\right)}{nh^{d}}+\sqrt{\frac{2dC_{k=2,P,K,\epsilon}\log\left(\frac{2RM_{K}\left\Vert K\right\Vert _{\infty}}{\sqrt{C_{k=2,P,K,\epsilon}}h^{1+(d_{\mathrm{vol}}-\epsilon)/2}}\right)}{nh^{2d-d_{\mathrm{vol}}+\epsilon}}}\right.\\
	& \qquad\qquad+\left.\sqrt{\frac{C_{k=2,P,K,\epsilon}\log(\frac{1}{\delta})}{nh^{2d-d_{\mathrm{vol}}+\epsilon}}}+\frac{\left\Vert K\right\Vert _{\infty}\log(\frac{1}{\delta})}{nh^{d}}\right)\\
	& \leq C_{R,M_{K},\left\Vert K\right\Vert _{\infty},d,\nu,d_{\mathrm{vol}},C_{k=2,P,K,\epsilon},\epsilon}\left(\frac{\left(\log\left(\frac{1}{h}\right)\right)_{+}}{nh^{d}}+\sqrt{\frac{\left(\log\left(\frac{1}{h}\right)\right)_{+}}{nh^{2d-d_{\mathrm{vol}}+\epsilon}}}+\sqrt{\frac{\log\left(\frac{2}{\delta}\right)}{nh^{2d-d_{\mathrm{vol}}+\epsilon}}}+\frac{\log\left(\frac{2}{\delta}\right)}{nh^{d}}\right),
	\end{align*}
	where $C_{R,M_{K},\left\Vert K\right\Vert _{\infty},d,\nu,d_{\mathrm{vol}},C_{k=2,P,K,\epsilon},\epsilon}$
	depends only on $R$, $M_{K}$, $\left\Vert K\right\Vert _{\infty}$,
	$d$, $\nu$, $d_{\mathrm{vol}}$, $C_{k=2,P,K,\epsilon}$, $\epsilon$.
\end{proof}

\section{Proof for Section \ref{sec:lower}}

Proposition \ref{prop:bound_lower} is shown by finding $x_{0}\in\mathbb{X}$ where the volume dimension is obtained, and analyzing the behavior of $\left|\hat{p}_{h_{n}}(x_{0})-p_{h_{n}}(x_{0})\right|$ by applying Central Limit Theorem.

\textbf{Proposition \ref{prop:bound_lower}.} \textit{
Suppose $P$ is a distribution satisfying Assumption~\ref{ass:dimension_exact_lower}
and with positive volume dimension $d_{{\rm vol}}>0$. Let $K$ be
a kernel function satisfying Assumption~\ref{ass:integrable}
with $k=1$ and $\lim_{t\to0}\inf_{\left\Vert x\right\Vert \leq t}K(x)>0$. Suppose
$\lim_{n} n h_{n}^{d_{{\rm vol}}} = \infty$.
Then, with probability $1-\delta$, the following holds for all large enough $n$ and small enough $h_{n}$: 
\[
\sup_{x\in\mathbb{X}}\left|\hat{p}_{h_{n}}(x)-p_{h_{n}}(x)\right|\geq C_{P,K,\delta}\sqrt{\frac{1}{nh_{n}^{2d-d_{{\rm vol}}}}}.
\]
where $C_{P,K,\delta}$ is a constant depending only on $P$, $K,$and
$\delta$.
}

\begin{proof}[Proof of Proposition \ref{prop:bound_lower}]

Note that $\lim_{t\to0}\inf_{\left\Vert x\right\Vert \leq t}K(x)>0$
implies that there exists $t_{0},K_{0}\in(0,\infty)$ such that 
\begin{equation}
K(x)\geq K_{0}I(\left\Vert x\right\Vert \leq t_{0}).\label{eq:lower_kernel_lower}
\end{equation}
Also, from $\sup_{x\in\mathbb{X}}\liminf_{r\to0}\frac{P(\mathbb{B}_{\mathbb{R}^{d}}(x,r))}{r^{d_{{\rm vol}}}}>0$,
we can choose $x_{0}\in\mathbb{X}$ such that $\liminf_{r\to0}\frac{P(\mathbb{B}_{\mathbb{R}^{d}}(x_{0},r))}{r^{d_{{\rm vol}}}}>0$.
From $\{h_{n}\}_{n\in\mathbb{N}}$ bounded, there exists $r_{0}>0$
and $p_{0}>0$ such that $r_{0}\geq h_{n}t_{0}$ for all $n\in\mathbb{N}$
and for all $r\leq r_{0}$, 
\begin{equation}
P(\mathbb{B}_{\mathbb{R}^{d}}(x_{0},r))\geq p_{0}r^{d_{{\rm vol}}}.\label{eq:lower_prob_lower}
\end{equation}
For $x\in\mathbb{X}$ and $h>0$, let $f_{x,h}:\mathbb{R}^{d}\to\mathbb{R}$
be $f_{x,h}=\frac{1}{h^{d}}\left(K_{x,h}-\mathbb{E}_{P}[K_{x,h}]\right)$,
so that at $x_{0}\in\mathbb{X}$, $\hat{p}_{h_{n}}(x_{0})-p_{h_{n}}(x_{0})$
is expanded as 
\[
\hat{p}_{h_{n}}(x_{0})-p_{h_{n}}(x_{0})=\frac{1}{n}\sum_{i=1}^{n}f_{x_{0},h_{n}}(X_{i}).
\]
Below we get a lower bound for $\mathbb{E}_{P}[f_{x_{0},h_{n}}^{2}]$.
First, fix $\epsilon<\frac{d_{{\rm vol}}}{2}$. Then from Lemma \ref{lem:kde_kernel_bound_two},
\begin{equation}
\mathbb{E}_{P}\left[\left|K_{x_{0},h}\right|\right]\leq C_{k=1,P,K,\epsilon}h^{d_{{\rm vol}}-\epsilon}.\label{eq:lower_one_upper}
\end{equation}
Now, we lower bound $\mathbb{E}_{P}[K_{x_{0},h}^{2}]$.
By applying \eqref{eq:lower_kernel_lower}, $\mathbb{E}_{P}[K_{x_{0},h}^{2}]$
is lower bounded as 
\begin{align*}
\mathbb{E}_{P}\left[K_{x_{0},h}^{2}\right] & \geq\mathbb{E}_{P}\left[K_{0}I\left(\left\Vert \frac{x_{0}-X_{i}}{h}\right\Vert \geq t_{0}\right)\right]\\
 & =K_{0}^{2}P(\mathbb{B}_{\mathbb{R}^{d}}(x_{0},ht_{0})).
\end{align*}
Then applying \eqref{eq:lower_prob_lower} gives a further lower bound
as 
\begin{align}
\mathbb{E}_{P}\left[K_{x_{0},h}^{2}\right] & \geq K_{0}^{2}p_{0}t_{0}^{d_{{\rm vol}}}h^{d_{{\rm vol}}}.\label{eq:lower_two_lower}
\end{align}
Then combining \eqref{eq:lower_one_upper} and \eqref{eq:lower_two_lower}
gives a lower bound of $\mathbb{E}_{P}[f_{x_{0},h}^{2}]$
as 
\begin{align*}
\mathbb{E}_{P}\left[f_{x_{0},h}^{2}\right] & =\frac{1}{h^{2d}}\left(\mathbb{E}_{P}\left[K_{x_{0},h}^{2}\right]-\left(\mathbb{E}_{P}\left[K_{x_{0},h}\right]\right)^{2}\right)\\
 & \geq h^{d_{{\rm vol}}-2d}(K_{0}^{2}p_{0}t_{0}^{d_{{\rm vol}}}-C_{k=1,P,K,\epsilon}^{2}h^{d_{{\rm vol}}-2\epsilon}).
\end{align*}
Hence from $d_{{\rm vol}}-2\epsilon>0$, there exists $h_{P,K}$ and
$C_{P,K}'$ depending only on $P$ and $K$ such that $h_{n}\leq h_{P,K}$
implies 
\begin{equation}
\mathbb{E}_{P}\left[f_{x_{0},h_{n}}^{2}\right]\geq C_{P,K}'h_{n}^{d_{{\rm vol}}-2d}.\label{eq:lower_var_lower}
\end{equation}
Now, let $s_{n}:=\sqrt{\sum_{i=1}^{n}\mathbb{E}_{P}[f_{x_{0},h_{n}}^{2}(X_{i})]}$.
Then \eqref{eq:lower_var_lower} gives
\[
s_{n}\geq\sqrt{C_{P,K}'nh_{n}^{d_{{\rm vol}}-2d}}.
\]
Then for any $\epsilon>0$, when $n$ is large enough so that $nh_{n}^{d_{{\rm vol}}}>\frac{\left\Vert K\right\Vert _{\infty}^{2}}{\epsilon^{2}C_{P,K}'}$,
then 
\[
\left\Vert f_{x_{0},h_{n}}\right\Vert _{\infty}\leq h^{-d}\left\Vert K\right\Vert _{\infty}<\epsilon\sqrt{C_{P,K}'nh_{n}^{d_{{\rm vol}}-2d}}\leq s_{n}.
\]
Hence Lindeberg condition holds as for $n$ large enough so that $nh_{n}^{d_{{\rm vol}}}>\frac{\left\Vert K\right\Vert _{\infty}^{2}}{\epsilon^{2}C_{P,K}'}$,
then
\[
\frac{1}{s_{n}^{2}}\sum_{i=1}^{n}\mathbb{E}\left[f_{x_{0},h_{n}}^{2}(X_{i})I\left(|f_{x_{0},h_{n}}(X_{i})|\geq\epsilon s_{n}\right)\right]=0.
\]
Hence, Lindeberg-Feller Central Limit Theorem gives 
\[
\sqrt{\frac{n}{\mathbb{E}_{P}[f_{x_{0},h_{n}}^{2}]}}(\hat{p}_{h_{n}}(x_{0})-p_{h_{n}}(x_{0}))\overset{d}{\to}N\left(0,1\right).
\]
Hence, for fixed $\delta\in(0,1)$, let $q_{\delta/2}\in\mathbb{R}$
be such that $P(|Z|\leq q_{\delta/2})=\frac{\delta}{2}$ for $Z\sim N(0,1)$,
then 
\[
\lim_{n\to\infty}P\left(\left|\sqrt{\frac{n}{\mathbb{E}_{P}[f_{x_{0},h_{n}}^{2}]}}(\hat{p}_{h_{n}}(x_{0})-p_{h_{n}}(x_{0}))\right|\geq q_{\delta/2}\right)=1-\frac{\delta}{2}.
\]
And hence there exists $N<\infty$ that for all $n\geq N$, 
\[
P\left(\left|\hat{p}_{h_{n}}(x_{0})-p_{h_{n}}(x_{0})\right|\geq q_{\delta/2}\sqrt{\frac{\mathbb{E}_{P}[f_{x_{0},h_{n}}^{2}]}{n}}\right)\geq1-\delta.
\]
Then applying \eqref{eq:lower_var_lower} implies that with probability
at least $1-\delta$, 
\[
\left|\hat{p}_{h_{n}}(x_{0})-p_{h_{n}}(x_{0})\right|\geq\sqrt{\frac{q_{\delta/2}^{2}C_{P,K}'}{nh_{n}^{2d-d_{{\rm vol}}}}}=C_{P,K,\delta}\sqrt{\frac{1}{nh_{n}^{2d-d_{{\rm vol}}}}},
\]
where $C_{P,K,\delta}=q_{\delta/2}\sqrt{C_{P,K}'}$ depends only on
$P$, $K$, and $\delta$. Then from 
\[
\sup_{x\in\mathbb{X}}\left|\hat{p}_{h_{n}}(x)-p_{h_{n}}(x)\right|\geq\left|\hat{p}_{h_{n}}(x_{0})-p_{h_{n}}(x_{0})\right|,
\]
we get the same lower bound for $\sup_{x\in\mathbb{X}}\left|\hat{p}_{h_{n}}(x)-p_{h_{n}}(x)\right|$
with probability at least $1-\delta$ as 
\[
\sup_{x\in\mathbb{X}}\left|\hat{p}_{h_{n}}(x)-p_{h_{n}}(x)\right|\geq\sqrt{\frac{q_{\delta/2}^{2}C_{P,K}'}{nh_{n}^{2d-d_{{\rm vol}}}}}=C_{P,K,\delta}\sqrt{\frac{1}{nh_{n}^{2d-d_{{\rm vol}}}}}.
\]

\end{proof}

\section{Proof for Section \ref{sec:derivative}}

For showing Lemma \ref{lem:derivative_kernel_bound_two}, we proceed similarly to proof of Lemma \ref{lem:kde_kernel_bound_two},
where we plug in $D^{s}K$ in the place of $K$.

\textbf{Lemma \ref{lem:derivative_kernel_bound_two}.} \textit{
	Let $(\mathbb{R}^{d},P)$
	be a probability space and let $X\sim P$. For any kernel $K$ satisfying Assumption~\ref{ass:derivative_integrable},
	the expectation of the square of the derivative of the kernel is upper bounded as 
	\[
	\mathbb{E}_{P}\left[\left(D^{s}K\left(\frac{x-X}{h}\right)\right)^{2}\right]\leq C_{s,P,K,\epsilon}h^{d_{\mathrm{vol}}-\epsilon},
	\]
	for any $\epsilon \in(0,d_{\rm{vol}})$, where $C_{s,P,K,\epsilon}$ is a constant depending only on $s$, $P$, $K$, $\epsilon$. Further, if $d_{\rm{vol}}=0$ or under Assumption~\ref{ass:dimension_exact}, $\epsilon$ can be $0$ in \eqref{eq:derivative_kernel_bound}.
}

\begin{proof}[Proof of Lemma \ref{lem:derivative_kernel_bound_two}]
	
	We first consider the case when $d_{{\rm vol}}=0$. Then $\mathbb{E}_{P}\left[\left(D^{s}K\left(\frac{x-X}{h}\right)\right)^{2}\right]$
	is simply bounded as 
	\[
	\mathbb{E}_{P}\left[\left(D^{s}K\left(\frac{x-X}{h}\right)\right)^{2}\right]\leq\left\Vert D^{s}K\right\Vert _{\infty}^{2}h^{0}.
	\]
	
	Now, we consider the case when $d_{{\rm vol}}>0$. Fix $\epsilon\in(0,d_{{\rm vol}})$.
	Under Assumption \ref{ass:dimension_exact}, $\epsilon$ can be chosen
	to be $0$.
	
	Let $C_{s,K,d_{\mathrm{vol}},\epsilon}:=\int_{0}^{\infty}t^{d_{\mathrm{vol}}-\epsilon-1}\sup_{\left\Vert x\right\Vert \leq t}(D^{s}K(x))^{2}dt$,
	then it is finite from \eqref{eq:deriv_kde_condition_integral_finite}
	and $\left\Vert D^{s}K\right\Vert _{\infty}<\infty$ in Assumption \ref{ass:derivative_vc}
	as 
	\begin{align*}
	\int_{0}^{\infty}t^{d_{\mathrm{vol}}-\epsilon-1}\sup_{\left\Vert x\right\Vert \leq t}(D^{s}K(x))^{2}dt & \leq\int_{0}^{1}t^{d_{\mathrm{vol}}-\epsilon-1}\left\Vert D^{s}K\right\Vert _{\infty}dt+\int_{1}^{\infty}t^{d_{{\rm vol}}-1}\sup_{\left\Vert x\right\Vert \leq t}(D^{s}K(x))^{2}dt\\
	& \leq\frac{\left\Vert D^{s}K\right\Vert _{\infty}}{d_{{\rm vol}}-\epsilon}+\int_{0}^{\infty}t^{d_{{\rm vol}}-1}\sup_{\left\Vert x\right\Vert \leq t}(D^{s}K(x))^{2}dt<\infty.
	\end{align*}
	Fix $\eta>0$, and let $\tilde{K}_{\eta}:[0,\infty)\to\mathbb{R}$
	be a continuous and strictly decreasing function satisfying $\tilde{K}_{\eta}(t)>\sup_{\left\Vert x\right\Vert \geq t}(D^{s}K(x))^{2}$
	for all $t\geq0$ and $\int_{0}^{\infty}t^{d_{\mathrm{vol}}-\epsilon-1}(\tilde{K}_{\eta}(t)-\sup_{\left\Vert x\right\Vert \geq t}(D^{s}K(x))^{2})dt=\eta$.
	Such existence is possible since $t\mapsto\sup_{\left\Vert x\right\Vert \geq t}(D^{s}K(x))^{2}$
	is nonincreasing function, so have at most countable discontinuous
	points, and $\int_{0}^{\infty}t^{d_{\mathrm{vol}}-\epsilon-1}\sup_{\left\Vert x\right\Vert \leq t}(D^{s}K(x))^{2}dt<\infty$.
	Then it is immediate to check that 
	\begin{equation}
	(D^{s}K(x))^{2}<\tilde{K}_{\eta}(\left\Vert x\right\Vert )\text{ for all }x\in\mathbb{R}.\label{eq:kde_kernel_bound_upper_onedim-1}
	\end{equation}
	Then $\int_{0}^{\infty}t^{d_{\mathrm{vol}}-\epsilon-1}\tilde{K}(t)dt$
	can be expanded as 
	\begin{align}
	\int_{0}^{\infty}t^{d_{\mathrm{vol}}-\epsilon-1}\tilde{K}_{\eta}(t)dt & =\int_{0}^{\infty}t^{d_{\mathrm{vol}}-\epsilon-1}\sup_{\left\Vert x\right\Vert \leq t}(D^{s}K(x))^{2}dt+\int_{0}^{\infty}t^{d_{\mathrm{vol}}-\epsilon-1}(\tilde{K}_{\eta}(t)-\sup_{\left\Vert x\right\Vert \geq t}(D^{s}K(x))^{2})dt\nonumber \\
	& =C_{s,K,d_{\mathrm{vol}},\epsilon}+\eta<\infty.\label{eq:kde_kernel_bound_upper_integral-1}
	\end{align}
	Now since $\tilde{K}_{\eta}$ is continuous and strictly decreasing,
	change of variables $t=\tilde{K}_{\eta}(u)$ is applicable, and then
	$\mathbb{E}_{P}\left[\left(D^{s}K\left(\frac{x-X}{h}\right)\right)^{2}\right]$
	can be expanded as 
	\begin{align*}
	\mathbb{E}_{P}\left[\left(D^{s}K\left(\frac{x-X}{h}\right)\right)^{2}\right] & =\int_{0}^{\infty}P\left(\left(D^{s}K\left(\frac{x-X}{h}\right)\right)^{2}>t\right)dt\\
	& =\int_{\infty}^{0}P\left(\left(D^{s}K\left(\frac{x-X}{h}\right)\right)^{2}>\tilde{K}_{\eta}(u)\right)d\tilde{K}_{\eta}(u).
	\end{align*}
	Now, from \eqref{eq:kde_kernel_bound_upper_onedim-1} and $\tilde{K}_{\eta}$
	being a strictly decreasing, we can upper bound $\mathbb{E}_{P}\left[\left(D^{s}K\left(\frac{x-X}{h}\right)\right)^{2}\right]$
	as 
	\begin{align*}
	\mathbb{E}_{P}\left[\left(D^{s}K\left(\frac{x-X}{h}\right)\right)^{2}\right] & \leq\int_{\infty}^{0}P\left(\tilde{K}_{\eta}\left(\frac{\left\Vert x-X\right\Vert }{h}\right)>\tilde{K}_{\eta}(u)\right)d\tilde{K}_{\eta}(u)\\
	& =\int_{\infty}^{0}P\left(\frac{\left\Vert x-X\right\Vert }{h}<u\right)d\tilde{K}_{\eta}(u)\\
	& =\int_{\infty}^{0}P\left(\mathbb{B}_{\mathbb{R}^{d}}(x,hu)\right)d\tilde{K}_{\eta}(u).
	\end{align*}
	Now, from Lemma \ref{lem:dimension_volume_lower} (and \eqref{eq:dimension_exact}
	for Assumption \ref{ass:dimension_exact} case), there exists $C_{d_{{\rm vol}}-\epsilon,P}<\infty$
	with $P\left(\mathbb{B}_{\mathbb{R}^{d}}(x,r)\right)\leq C_{d_{{\rm vol}}-\epsilon,P}r^{d_{\mathrm{vol}}-\epsilon}$
	for all $x\in\mathbb{X}$ and $r>0$. Then $\mathbb{E}_{P}\left[\left(D^{s}K\left(\frac{x-X}{h}\right)\right)^{2}\right]$
	is further upper bounded as 
	\begin{align}
	\mathbb{E}_{P}\left[\left(D^{s}K\left(\frac{x-X}{h}\right)\right)^{2}\right] & \leq\int_{\infty}^{0}C_{d_{{\rm vol}}-\epsilon,P}(hu)^{d_{\mathrm{vol}}-\epsilon}d\tilde{K}(u)\nonumber \\
	& =C_{d_{{\rm vol}}-\epsilon,P}h^{d_{\mathrm{vol}}-\epsilon}\int_{\infty}^{0}u^{d_{\mathrm{vol}}-\epsilon}d\tilde{K}(u).\label{eq:kde_kernel_bound_upperbound_stieltjes-1}
	\end{align}
	Now, $\int_{\infty}^{0}u^{d_{\mathrm{vol}}-\epsilon}d\tilde{K}(u)$
	can be computed using integration by part. Note first that $\int_{0}^{\infty}t^{d_{\mathrm{vol}}-\epsilon-1}\tilde{K}(t)dt<\infty$
	implies 
	\[
	\lim_{t\to\infty}t^{d_{\mathrm{vol}}-\epsilon}\tilde{K}(t)=0.
	\]
	To see this, note that $t^{d_{\mathrm{vol}}-\epsilon}\tilde{K}(t)$
	is expanded as 
	\[
	t^{d_{\mathrm{vol}}-\epsilon}\tilde{K}(t)=\int_{0}^{t}u^{d_{\mathrm{vol}}-\epsilon}d\tilde{K}(u)+\int_{0}^{t}(d_{\mathrm{vol}}-\epsilon)u^{d_{\mathrm{vol}}-\epsilon-1}\tilde{K}(u)du,
	\]
	then $\int_{0}^{\infty}(d_{\mathrm{vol}}-\epsilon)u^{d_{\mathrm{vol}}-\epsilon-1}\tilde{K}(u)du<\infty$
	and $\int_{0}^{t}u^{d_{\mathrm{vol}}-\epsilon}d\tilde{K}(u)$ being
	monotone function of $t$ imply that $\lim_{t\to\infty}t^{d_{\mathrm{vol}}-\epsilon}\tilde{K}(t)$
	exists. Now, suppose $\lim_{t\to\infty}t^{d_{\mathrm{vol}}-\epsilon}\tilde{K}(t)=a>0$,
	then we can choose $t_{0}>0$ such that $t^{d_{\mathrm{vol}}-\epsilon}\tilde{K}(t)>\frac{a}{2}$
	for all $t\geq t_{0}$, and then 
	\[
	\infty>\int_{0}^{\infty}t^{d_{\mathrm{vol}}-\epsilon-1}\tilde{K}(t)dt\geq\int_{t_{0}}^{\infty}t^{d_{\mathrm{vol}}-\epsilon-1}\tilde{K}(t)dt\geq\frac{a}{2}\int_{t_{0}}^{\infty}t^{-1}dt=\infty,
	\]
	which is a contradiction. Hence $\lim_{t\to\infty}t^{d_{\mathrm{vol}}-\epsilon}\tilde{K}(t)=0$.
	Now, applying integration by part to $\int_{\infty}^{0}u^{d_{\mathrm{vol}}-\epsilon}d\tilde{K}(u)$
	with $d_{{\rm vol}}-\epsilon>0$ gives 
	\begin{align}
	\int_{\infty}^{0}u^{d_{\mathrm{vol}}-\epsilon}d\tilde{K}(u) & =\left[u^{d_{\mathrm{vol}}-\epsilon}\tilde{K}(u)\right]_{\infty}^{0}-\int_{\infty}^{0}(d_{\mathrm{vol}}-\epsilon)u^{d_{\mathrm{vol}}-\epsilon-1}\tilde{K}(u)du\nonumber \\
	& =\int_{0}^{\infty}(d_{\mathrm{vol}}-\epsilon)u^{d_{\mathrm{vol}}-\epsilon-1}\tilde{K}(u)du.\label{eq:kde_kernel_bound_integral_part-1}
	\end{align}
	Then applying \eqref{eq:kde_kernel_bound_upper_integral-1} and \eqref{eq:kde_kernel_bound_integral_part-1}
	to \eqref{eq:kde_kernel_bound_upperbound_stieltjes-1} gives an upper
	bound for $\mathbb{E}_{P}\left[\left(D^{s}K\left(\frac{x-X}{h}\right)\right)^{2}\right]$
	as 
	\begin{equation}
	\mathbb{E}_{P}\left[\left(D^{s}K\left(\frac{x-X}{h}\right)\right)^{2}\right]\leq C_{d_{{\rm vol}}-\epsilon,P}(d_{\mathrm{vol}}-\epsilon)h^{d_{\mathrm{vol}}-\epsilon}(C_{s,K,d_{\mathrm{vol}},\epsilon}+\eta).\label{eq:kde_kernel_bound_upperbound_epsilon-1}
	\end{equation}
	And then note that RHS of \eqref{eq:kde_kernel_bound_upperbound_epsilon-1}
	holds for any $\eta>0$, and hence $\mathbb{E}_{P}\left[\left(D^{s}K\left(\frac{x-X}{h}\right)\right)^{2}\right]$
	is further upper bounded as 
	\begin{align*}
	\mathbb{E}_{P}\left[\left(D^{s}K\left(\frac{x-X}{h}\right)\right)^{2}\right] & \leq\inf_{\eta>0}\left\{ C_{d_{{\rm vol}}-\epsilon,P}(d_{\mathrm{vol}}-\epsilon)h^{d_{\mathrm{vol}}-\epsilon}(C_{s,K,d_{\mathrm{vol}},\epsilon}+\eta)\right\} \\
	& =C_{d_{{\rm vol}}-\epsilon,P}(d_{\mathrm{vol}}-\epsilon)C_{s,K,d_{\mathrm{vol}},\epsilon}h^{d_{\mathrm{vol}}-\epsilon}\\
	& =C_{s,P,K,\epsilon}h^{d_{{\rm vol}}-\epsilon},
	\end{align*}
	where $C_{k,P,K,\epsilon}=C_{d_{{\rm vol}}-\epsilon,P}(d_{\mathrm{vol}}-\epsilon)C_{s,K,d_{\mathrm{vol}},\epsilon}$.
	
\end{proof}

For proving Theorem \ref{thm:derivative_band_ray_uniform}, we proceed similarly to the proof of Theorem \ref{thm:kde_band_ray_uniform}.
Analogous to bounding $\mathbb{E}_{P}[K_{x,h}^{2}]$ by
Lemma \ref{lem:kde_kernel_bound_two}, we bound $\mathbb{E}_{P}[(D^{s}K_{x,h})^{2}]$
by Lemma \ref{lem:derivative_kernel_bound_two}.

\textbf{Theorem \ref{thm:derivative_band_ray_uniform}.} \textit{
	Let $P$ be a distribution and $K$ be a kernel function satisfying
	Assumption~\ref{ass:derivative_leibniz},~\ref{ass:derivative_integrable}, and~\ref{ass:derivative_vc}. Then, with probability at least $1-\delta$,
	\begin{align*}
	& \sup_{h\geq l_{n},x\in\mathbb{X}}\left|D^{s}\hat{p}_{h}(x)-D^{s}p_{h}(x)\right|\\
	& \leq C\left(\frac{\left(\log\left(1 / l_{n}\right)\right)_{+}}{nl_{n}^{d+|s|}}+\sqrt{\frac{\left(\log\left(1 / l_{n}\right)\right)_{+}}{nl_{n}^{2d+2|s|-d_{\mathrm{vol}}+\epsilon}}}
	+\sqrt{\frac{\log\left(2 / \delta\right)}{nl_{n}^{2d+2|s|-d_{\mathrm{vol}}+\epsilon}}}+\frac{\log\left(2 / \delta\right)}{nl_{n}^{d+|s|}}\right),
	\end{align*}
	for any $\epsilon\in(0,d_{\rm{vol}})$, where $C$ is a constant depending only on $A$, $\left\Vert D^{s}K\right\Vert _{\infty}$, $d$,
	$\nu$, $d_{\mathrm{vol}}$, $C_{s,P,K,\epsilon}$, $\epsilon$. Further, if $d_{\rm{vol}}=0$ or under Assumption~\ref{ass:dimension_exact}, $\epsilon$ can be $0$ in \eqref{eq:derivative_band_ray_uniform_bound}.
}	

\begin{proof}[Proof of Theorem \ref{thm:derivative_band_ray_uniform}]
	
	For $x\in\mathbb{X}$ and $h\geq l_{n}$, let $D^{s}K_{x,h}:\mathbb{R}^{d}\to\mathbb{R}$
	be $D^{s}K_{x,h}(\cdot)=D^{s}K\left(\frac{x-\cdot}{h}\right)$, and let $\tilde{\mathcal{F}}_{K,[l_{n},\infty)}^{s}:=\left\{ \frac{1}{h^{d+|s|}}D^{s}K_{x,h}:\,x\in\mathbb{X},h\geq l_{n}\right\} $
	be a class of normalized kernel functions centered on $\mathbb{X}$
	and bandwidth in $[l_{n},\infty)$. Note that $D^{s}\hat{p}_{h}(x)-D^{s}p_{h}(x)$
	can be expanded as 
	\begin{align*}
	D^{s}\hat{p}_{h}(x)-D^{s}p_{h}(x)&=\frac{1}{nh^{d+|s|}}\sum_{i=1}^{n}D^{s}K\left(\frac{x-X_{i}}{h}\right)-\mathbb{E}_{P}\left[\frac{1}{h^{d+|s|}}D^{s}K\left(\frac{x-X_{i}}{h}\right)\right]\\
	&=\frac{1}{n}\sum_{i=1}^{n}\frac{1}{h^{d+|s|}}D^{s}K_{x,h}(X_{i})-\mathbb{E}_{P}\left[\frac{1}{h^{d+|s|}}D^{s}K_{x,h}\right].
	\end{align*}
	Hence $\sup_{h\geq l_{n},x\in\mathbb{X}}\left|D^{s}\hat{p}_{h}(x)-D^{s}p_{h}(x)\right|$
	can be expanded as 
	\begin{equation}
	\sup_{h\geq l_{n},x\in\mathbb{X}}\left|D^{s}\hat{p}_{h}(x)-D^{s}p_{h}(x)\right|=\sup_{f\in\tilde{\mathcal{F}}_{K,[l_{n},\infty)}^{s}}\left|\frac{1}{n}\sum_{i=1}^{n}f(X_{i})-\mathbb{E}_{P}\left[f(X)\right]\right|.\label{eq:kde_uniform_expansion-2}
	\end{equation}
	Now, it is immediate to check that 
	\begin{equation}
	\left\Vert f\right\Vert _{\infty}\leq l_{n}^{-d-|s|}\left\Vert D^{s}K\right\Vert _{\infty}.\label{eq:kde_uniform_bound_inf-2}
	\end{equation}
	For bounding the VC dimension of $\tilde{\mathcal{F}}_{K,[l_{n},\infty)}^{s}$,
	consider $\mathcal{F}_{K,[l_{n},\infty)}^{s}:=\left\{ D^{s}K_{x,h}:\,x\in\mathbb{X},h\geq l_{n}\right\} $
	be a class of unnormalized kernel functions centered on $\mathbb{X}$
	and bandwidth in $[l_{n},\infty)$. Fix $\eta<l_{n}^{-d-|s|}\left\Vert D^{s}K\right\Vert _{\infty}$
	and a probability measure $Q$ on $\mathbb{R}^{d}$. Suppose $\left[l_{n},\left(\frac{\eta}{2\left\Vert D^{s}K\right\Vert _{\infty}}\right)^{-1/(d+|s|)}\right]$
	is covered by balls $\Bigl\{ \Bigl(h_{i}-\frac{l_{n}^{d+|s|+1}\eta}{2(d+|s|)\left\Vert D^{s}K\right\Vert _{\infty}},h_{i}+\frac{l_{n}^{d+|s|+1}\eta}{2(d+|s|)\left\Vert D^{s}K\right\Vert _{\infty}}\Bigr):\,1\leq i\leq N_{1}\Bigr\} $
	and $(\mathcal{F}_{K,[l_{n},\infty)}^{s},L_{2}(Q))$ is covered by
	balls $\Bigl\{ \mathbb{B}_{L_{2}(Q)}\left(f_{j},\frac{l_{n}^{d+|s|}\eta}{2}\right):\,1\leq j\leq N_{2}\Bigr\} $,
	and let $f_{i,j}:=h_{i}^{-d-|s|}f_{j}$ for $1\leq i\leq N_{1}$ and
	$1\leq j\leq N_{2}$. Also, choose $h_{0}>\left(\frac{\eta}{2\left\Vert D^{s}K\right\Vert _{\infty}}\right)^{-1/(d+|s|)}$,
	$x_{0}\in\mathbb{X}$, and let $f_{0}=\frac{1}{h_{0}^{d+|s|}}D^{s}K_{x_{0},h_{0}}$.
	We will show that 
	\begin{equation}
	\left\{ \mathbb{B}_{L_{2}(Q)}\left(f_{i,j},\eta\right):\,1\leq i\leq N_{1},\,1\leq j\leq N_{2}\right\} \cup\left\{ \mathbb{B}_{L_{2}(Q)}\left(f_{0},\eta\right)\right\} \text{ covers }\tilde{\mathcal{F}}_{K,[l_{n},\infty)}^{s}.\label{eq:kde_uniform_covering-1}
	\end{equation}
	For the first case when $h\leq\left(\frac{\eta}{2\left\Vert D^{s}K\right\Vert _{\infty}}\right)^{-1/(d+|s|)}$,
	find $h_{i}$ and $f_{j}$ with $h\in\Bigl(h_{i}-\frac{l_{n}^{d+|s|+1}\eta}{2(d+|s|)\left\Vert D^{s}K\right\Vert _{\infty}},h_{i}+\frac{l_{n}^{d+|s|+1}\eta}{2(d+|s|)\left\Vert D^{s}K\right\Vert _{\infty}}\Bigr)$
	and $K_{x,h}\in\mathbb{B}_{L_{2}(Q)}\left(f_{j},\frac{l_{n}^{d+|s|}\eta}{2}\right)$.
	Then the distance between $\frac{1}{h^{d+|s|}}D^{s}K_{x,h}$ and $\frac{1}{h_{i}^{d+|s|}}f_{j}$
	is upper bounded as 
	\begin{align}
	& \left\Vert \frac{1}{h^{d+|s|}}D^{s}K_{x,h}-\frac{1}{h_{i}^{d+|s|}}f_{j}\right\Vert _{L_{2}(Q)}\nonumber \\
	& \leq\left\Vert \frac{1}{h^{d+|s|}}D^{s}K_{x,h}-\frac{1}{h_{i}^{d+|s|}}D^{s}K_{x,h}\right\Vert _{L_{2}(Q)}+\left\Vert \frac{1}{h_{i}^{d+|s|}}D^{s}K_{x,h}-\frac{1}{h_{i}^{d+|s|}}f_{j}\right\Vert _{L_{2}(Q)}.\label{eq:kde_uniform_covering_decomposition-1}
	\end{align}
	Now, the first term of \eqref{eq:kde_uniform_covering_decomposition-1}
	is upper bounded as 
	\begin{align}
	\left\Vert \frac{1}{h^{d+|s|}}D^{s}K_{x,h}-\frac{1}{h_{i}^{d+|s|}}D^{s}K_{x,h}\right\Vert _{L_{2}(Q)} & =\left|\frac{1}{h^{d+|s|}}-\frac{1}{h_{i}^{d+|s|}}\right|\left\Vert D^{s}K_{x,h}\right\Vert _{L_{2}(Q)}\nonumber \\
	& =\left|h_{i}-h\right|\sum_{k=0}^{d+|s|-1}h_{i}^{k-d-|s|}h^{-1-k}\left\Vert D^{s}K_{x,h}\right\Vert _{L_{2}(Q)}\nonumber \\
	& \leq\left|h_{i}-h\right|(d+|s|)l_{n}^{-d-|s|-1}\left\Vert D^{s}K\right\Vert _{\infty}<\frac{\eta}{2}.\label{eq:kde_uniform_covering_decomposition_first-1}
	\end{align}
	Also, the second term of \eqref{eq:kde_uniform_covering_decomposition-1}
	is upper bounded as 
	\begin{align}
	\left\Vert \frac{1}{h_{i}^{d+|s|}}D^{s}K_{x,h}-\frac{1}{h_{i}^{d+|s|}}f\right\Vert _{L_{2}(Q)} & =\frac{1}{h_{i}^{d+|s|}}\left\Vert D^{s}K_{x,h}-f\right\Vert _{L_{2}(Q)}\nonumber \\
	& \leq l_{n}^{-d-|s|}\left\Vert D^{s}K_{x,h}-f\right\Vert _{L_{2}(Q)}<\frac{\eta}{2}.\label{eq:kde_uniform_covering_decomposition_second-1}
	\end{align}
	Hence applying \eqref{eq:kde_uniform_covering_decomposition_first-1}
	and \eqref{eq:kde_uniform_covering_decomposition_second-1} to \eqref{eq:kde_uniform_covering_decomposition-1}
	gives 
	\[
	\left\Vert \frac{1}{h^{d+|s|}}D^{s}K_{x,h}-\frac{1}{h_{i}^{d+|s|}}f_{j}\right\Vert _{L_{2}(Q)}<\eta.
	\]
	For the second case when $h>\left(\frac{\eta}{2\left\Vert D^{s}K\right\Vert _{\infty}}\right)^{-1/(d+|s|)}$,
	$\left\Vert \frac{1}{h^{d+|s|}}D^{s}K_{x,h}\right\Vert _{L_{2}(Q)}\leq\left\Vert \frac{1}{h^{d+|s|}}D^{s}K_{x,h}\right\Vert _{\infty}<\frac{\eta}{2}$
	holds, and hence 
	\[
	\left\Vert \frac{1}{h^{d+|s|}}D^{s}K_{x,h}-f_{0}\right\Vert _{L_{2}(Q)}\leq\left\Vert \frac{1}{h^{d+|s|}}D^{s}K_{x,h}\right\Vert _{L_{2}(Q)}+\left\Vert f_{0}\right\Vert _{L_{2}(Q)}<\eta.
	\]
	Therefore, \eqref{eq:kde_uniform_covering-1} is shown. Hence combined
	with Assumption \ref{ass:derivative_vc} gives that for every probability
	measure $Q$ on $\mathbb{R}^{d}$ and for every $\eta\in(0,h^{-d}\left\Vert D^{s}K\right\Vert _{\infty})$,
	the covering number $\mathcal{N}(\tilde{\mathcal{F}}_{K,[l_{n},\infty)},L_{2}(Q),\eta)$
	is upper bounded as 
	\begin{align}
	& \sup_{Q}\mathcal{N}(\tilde{\mathcal{F}}_{K,[l_{n},\infty)},L_{2}(Q),\eta) \nonumber\\
	& \leq\mathcal{N}\left(\left[l_{n},\left(\frac{\eta}{2\left\Vert D^{s}K\right\Vert _{\infty}}\right)^{-1/(d+|s|)}\right],|\cdot|,\frac{l_{n}^{d+|s|+1}\eta}{2(d+|s|)\left\Vert D^{s}K\right\Vert _{\infty}}\right)\sup_{Q}\mathcal{N}\left(\mathcal{F}_{K,[l_{n},\infty)},L_{2}(Q),\frac{l_{n}^{d+|s|}\eta}{2}\right)+1\nonumber \\
	& \leq\frac{2(d+|s|)\left\Vert D^{s}K\right\Vert _{\infty}}{l_{n}^{d+|s|+1}\eta}\left(\frac{2\left\Vert D^{s}K\right\Vert _{\infty}}{\eta}\right)^{1/(d+|s|)}\left(\frac{2A\left\Vert D^{s}K\right\Vert _{\infty}}{l_{n}^{d+|s|}\eta}\right)^{\nu}+1\nonumber \\
	& \leq\left(\frac{2A(d+|s|)\left\Vert D^{s}K\right\Vert _{\infty}}{l_{n}^{d+|s|}\eta}\right)^{\nu+2}.\label{eq:kde_uniform_bound_vc-2}
	\end{align}
	Also, Lemma \ref{lem:derivative_kernel_bound_two} implies that under Assumption
	\ref{ass:derivative_integrable}, for any $\epsilon\in(0,d_{{\rm {vol}}})$ (and
	$\epsilon$ can be $0$ if $d_{{\rm {vol}}}=0$ or under Assumption
	\ref{ass:dimension_exact}), 
	\begin{equation}
	\mathbb{E}_{P}\left[\left(\frac{1}{h^{d+|s|}}D^{s}K_{x,h}\right)^{2}\right]\leq C_{s,P,K,\epsilon}l_{n}^{-2d-2|s|+d_{\mathrm{vol}}-\epsilon}.\label{eq:kde_uniform_bound_two-2}
	\end{equation}
	Hence from \eqref{eq:kde_uniform_bound_inf-2}, \eqref{eq:kde_uniform_bound_vc-2},
	and \eqref{eq:kde_uniform_bound_two-2}, applying Theorem \ref{thm:function_uniform}
	to \eqref{eq:kde_uniform_expansion-2} gives that $\sup_{h\geq l_{n},x\in\mathbb{X}}\bigl|D^{s}\hat{p}_{h}(x)-D^{s}p_{h}(x)\bigr|$
	is upper bounded with probability at least $1-\delta$ as 
	\begin{align*}
	& \sup_{h\geq l_{n},x\in\mathbb{X}}\left|D^{s}\hat{p}_{h}(x)-D^{s}p_{h}(x)\right|\\
	& \leq C\left(\frac{2(\nu+2)\left\Vert D^{s}K\right\Vert _{\infty}\log\left(\frac{2A(d+|s|)\left\Vert D^{s}K\right\Vert _{\infty}}{\sqrt{C_{s,P,K,\epsilon}}l_{n}^{(d_{\mathrm{vol}}-\epsilon)/2}}\right)}{nl_{n}^{d+|s|}}+\sqrt{\frac{2(\nu+2)C_{s,P,K,\epsilon}\log\left(\frac{2A(d+|s|)\left\Vert D^{s}K\right\Vert _{\infty}}{\sqrt{C_{s,P,K,\epsilon}}l_{n}^{(d_{\mathrm{vol}}-\epsilon)/2}}\right)}{nl_{n}^{2d+2|s|-d_{\mathrm{vol}}+\epsilon}}}\right.\\
	& \qquad\qquad+\left.\sqrt{\frac{C_{s,P,K,\epsilon}\log(\frac{1}{\delta})}{nl_{n}^{2d+2|s|-d_{\mathrm{vol}}+\epsilon}}}+\frac{\left\Vert D_{s}K\right\Vert _{\infty}\log(\frac{1}{\delta})}{nl_{n}^{d+|s|}}\right)\\
	& \leq C_{A,\left\Vert D^{s}K\right\Vert _{\infty},d,\nu,d_{\mathrm{vol}},C_{s,P,K,\epsilon}}\left(\frac{\left(\log\left(\frac{1}{l_{n}}\right)\right)_{+}}{nl_{n}^{d+|s|}}+\sqrt{\frac{\left(\log\left(\frac{1}{l_{n}}\right)\right)_{+}}{nl_{n}^{2d+2|s|-d_{\mathrm{vol}}+\epsilon}}}+\sqrt{\frac{\log\left(\frac{2}{\delta}\right)}{nl_{n}^{2d+2|s|-d_{\mathrm{vol}}+\epsilon}}}+\frac{\log\left(\frac{2}{\delta}\right)}{nl_{n}^{d+|s|}}\right),
	\end{align*}
	where $C_{A,\left\Vert D^{s}K\right\Vert _{\infty},d,\nu,d_{\mathrm{vol}},C_{s,P,K,\epsilon},\epsilon}$
	depends only on $A$, $\left\Vert D^{s}K\right\Vert _{\infty}$, $d$,
	$\nu$, $d_{\mathrm{vol}}$, $C_{s,P,K,\epsilon}$, $\epsilon$.

\end{proof}

For showing Corollary \ref{cor:derivative_band_ray_uniform_probconv}, we proceed similarly to the proof of Corollary \ref{cor:kde_band_ray_uniform_probconv},
where we plug in $D^{s}K$ in the place of $K$.

\textbf{Corollary \ref{cor:derivative_band_ray_uniform_probconv}.} \textit{
	Let $P$ be a distribution and $K$ be a kernel function satisfying
	Assumption~\ref{ass:derivative_leibniz},~\ref{ass:derivative_integrable}, and~\ref{ass:derivative_vc}.
	Suppose 
	\[
	\limsup_{n}\frac{\left(\log\left(1 / l_{n}\right)\right)_{+}+\log\left(2 / \delta\right)}{nl_{n}^{d_{\mathrm{vol}}-\epsilon}}<\infty,
	\]
	for fixed $\epsilon\in(0,d_{\rm{vol}})$. Then, with probability at least $1-\delta$,
	\[
	\sup_{h\geq l_{n},x\in\mathbb{X}}\left|D^{s}\hat{p}_{h}(x)-D^{s}p_{h}(x)\right|
	\leq C'\sqrt{\frac{\left(\log\left(1 / l_{n}\right)\right)_{+}+\log\left(2 / \delta\right)}{nl_{n}^{2d+2|s|-d_{\mathrm{vol}}+\epsilon}}},
	\]
	where $C'$ is a constant depending only on $A$, $\left\Vert D^{s}K\right\Vert _{\infty}$, $d$,
	$\nu$, $d_{\mathrm{vol}}$, $C_{s,P,K,\epsilon}$, $\epsilon$. Further, if $d_{\rm{vol}}=0$ or under Assumption~\ref{ass:dimension_exact}, $\epsilon$ can be $0$.
}

\begin{proof}[Proof of Corollary \ref{cor:derivative_band_ray_uniform_probconv}]
	
	From \eqref{eq:derivative_band_ray_uniform_bound} in Theorem \ref{thm:derivative_band_ray_uniform},
	$\sup_{h\geq l_{n},x\in\mathbb{X}}\left|D^{s}\hat{p}_{h}(x)-D^{s}p_{h}(x)\right|$
	is upper bounded with probability at least $1-\delta$ as 
	\begin{align*}
	& \sup_{h\geq l_{n},x\in\mathbb{X}}\left|D^{s}\hat{p}_{h}(x)-D^{s}p_{h}(x)\right|\\
	& \leq C_{A,\left\Vert D^{s}K\right\Vert _{\infty},d,\nu,d_{\mathrm{vol}},C_{s,P,K,\epsilon}}\left(\frac{\left(\log\left(\frac{1}{l_{n}}\right)\right)_{+}}{nl_{n}^{d+|s|}}+\sqrt{\frac{\left(\log\left(\frac{1}{l_{n}}\right)\right)_{+}}{nl_{n}^{2d+2|s|-d_{\mathrm{vol}}+\epsilon}}}+\sqrt{\frac{\log\left(\frac{2}{\delta}\right)}{nl_{n}^{2d+2|s|-d_{\mathrm{vol}}+\epsilon}}}+\frac{\log\left(\frac{2}{\delta}\right)}{nl_{n}^{d+|s|}}\right)\\
	& =C_{A,\left\Vert D^{s}K\right\Vert _{\infty},d,\nu,d_{\mathrm{vol}},C_{s,P,K,\epsilon}}\\
	& \quad \times\left(\sqrt{\frac{\left(\log\left(\frac{1}{l_{n}}\right)\right)_{+}}{nl_{n}^{2d+2|s|-d_{\mathrm{vol}}+\epsilon}}}\left(\sqrt{\frac{\left(\log\left(\frac{1}{l_{n}}\right)\right)_{+}}{nl_{n}^{d_{\mathrm{vol}}-\epsilon}}}+1\right)+\sqrt{\frac{\log\left(\frac{2}{\delta}\right)}{nl_{n}^{2d+2|s|-d_{\mathrm{vol}}+\epsilon}}}\left(\sqrt{\frac{\log\left(\frac{2}{\delta}\right)}{nl_{n}^{d_{\mathrm{vol}}-\epsilon}}}+1\right)\right).
	\end{align*}
	
	Then from $\lim\sup_{n}\frac{\left(\log\left(\frac{1}{l_{n}}\right)\right)_{+}+\log\left(\frac{2}{\delta}\right)}{nl_{n}^{d_{\mathrm{vol}}-\epsilon}}<\infty$,
	there exists some constant $C'$ with $\left(\log\left(\frac{1}{l_{n}}\right)\right)_{+}+\log\left(\frac{2}{\delta}\right)\leq C'nl_{n}^{d_{\mathrm{vol}}+\epsilon}$.
	And hence $\sup_{h\geq l_{n},x\in\mathbb{X}}\left|D^{s}\hat{p}_{h}(x)-D^{s}p_{h}(x)\right|$
	is upper bounded with probability $1-\delta$ as 
	\begin{align*}
	& \sup_{h\geq l_{n},x\in\mathbb{X}}\left|D^{s}\hat{p}_{h}(x)-D^{s}p_{h}(x)\right|\\
	& \leq C_{A,\left\Vert D^{s}K\right\Vert _{\infty},d,\nu,d_{\mathrm{vol}},C_{s,P,K,\epsilon}}\left(\sqrt{\frac{\left(\log\left(\frac{1}{l_{n}}\right)\right)_{+}}{nl_{n}^{2d+2|s|-d_{\mathrm{vol}}+\epsilon}}}\left(\sqrt{C'}+1\right)+\sqrt{\frac{\log\left(\frac{1}{\delta}\right)}{nl_{n}^{2d+2|s|-d_{\mathrm{vol}}+\epsilon}}}\left(\sqrt{C'}+1\right)\right)\\
	& \leq C'_{A,\left\Vert D^{s}K\right\Vert _{\infty},d,\nu,d_{\mathrm{vol}},C_{s,P,K,\epsilon}}\sqrt{\frac{\left(\log\left(\frac{1}{l_{n}}\right)\right)_{+}+\log\left(\frac{1}{\delta}\right)}{nl_{n}^{2d+2|s|-d_{\mathrm{vol}}+\epsilon}}},
	\end{align*}
	where $C'_{A,\left\Vert D^{s}K\right\Vert _{\infty},d,\nu,d_{\mathrm{vol}},C_{s,P,K,\epsilon}}$
	depending only on $A$, $\left\Vert D^{s}K\right\Vert _{\infty}$, $d$,
	$\nu$, $d_{\mathrm{vol}}$, $C_{s,P,K,\epsilon}$, $\epsilon$.
	
\end{proof}

For proving Lemma \ref{lem:derivative_vc}, we proceed similarly to the proof of Lemma \ref{lem:kde_vc}, where we plug in $D^{s}K$ in the place of $K$.

\textbf{Lemma \ref{lem:derivative_vc}.} \textit{
	Suppose there exists $R>0$ with $\mathbb{X}\subset\mathbb{B}_{\mathbb{R}^{d}}(0,R)$.
	Also, suppose that $D^{s}K$ is $M_{K}$-Lipschitz, i.e. 
	$$
	\left\Vert D^{s}K(x)-D^{s}K(y)\right\Vert _{2}\leq M_{K}\left\Vert x-y\right\Vert _{2}.
	$$
	Then for all $\eta\in\left(0,\left\Vert D^{s}K\right\Vert _{\infty}\right)$,
	the supremum of the $\eta$-covering number $\mathcal{N}(\mathcal{F}_{K,h}^{s},L_{2}(Q),\eta)$
	over all measure $Q$ is upper bounded as 
	\[
	\sup_{Q}\mathcal{N}(\mathcal{F}_{K,h}^{s},L_{2}(Q),\eta)\leq\left(\frac{2RM_{K}h^{-1}+\left\Vert D^{s}K\right\Vert _{\infty}}{\eta}\right)^{d}.
	\]
}

\begin{proof}[Proof of Lemma \ref{lem:derivative_vc}]
	
	For fixed $\eta>0$, let $x_{1},\ldots,x_{M}$ be the maximal $\eta$-covering
	of $\mathbb{B}_{\mathbb{R}^{d}}(0,R)$, with $M=\mathcal{M}\left(\mathbb{B}_{\mathbb{R}^{d}}(0,R),\left\Vert \cdot\right\Vert _{2},\eta\right)$
	being the packing number of $\mathbb{B}_{\mathbb{R}^{d}}(0,R)$. Then
	$\mathbb{B}_{\mathbb{R}^{d}}(x_{i},\eta)$ and $\mathbb{B}_{\mathbb{R}^{d}}(x_{j},\eta)$
	do not intersect for any $i,j$ and $\bigcup_{i=1}^{M}\mathbb{B}_{\mathbb{R}^{d}}(x_{i},\eta)\subset\mathbb{B}_{\mathbb{R}^{d}}(x_{i},R+\eta)$,
	and hence 
	\begin{equation}
	\sum_{i=1}^{M}\lambda_{d}\left(\mathbb{B}_{\mathbb{R}^{d}}(x_{i},\eta)\right)\leq\lambda_{d}\left(\mathbb{B}_{\mathbb{R}^{d}}(x_{i},R+\eta)\right).\label{eq:kde_vc_support_probability-1}
	\end{equation}
	Then $\lambda_{d}\left(\mathbb{B}_{\mathbb{R}^{d}}(x,r)\right)=r^{d}\lambda_{d}\left(\mathbb{B}_{\mathbb{R}^{d}}(0,1)\right)$
	gives the upper bound on $\mathcal{M}(\mathbb{B}_{\mathbb{R}^{d}}(0,R),\left\Vert \cdot\right\Vert _{2},\eta)$
	as 
	\[
	\mathcal{M}\left(\mathbb{B}_{\mathbb{R}^{d}}(0,R),\left\Vert \cdot\right\Vert _{2},\eta\right)\leq\left(1+\frac{R}{\eta}\right)^{d}.
	\]
	Then $\mathbb{X}\subset\mathbb{B}_{\mathbb{R}^{d}}(0,R)$ and the
	relationship between covering number and packing number gives the
	upper bound on the covering number $\mathcal{N}\left(\mathbb{X},\left\Vert \cdot\right\Vert _{2},\eta\right)$
	as 
	\begin{equation}
	\mathcal{N}\left(\mathbb{X},\left\Vert \cdot\right\Vert _{2},\eta\right)\leq\mathcal{N}\left(\mathbb{B}_{\mathbb{R}^{d}}(0,R),\left\Vert \cdot\right\Vert _{2},\eta\right)\leq\mathcal{M}\left(\mathbb{B}_{\mathbb{R}^{d}}(0,R),\left\Vert \cdot\right\Vert _{2},\frac{\eta}{2}\right)\leq\left(1+\frac{2R}{\eta}\right)^{d}.\label{eq:kde_vc_support_covering-1}
	\end{equation}
	Now, note that for all $x,y\in\mathbb{X}$ and for all $z\in\mathbb{R}^{d}$,
	$\left|D^{s}K_{x,h}(z)-D^{s}K_{y,h}(z)\right|$ is upper bounded as
	\begin{align*}
	\left|D^{s}K_{x,h}(z)-D^{s}K_{y,h}(z)\right|&=\left|D^{s}K\left(\frac{x-z}{h}\right)-D^{s}K\left(\frac{y-z}{h}\right)\right|\\
	& \leq\frac{M_{K}}{h}\left\Vert (x-z)-(y-z)\right\Vert _{2}=\frac{M_{K}}{h}\left\Vert x-y\right\Vert _{2}.
	\end{align*}
	Hence for any measure $Q$ on $\mathbb{R}^{d}$, $\left\Vert D^{s}K_{x,h}-D^{s}K_{y,h}\right\Vert _{L_{2}(Q)}$
	is upper bounded as 
	\[
	\left\Vert D^{s}K_{x,h}-D^{s}K_{y,h}\right\Vert _{L_{2}(Q)}=\sqrt{\int(D^{s}K_{x,h}(z)-D^{s}K_{y,h}(z))^{2}dQ(z)}\leq\frac{M_{K}}{h}\left\Vert x-y\right\Vert _{2}.
	\]
	Hence applying this to \eqref{eq:kde_vc_support_covering-1} implies
	that for all $\eta>0$, the supremum of the covering number $\mathcal{N}(\mathcal{F}_{K,h},L_{2}(Q),\eta)$
	over all measure $Q$ is upper bounded as 
	\[
	\sup_{Q}\mathcal{N}(\mathcal{F}_{K,h}^{s},L_{2}(Q),\eta)\leq\mathcal{N}\left(\mathbb{X},\left\Vert \cdot\right\Vert _{2},\frac{h\eta}{M_{K}}\right)\leq\left(1+\frac{2RM_{K}}{h\eta}\right)^{d}.
	\]
	Hence for all $\eta\in\left(0,\left\Vert D^{s} K\right\Vert _{\infty}\right)$,
	\[
	\sup_{Q}\mathcal{N}(\mathcal{F}_{K,h}^{s},L_{2}(Q),\eta)\leq\left(\frac{2RM_{K}h^{-1}+\left\Vert D^{s}K\right\Vert _{\infty}}{\eta}\right)^{d}.
	\]
	
\end{proof}

For Corollary \ref{cor:derivative_uniform_band_one_probconv}, we proceed similarly to the proof of Corollary \ref{cor:kde_uniform_band_one_probconv},
where we plug in $D^{s}K$ in the place of $K$.

\textbf{Corollary \ref{cor:derivative_uniform_band_one_probconv}.} \textit{
	Suppose there exists $R>0$ with $\mathrm{supp}(P) = \mathbb{X}\subset\mathbb{B}_{\mathbb{R}^{d}}(0,R)$. Let $K$ be a kernel function with $M_K$-Lipschitz continuous derivative satisfying Assumption~\ref{ass:derivative_integrable}.
	If
	\[
	\limsup_{n}\frac{\left(\log\left(1 / h_{n}\right)\right)_{+}+\log\left(2 / \delta\right)}{nh_{n}^{d_{\mathrm{vol}}-\epsilon}}<\infty,
	\]
	for fixed $\epsilon\in(0,d_{\rm{vol}})$. Then, with probability at least $1-\delta$, 
	\[
	\sup_{x\in\mathbb{X}}\left|D^{s}\hat{p}_{h}(x)-D^{s}p_{h}(x)\right|
	\leq C''\sqrt{\frac{(\log(\frac{1}{h_{n}}))_{+}+\log(\frac{2}{\delta})}{nh_{n}^{2d+2|s|-d_{\mathrm{vol}+\epsilon}}}},
	\]
	where $C''$ is a constant depending only on $A$, $\left\Vert D^{s}K\right\Vert _{\infty}$, $d$,
	$M_k$, $d_{\mathrm{vol}}$, $C_{s,P,K,\epsilon}$, $\epsilon$.
	Further, if $d_{\rm{vol}}=0$ or under Assumption~\ref{ass:dimension_exact}, $\epsilon$ can be $0$.
}

\begin{proof}[Proof of Corollary \ref{cor:derivative_uniform_band_one_probconv}]
	
	For $x\in\mathbb{X}$, let $D^{s}K_{x,h}:\mathbb{R}^{d}\to\mathbb{R}$
	be $D^{s}K_{x,h}(\cdot)=D^{s}K\left(\frac{x-\cdot}{h}\right)$, and
	let $\tilde{\mathcal{F}}_{K,h}^{s}:=\left\{ \frac{1}{h^{d+|s|}}D^{s}K_{x,h}:\,x\in\mathbb{X}\right\} $
	be a class of normalized kernel functions centered on $\mathbb{X}$
	and bandwidth $h$. Note that $D^{s}\hat{p}_{h}(x)-D^{s}p_{h}(x)$
	can be expanded as 
	\begin{align*}
	D^{s}\hat{p}_{h}(x)-D^{s}p_{h}(x)&=\frac{1}{nh^{d+|s|}}\sum_{i=1}^{n}D^{s}K\left(\frac{x-X_{i}}{h}\right)-\mathbb{E}_{P}\left[\frac{1}{h^{d+|s|}}D^{s}K\left(\frac{x-X_{i}}{h}\right)\right]\\
	&=\frac{1}{n}\sum_{i=1}^{n}\frac{1}{h^{d+|s|}}D^{s}K_{x,h}(X_{i})-\mathbb{E}_{P}\left[\frac{1}{h^{d+|s|}}D^{s}K_{x,h}\right].
	\end{align*}
	Hence $\sup_{x\in\mathbb{X}}\left|D^{s}\hat{p}_{h}(x)-D^{s}p_{h}(x)\right|$
	can be expanded as 
	\begin{equation}
	\sup_{x\in\mathbb{X}}\left|D^{s}\hat{p}_{h}(x)-D^{s}p_{h}(x)\right|=\sup_{f\in\tilde{\mathcal{F}}_{K,h}}\left|\frac{1}{n}\sum_{i=1}^{n}f(X_{i})-\mathbb{E}_{P}\left[f(X)\right]\right|.\label{eq:kde_uniform_expansion-1-1}
	\end{equation}
	Now, it is immediate to check that 
	\begin{equation}
	\left\Vert f\right\Vert _{\infty}\leq h^{-d-|s|}\left\Vert D^{s}K\right\Vert _{\infty}.\label{eq:kde_uniform_bound_inf-1-1}
	\end{equation}
	Also, since $\tilde{\mathcal{F}}_{K,h}^{s}=h^{-d-|s|}\mathcal{F}_{K,h}^{s}$,
	VC dimension is uniformly bounded as Lemma \ref{lem:derivative_vc} gives
	that for every probability measure $Q$ on $\mathbb{R}^{d}$ and for
	every $\eta\in(0,h^{-d-|s|}\left\Vert D^{s}K\right\Vert _{\infty})$, the
	covering number $\mathcal{N}(\tilde{\mathcal{F}}_{K,h}^{s},L_{2}(Q),\eta)$
	is upper bounded as 
	\begin{align}
	\sup_{Q}\mathcal{N}(\tilde{\mathcal{F}}_{K,h}^{s},L_{2}(Q),\eta) & =\sup_{Q}\mathcal{N}(\mathcal{F}_{K,h},L_{2}(Q),h^{d+|s|}\eta)\nonumber \\
	& \leq\left(\frac{2RM_{K}h^{-1}+\left\Vert D^{s}K\right\Vert _{\infty}}{h^{d+|s|}\eta}\right)^{d}\nonumber \\
	& \leq\left(\frac{2RM_{K}\left\Vert D^{s}K\right\Vert _{\infty}}{h^{d+|s|+1}\eta}\right)^{d}.\label{eq:kde_uniform_bound_vc-1-1}
	\end{align}
	Also, Lemma \ref{lem:derivative_kernel_bound_two} implies that under Assumption
	\ref{ass:integrable}, for any $\epsilon\in(0,d_{{\rm {vol}}})$ (and
	$\epsilon$ can be $0$ if $d_{{\rm {vol}}}=0$ or under Assumption
	\ref{ass:dimension_exact}), 
	\begin{equation}
	\mathbb{E}_{P}\left[\left(\frac{1}{h^{d+|s|}}D^{s}K_{x,h}\right)^{2}\right]\leq C_{s,P,K,\epsilon}h^{-2d-2|s|+d_{\mathrm{vol}}-\epsilon}.\label{eq:kde_uniform_bound_two-1-1}
	\end{equation}
	Hence from \eqref{eq:kde_uniform_bound_inf-1-1}, \eqref{eq:kde_uniform_bound_vc-1-1},
	and \eqref{eq:kde_uniform_bound_two-1-1}, applying Theorem \ref{thm:function_uniform}
	to \eqref{eq:kde_uniform_expansion-1-1} gives that $\sup_{x\in\mathbb{X}}\left|D^{s}\hat{p}_{h}(x)-D^{s}p_{h}(x)\right|$
	is upper bounded with probability at least $1-\delta$ as 
	\begin{align*}
	& \sup_{x\in\mathbb{X}}\left|\hat{p}_{h}(x)-p_{h}(x)\right|\\
	& \leq C\left(\frac{2d\left\Vert D^{s}K\right\Vert _{\infty}\log\left(\frac{2RM_{K}\left\Vert D^{s}K\right\Vert _{\infty}}{\sqrt{C_{s,P,K,\epsilon}}h^{1+(d_{\mathrm{vol}}-\epsilon)/2}}\right)}{nh^{d+|s|}}+\sqrt{\frac{2dC_{s,P,K,\epsilon}\log\left(\frac{2RM_{K}\left\Vert D^{s}K\right\Vert _{\infty}}{\sqrt{C_{s,P,K,\epsilon}}h^{1+(d_{\mathrm{vol}}-\epsilon)/2}}\right)}{nh^{2d+2|s|-d_{\mathrm{vol}}+\epsilon}}}\right.\\
	& \qquad\qquad+\left.\sqrt{\frac{C_{s,P,K,\epsilon}\log(\frac{1}{\delta})}{nh^{2d+2|s|-d_{\mathrm{vol}}+\epsilon}}}+\frac{\left\Vert D^{s}K\right\Vert _{\infty}\log(\frac{1}{\delta})}{nh^{d}}\right)\\
	& \leq C_{R,M_{K},\left\Vert D^{s}K\right\Vert _{\infty},d,\nu,d_{\mathrm{vol}},C_{s,P,K,\epsilon},\epsilon}\\
	& \quad \times\left(\frac{\left(\log\left(\frac{1}{h}\right)\right)_{+}}{nh^{d}}+\sqrt{\frac{\left(\log\left(\frac{1}{h}\right)\right)_{+}}{nh^{2d+2|s|-d_{\mathrm{vol}}+\epsilon}}}+\sqrt{\frac{\log\left(\frac{2}{\delta}\right)}{nh^{2d+2|s|-d_{\mathrm{vol}}+\epsilon}}}+\frac{\log\left(\frac{2}{\delta}\right)}{nh^{d}}\right),
	\end{align*}
	where $C_{R,M_{K},\left\Vert D^{s}K\right\Vert _{\infty},d,\nu,d_{\mathrm{vol}},C_{s,P,K,\epsilon},\epsilon}$
	depends only on $R$, $M_{K}$, $\left\Vert D^{s}K\right\Vert _{\infty}$,
	$d$, $\nu$, $d_{\mathrm{vol}}$, $C_{s,P,K,\epsilon}$, $\epsilon$.
	
\end{proof}

\end{document}